\documentclass[12pt]{amsart}
\usepackage{amsmath, amsthm, amssymb}
\usepackage{enumerate}
\usepackage{amsfonts, amscd, bbm, eucal, latexsym, extarrows, mathrsfs}
\usepackage[utf8]{inputenc}
\usepackage[T1]{fontenc}
\usepackage[bookmarks=false]{hyperref}
\usepackage[usenames,dvipsnames]{color}
\usepackage{mathtools}
\usepackage{version}

\usepackage{fullpage}
\usepackage[all,cmtip]{xy}
\usepackage{color}
\usepackage{microtype}
\usepackage{cancel}
\usepackage{fourier}

\usepackage{mathtools}

\theoremstyle{plain}
\newtheorem{theorem}{Theorem}[section]
\newtheorem*{maintheorem}{Main Theorem}
\newtheorem{proposition}[theorem]{Proposition}		
\newtheorem{corollary}[theorem]{Corollary}
\newtheorem{lemma}[theorem]{Lemma}
\newtheorem{conjecture}{Conjecture}
\newtheorem*{BCOV}{BCOV conjecture at genus one}
\newtheorem*{BCOV-new}{Refined BCOV conjecture at genus one}
\newtheorem*{conjecture-intro}{Conjecture}

\newtheorem{definition}[theorem]{Definition}

\theoremstyle{remark}
\newtheorem{remark}[theorem]{Remark}

\def\and{\textrm{ and }}
\def\bcov{{\scriptscriptstyle{\rm{BCOV}}}}
\def\l2{{\scriptscriptstyle{\mathrm{L}^2}}}
\def\quillen{{\scriptscriptstyle{\mathrm{Q}}}}

\def\quillen{{\scriptscriptstyle{\textrm{Q}}}}

\newcommand{\taubcov}[1]{\tau_{\bcov}{( #1})}

\newcounter{tmp}

\newcommand{\Dbold}{{\bf D}}

\newcommand{\ov}{\overline}

\newcommand{\qbold}{{\bf q}}

\newcommand{\ABbb}{\mathbb A}

\newcommand{\CBbb}{\mathbb C}
\newcommand{\DBbb}{\mathbb D}
\newcommand{\EBbb}{\mathbb E}

\newcommand{\GBbb}{\mathbb G}

\newcommand{\PBbb}{\mathbb P}
\newcommand{\QBbb}{\mathbb Q}
\newcommand{\RBbb}{\mathbb R}

\newcommand{\TBbb}{\mathbb T}

\newcommand{\VBbb}{\mathbb V}
\newcommand{\WBbb}{\mathbb W}

\newcommand{\ZBbb}{\mathbb Z}

\newcommand{\Ccal}{\mathcal C}

\newcommand{\Ecal}{\mathcal E}
\newcommand{\Fcal}{\mathcal F}

\newcommand{\Hcal}{\mathcal H}

\newcommand{\Kcal}{\mathcal K}

\newcommand{\Mcal}{\mathcal M}

\newcommand{\Ocal}{\mathcal O}

\newcommand{\Vcal}{\mathcal V}
\newcommand{\Wcal}{\mathcal W}
\newcommand{\Xcal}{\mathcal X}
\newcommand{\Ycal}{\mathcal Y}
\newcommand{\Zcal}{\mathcal Z}

\newcommand{\GL}{\mathsf{GL}}

\DeclareMathOperator{\ACH}{\widehat{\operatorname{CH}}}
\DeclareMathOperator{\cHat}{\widehat{c}}
\DeclareMathOperator{\chHat}{\widehat{ch}}
\DeclareMathOperator{\tdHat}{\widehat{Td}}
\DeclareMathOperator{\GW}{GW}
\providecommand{\td}{\operatorname{Td}}
\providecommand{\ch}{\operatorname{ch}}

\providecommand{\vol}{\operatorname{vol}}

\providecommand{\hess}{\mathop{d\!d^c}\nolimits}
\providecommand{\res}{\operatorname{res}}

\providecommand{\Gr}{\operatorname{Gr}}

\providecommand{\Spec}{\operatorname{Spec}}

\DeclareMathOperator{\Hom}{Hom}

\DeclareMathOperator{\Aut}{Aut}

\DeclareMathOperator{\id}{id}

\DeclareMathOperator{\Sym}{Sym}

\DeclareMathOperator{\tr}{tr}

\DeclareMathOperator{\Div}{div}
\DeclareMathOperator{\ord}{ord}
\DeclareMathOperator{\an}{an}

\newcommand{\Pic}{{\rm Pic}}

\DeclareMathOperator{\Real}{Re}
\DeclareMathOperator{\Imag}{Im}

\DeclareMathOperator{\Res}{Res}
\DeclareMathOperator{\nt}{nt}

\DeclareMathOperator{\per}{per}
\DeclareMathOperator{\GRR}{\mathbf{GRR}}

\newcommand{\KS}{{\rm KS}}

\newcommand{\prim}{{\rm prim}}

\numberwithin{equation}{section}

\begin{document}
 \title{On genus one mirror symmetry in higher dimensions and~the~BCOV~conjectures}
\author{Dennis Eriksson}
\author{Gerard Freixas i Montplet}
\author{Christophe Mourougane}
\address{Dennis Eriksson \\ Department of Mathematics \\ Chalmers University of Technology and  University of Gothenburg}
\email{dener@chalmers.se}

\address{Gerard Freixas i Montplet \\ C.N.R.S. -- Institut de Math\'ematiques de Jussieu - Paris Rive Gauche}
\email{gerard.freixas@imj-prg.fr}

\address{Christophe Mourougane\\Institut de Recherche Math\'ematique de Rennes (IRMAR)}
\email{christophe.mourougane@univ-rennes1.fr}

\begin{abstract}
 The mathematical physicists Bershadsky--Cecotti--Ooguri--Vafa (BCOV) proposed, in a seminal article from '94,  a conjecture extending genus zero mirror symmetry to higher genera. With a view towards a refined formulation of the Grothendieck--Riemann--Roch theorem, we offer a mathematical  description of the BCOV conjecture at genus one. As an application of the arithmetic Riemann--Roch theorem of Gillet--Soul\'e and of our previous results on the BCOV invariant, we establish this conjecture for Calabi--Yau hypersurfaces in projective spaces.  Our contribution takes place on the $B$-side, and together with the work of Zinger on the $A$-side, it provides the first complete examples of the mirror symmetry program in higher dimensions. The case of quintic threefolds was studied by Fang--Lu--Yoshikawa. 
 Our approach also lends itself to arithmetic considerations of the BCOV invariant, and we study a Chowla--Selberg type theorem expressing it in terms of special $\Gamma$ values for certain Calabi--Yau manifolds with complex multiplication.  
\end{abstract}

\subjclass[2010]{Primary: 14J32, 14J33, 58J52. Secondary: 32G20.}
\keywords{Genus one mirror symmetry, Calabi--Yau manifolds, BCOV theory, variations of Hodge structures, arithmetic Riemann--Roch theorem}
\maketitle

\begin{flushright}\begin{em}
 to Jean-Pierre Demailly, in memoriam.
    \end{em}
    \end{flushright}

\setcounter{tocdepth}{1}
\tableofcontents

\vspace*{14pt}
\section{Introduction}
\begingroup
\setcounter{tmp}{\value{theorem}}
\setcounter{theorem}{0}
\renewcommand\thetheorem{\Alph{theorem}}

The purpose of this article is to establish higher dimensional cases of genus one mirror symmetry, as envisioned by mathematical physicists Bershadsky--Cecotti--Ooguri--Vafa (henceforth abbreviated BCOV) in their influential paper \cite{BCOV}. Precisely, we relate the generating series of genus one Gromov--Witten invariants on Calabi--Yau hypersurfaces to an invariant of a mirror family, built out of holomorphic analytic torsions. The invariant, whose existence was conjectured in \emph{loc. cit.}, was mathematically defined and studied in our previous paper \cite{cdg2}. We refer to it as the BCOV invariant $\tau_{\bcov}$. 
In dimension 3, the construction of the BCOV invariant and its relation to mirror symmetry were established by Fang--Lu--Yoshikawa \cite{FLY}, relying on previous results by \cite{Zingerstvsred, Zingerreduced}. 

Our approach parallels the Kodaira--Spencer formulation of the Yukawa coupling in genus zero, and can be recast as a refined version of the Grothendieck--Riemann--Roch theorem \emph{\`a la} Deligne \cite{Deligne-determinant}. We hope this point of view will also be inspiring to study higher genus Gromov--Witten invariants and the $B$-side of mirror symmetry in dimension 3. In this setting, the $A$-side has received a lot of attention recently.

\subsection{The classical BCOV conjecture at genus one}\label{subsec:BCOV-conj-intro}
Let $X$ be a Calabi--Yau manifold of dimension~$n$. In this article, this will mean a complex projective connected manifold with trivial canonical sheaf.
We now briefly recall the BCOV program at genus one. 

On the one hand, on what is referred to as the $A$-side, we consider enumerative invariants associated to $X$. For this, recall first that for every curve class $\beta$ in $H_2(X, \ZBbb)$, there is a proper Deligne--Mumford stack of stable maps from genus $g$ curves to $X$, whose fundamental class is $\beta$:
\begin{displaymath}
    \overline{\Mcal}_{g}(X, \beta) = \left\lbrace f: C \to X\mid g(C)=g,\ f\ \text{ stable and }\ f_*[C] = \beta \right\rbrace.
\end{displaymath}
The virtual dimension of this stack can be computed to be (cf. \cite{GrWiAlGe}, in particular the introduction) $$\int_\beta c_1(X) + (\operatorname{dim}(X)-3)(1-g) = (\operatorname{dim}(X)-3)(1-g).$$ Whenever $\operatorname{dim}(X)=3$ or $g = 1$ this is of virtual dimension 0 and one can consider the Gromov--Witten invariants 
\begin{displaymath} 
    \GW_g(X,\beta) = \deg\ [\overline{\Mcal}_{g}(X, \beta)]^{\mathrm{vir}}  \in \QBbb.
\end{displaymath} 
Since the main focus of our paper is higher dimension, we henceforth impose $g=1$. One then defines the formal power series
\begin{equation}\label{eq:F1A-intro}
    F_1^A(\tau)= \frac{-1}{24} \int_X \mathrm{c}_{n-1}(X) \cap 2\pi i \tau+\sum_{\beta > 0} \GW_1(X, \beta) e^{2\pi i\langle \tau, \beta\rangle},
\end{equation}
where $\tau$ belongs to the complexified K\"ahler cone\footnote{If $\Kcal_{X}$ denotes the K\"ahler cone of $X$, we define $\Hcal_X$ as $ H^{1,1}_{\RBbb}(X)/H^{1,1}_{\ZBbb}(X) + i \Kcal_{X}$.} $\Hcal_{X}$, and $\beta$ runs over the non-zero effective curve classes.

On the other hand, on what is referred to as the $B$-side, BCOV introduced a spectral quantity $\Fcal_1^B$ built out of holomorphic Ray--Singer analytic torsions of a mirror Calabi--Yau manifold $X^{\vee}$. It depends on an auxiliary choice of a K\"ahler structure  $\omega$ on $X^{\vee}$, and can be recast as
\begin{displaymath}
    \Fcal_{1}^{B}(X^{\vee},\omega)=\prod_{0\leq p,q\leq n}(\det\Delta^{p,q}_{\ov{\partial}})^{(-1)^{p+q}pq},
\end{displaymath}
where $\det\Delta^{p,q}_{\ov{\partial}}$ is the $\zeta$-regularized determinant of the Dolbeault Laplacian acting on $A^{p,q}(X^{\vee})$.  
In our previous work \cite{cdg2} we normalized this quantity to make it independent of the choice of $\omega$:
\begin{displaymath}
    \taubcov{X^{\vee}}=C(X^{\vee},\omega)\cdot\Fcal_{1}^{B}(X^{\vee},\omega),
\end{displaymath}
for some explicit constant $C(X^\vee, \omega)$. Thus $\taubcov{X^{\vee}}$ only depends on the complex structure of the Calabi--Yau manifold, in accordance with the philosophy that the $B$-model only depends on variations of the complex structure on $X^{\vee}$. 
 
Mirror symmetry predicts that given $X$, there is a mirror family of Calabi--Yau manifolds over a punctured multi-disc around the origin  $\varphi\colon\Xcal^\vee \to \Dbold^\times = (\DBbb^\times)^d$, with maximally unipotent monodromies and $d = h^{1,1}(X)=h^{1}(T_{X^\vee})$.\footnote{Such families are also called \emph{large complex structure limits} of Calabi--Yau manifolds.} Here we denoted by $X^\vee$ any member of the mirror family. The $A$-side and the $B$-side should be related by a distinguished biholomorphism onto its image $\Dbold^{\times}\to\Hcal_{X}$, which is referred to as the mirror map and is denoted $q\mapsto\tau(q)$. The mirror map sends the origin of the multi-disc to infinity. Fixing a basis of ample classes on $X$, we can think of it as a change of coordinates on $\Dbold^{\times}$. In the special case of $d=1$, one such a map is constructed as a quotient of carefully selected periods in \cite{morrison-mirror}. 

\begin{BCOV} 
Let $X$ be a Calabi--Yau manifold and $\varphi\colon\Xcal^\vee \to \Dbold^\times$ a mirror family as above. 
\begin{enumerate}
    \item There is a procedure, called \emph{passing to the holomorphic limit}, to extract from $ \taubcov{\Xcal^\vee_{q}}$ as $q\to 0$ a holomorphic function $F_1^B(q)$.
    \item The functions $F_{1}^{A}$ and $F_{1}^{B}$ are related via the mirror map by
\begin{displaymath}
  F_1^B(q)=F_1^A(\tau(q)).
\end{displaymath}
\end{enumerate}
\end{BCOV}

Passing to the holomorphic limit is often interpreted as considering a Taylor expansion of $\taubcov{\Xcal^\vee_q}$ in $\tau(q)$ and $\overline{\tau(q)}$, and keeping the holomorphic part. In this article, we will instead use a procedure based on degenerations of Hodge structures. 

\subsection{Grothendieck--Riemann--Roch formulation of the BCOV conjecture at genus one}

The purpose of this subsection is to formulate a version of the BCOV conjecture producing the holomorphic function $F_1^B$ without any reference to spectral theory, holomorphic anomaly equations or holomorphic limits. Our formulation parallels the Hodge theoretic approach to the Yukawa coupling in 3-dimensional genus zero mirror symmetry: the key ingredients going into its construction are the Kodaira--Spencer mappings between Hodge bundles, and canonical trivializations of those.

To state a simplified form of our conjecture, we need to introduce the BCOV line bundle $\lambda_{\bcov}(\Xcal^{\vee}/\Dbold^{\times})$ of the mirror family $\varphi\colon\Xcal^{\vee}\to\Dbold^{\times}$. The BCOV line of a Calabi--Yau manifold $X^\vee$ is defined to be
\begin{equation*}
   \lambda_{\bcov}(X^{\vee}) = \bigotimes_{0\leq p,q\leq n} \det H^q(X^{\vee}, \Omega_{X^{\vee}}^p)^{(-1)^{p+q} p}.
\end{equation*}
For a family of Calabi--Yau manifolds it glues together to a holomorphic line bundle on the base. Also, we denote by $\chi$ the Euler characteristic of any fiber of $\varphi$ and by $K_{\Xcal^{\vee}/\Dbold^{\times}}$ the relative canonical bundle.

\begin{BCOV-new}\label{conj:BCOV-new}
Let $X$ be a Calabi-Yau manifold and $\varphi\colon\Xcal^{\vee} \to \Dbold^\times$ a mirror family as in \textsection \ref{subsec:BCOV-conj-intro}. 
\begin{enumerate}
   \item There exists a natural isomorphism of line bundles, 
\begin{equation}\label{eq:GRR-iso-intro}
    \GRR\colon\lambda_{\bcov}(\Xcal^{\vee}/\Dbold^{\times})^{\otimes 12 \kappa}\overset{\sim}{\longrightarrow} \varphi_{\ast}(K_{\Xcal^{\vee}/\Dbold^{\times}})^{\otimes\chi \kappa}, 
\end{equation}
together with natural trivializing sections of both sides. Here $\kappa$ is a non-zero integer which only depends on the relative dimension of $\varphi$. 
    \item In the natural trivializations, the isomorphism $\GRR$ can be expressed as a holomorphic function, which when written as $\exp\left((-1)^n F_1^B(q)\right)^{24 \kappa}$ satisfies
    \begin{displaymath}
        F_1^{B}(q)=F_1^A(\tau(q)).
    \end{displaymath}
\end{enumerate}
\end{BCOV-new}

The existence of some isomorphism in \eqref{eq:GRR-iso-intro} is provided by the Grothendieck--Riemann--Roch theorem in Chow theory, the key point of the conjecture being the naturality requirement. In fact, an influential program by Deligne \cite{Deligne-determinant} suggests that the codimension one part of the usual Grothendieck--Riemann--Roch equality can be lifted to a base change invariant isometry of line bundles, when equipped with natural metrics. An intermediate version of this exists via the arithmetic Riemann--Roch theorem of Gillet--Soul\'e \cite{GS:ARR}, which provides an equality of isometry classes of hermitian line bundles. Properly interpreted, this establishes a link between the BCOV invariant and a metric evaluation of \eqref{eq:GRR-iso-intro}.

A more detailed treatment of the formulation of the conjecture is given in Section \ref{sec:conjectures}, and examples related to the existing literature are also discussed. In fact, we focus on mirror families with a strong degeneration property formalized by Deligne in \cite{Delignemirror}, and expressed as a Hodge--Tate condition on the limiting Hodge structures of all the cohomology groups. In this case, from general principles in the theory of degenerations of Hodge structures, we can indeed construct natural trivializations of the line bundles in \eqref{eq:GRR-iso-intro}.

\subsection{Main results}
In this subsection we discuss the framework and statements of our results. For Calabi--Yau hypersurfaces in projective space, our main theorem settles the BCOV conjecture and its refinement.

Let $X$ be a Calabi--Yau hypersurface in $\PBbb^n_{\CBbb}$, with $n \geq 4$. Its complexified K\"ahler cone is one-dimensional, induced by restriction from that of the ambient projective space. The mirror family $f\colon \Zcal \to U$ can be realized using a crepant resolution of the quotient of the Dwork pencil
\begin{equation}\label{intro:Dworkpencil}
    x_0^{n+1}+\ldots + x_{n}^{n+1} - (n+1) \psi x_0  \ldots x_n = 0, \quad \psi \in U=\CBbb\setminus\mu_{n+1},
\end{equation} 
by the subgroup of $\GL_{n+1}(\CBbb)$ given by $G = \left \{  g\cdot (x_0, \ldots , x_n) = (\xi_0 x_0, \ldots , \xi_n x_n), \xi_i^{n+1}=1, \prod \xi_i =1 \right\}$. 
Moreover, $f\colon \Zcal \to U$ can be naturally extended across $\mu_{n+1}$ to a degeneration with ordinary double point singularities, sometimes referred to as a conifold degeneration.

The monodromy around $\psi= \infty$ is maximally unipotent and the properties of the limiting Hodge structure can be used to define a natural flag of homology cycles. Using this we can produce natural holomorphic trivializations $\widetilde{\eta}_k$, in a neighborhood of $\psi = \infty$, of the determinants of the primitive Hodge bundles $\det(R^k f_*\Omega_{\Zcal/U}^{n-1-k})_{\prim}$.\footnote{The primitive Hodge bundle $(R^{k}f_*\Omega_{\Zcal/U}^{n-1-k})_{\prim}$ is actually of rank one if $2k\neq n-1$.} These holomorphic trivializations have unipotent lower triangular period matrices. These sections have natural $L^2$ norms given by Hodge theory. The product $\otimes_{k=0}^{n-1}\widetilde{\eta}_{k}^{(n-1-k)(-1)^{n-1}}$ is the essential building block of a natural frame $\widetilde{\eta}_{\bcov}$ of $\lambda_{\bcov}(\Zcal/ U)$.

Finally, let $F_{1}^{A}(\tau(\psi))$ be the generating series defined as in \eqref{eq:F1A-intro}, for a general Calabi--Yau hypersurface $X\subset\PBbb^{n}_{\CBbb}$. Here $\psi \mapsto \tau(\psi)$ is the mirror map. 
Then our  main result (Theorem \ref{thm:BCOV-mirror} and Theorem \ref{thm:refined-BCOV-Dwork}) can be stated as follows: \footnote{To facilitate the comparison with the BCOV conjecture, notice that $X$ has now dimension $n-1$ instead of $n$.} 
\begin{maintheorem}\label{thm:main}
Let $n \geq 4$. Consider a Calabi--Yau hypersurface $X\subset\PBbb^{n}_{\CBbb}$ and the mirror family $f\colon\Zcal\to U$ above. 
\begin{enumerate}
    \item In a neighborhood of infinity, the BCOV invariant of $Z_{\psi}$ factors~as

\begin{displaymath}
    \taubcov{Z_{\psi}}=C\left|\exp\left((-1)^{n-1}F_{1}^{B}(\psi)\right)\right|^{4} \left(\frac{\|\widetilde{\eta}_{0}\|^{\chi(Z_{\psi})/12}_{\l2}}{\|\widetilde{\eta}_{\bcov}\|_{\l2}}\right)^2,
\end{displaymath}
where $F_{1}^{B}(\psi)$ is a multivalued holomorphic function with $F_{1}^{B}(\psi)=F_{1}^{A}(\tau(\psi))$ as formal series in  $\psi$, and $C$ is a positive constant;
    \item Up to a constant, the refined BCOV conjecture at genus one is true for $X$ and its mirror family, with the choices of trivializing sections $\widetilde{\eta}_{\bcov}$ and $\widetilde{\eta}_{0}$.
\end{enumerate}
\end{maintheorem}

Actually, the theorem also holds in the case of cubic curves (as follows from \textsection \ref{subsec:ApplicationKronecker}) and  quartic surfaces. We also show, more generally, in Proposition \ref{prop:refined-BCOV-K3} that the refined BCOV conjecture holds, up to a constant, for $K3$ surfaces.

The first part of the theorem extends to arbitrary dimensions previous work of Fang--Lu--Yoshikawa \cite[Thm. 1.3]{FLY} in dimension 3. 
In their approach, all the Hodge bundles have geometric meaning in terms of Weil--Petersson geometry and Kuranishi families. The lack thereof is an additional complication in our setting. 

To our knowledge, our theorem is the first complete example of higher dimensional mirror symmetry, of BCOV type at genus one, established in the mathematics literature. It confirms various instances that had informally been utilized for computational purposes, \emph{e.g.} \cite[Sec. 6]{Klemm-Pandharipande} in dimension 4.

We remark that there is an alternative approach to the BCOV theory, in arbitrary genera, provided by Costello and Li, described in their preprint \cite{CostelloLi}. It would be interesting to compare the results in this article with their program.

\subsection{Overview of proof of the main theorem}

\subsubsection*{Arithmetic Riemann--Roch} In the algebro-geometric setting, the arithmetic Riemann--Roch theorem from Arakelov theory allows us to compute the BCOV invariant of a family of Calabi--Yau varieties in terms of $L^2$ norms of auxiliary sections of Hodge bundles. This bypasses some arguments in former approaches, such as \cite{FLY}, based on the holomorphic anomaly equation (cf. \cite[Proposition 5.9]{cdg2}). It determines the BCOV invariant up to a meromorphic function, in fact a rational function.\footnote{This rational function compares to the so-called \emph{holomorphic ambiguity} in the physics literature.} The divisor of this rational function is encapsulated in the asymptotics of the $L^2$ norms and the BCOV invariant. In the special case when the base is a Zariski open set of $\PBbb^1_{\CBbb}$, as for the Dwork pencil \eqref{intro:Dworkpencil} and the mirror family, this divisor is determined by all but one point. Hence so is the function itself, up to constant. The arithmetic Riemann--Roch theorem simultaneously allows us to establish the existence of an isomorphism $\GRR$ as in \eqref{eq:GRR-iso-intro}.

\subsubsection*{Hodge bundles of the mirror family} The construction of the auxiliary sections is first of all based on a comparison of the Hodge bundles of the mirror family with the $G$-invariant part of the Hodge bundles on the Dwork pencil \eqref{intro:Dworkpencil}, explained in Section \ref{sec:dworkandmirrorfamilies}. Using the residue method of Griffiths we construct algebraic sections of the latter. These are then transported into sections $\eta_{k}$ of the Hodge bundles of the crepant resolution, \emph{i.e.} the mirror family. This leads us to a systematic geometric study of these sections in connection with Deligne extensions and limiting Hodge structures at various key points, notably at $\mu_{n+1}$ where ordinary double point singularities arise. We rely heavily on knowledge of the Yukawa coupling and our previous work in \cite[Sec. 2]{cdg2} on logarithmic Hodge bundles and semi-stable reduction. The arguments are elaborated in  Section \ref{sec:degHodgebund}.

\subsubsection*{Asymptotics of $L^2$ norms and the BCOV invariant} The above arithmetic Riemann-Roch reduction leads us to study the norm of the auxiliary sections outside of the maximally unipotent monodromy point, enabling us to focus on ordinary double points. Applying our previous result \cite[Thm. 4.4]{cdg2} to the auxiliary sections, we find that the behaviour of their $L^{2}$ norms is expressed in terms of monodromy eigenvalues, and the possible zeros or poles as determined by the geometric considerations of the preceding paragraph. The monodromy is characterized by the Picard--Lefschetz theorem. As for the asymptotics for the BCOV invariant, they were already accomplished in \cite[Thm. B]{cdg2}. This endeavor results in Theorem \ref{thm:pre-formula}, which is a description of the rational function occurring in the arithmetic Riemann--Roch theorem.

\subsubsection*{Connection to enumerative geometry} The BCOV conjecture suggests that we need to study the BCOV invariant close to $\psi=\infty$. However, the formula in Theorem \ref{thm:pre-formula} is not adapted to the mirror symmetry setting, for example the sections $\eta_k$ do not make any reference to $H^{n-1}_{\lim}$. 
We proceed to normalize  the $\eta_k$ by dividing by holomorphic periods, for a fixed basis of the weight filtration on the homology $(H_{n-1})_{\lim}$, to obtain the sections $\widetilde{\eta}_{k}$ of the main theorem. Rephrasing Theorem \ref{thm:pre-formula} with these sections, we thus arrive at an expression for the  $F_{1}^B$ in the theorem. Combined with results of Zinger \cite{Zingerstvsred, Zingerreduced}, this yields the relation to the generating series of Gromov--Witten invariants in the mirror coordinate. Lastly, the refined BCOV conjecture is deduced in this case through a reinterpretation of the BCOV invariant and the arithmetic Riemann--Roch theorem.

\subsection{Applications to Kronecker limit formulas}\label{subsec:ApplicationKronecker}
\subsubsection*{Classical first Kronecker limit formula} 
The simplest Calabi--Yau varieties are elliptic curves, which can conveniently be presented as $\CBbb/(\ZBbb+\tau \ZBbb)$, for $\tau$ in the Poincar\'e upper half-plane. The generating series \eqref{eq:F1A-intro} of Gromov--Witten invariants is then given by $-\frac{1}{24}\log \Delta(\tau)$, where $\Delta(\tau) = q \prod (1-q^n)^{24}$ and $q=e^{2\pi i\tau}$. The corresponding function $\Fcal_{1}^{B}$ is computed as $\exp(\zeta'_\tau (0))$, where 
\begin{displaymath} 
    \zeta_\tau(s) =(2\pi)^{-2s}\sum_{(m,n)\neq (0,0)} \frac{(\Imag\tau)^{s}}{|m+n\tau|^{2s}}.
\end{displaymath} 
The BCOV conjecture at genus one is deduced from the equality
\begin{equation} \label{eq:firstKronecker}
   \exp(-\zeta'_\tau (0)) = \frac{1}{(2\pi)^2} \Imag(\tau)|\Delta(\tau)|^{1/6}.
\end{equation}
This is a formulation of the first Kronecker limit formula, see \emph{e.g.} \cite[Intro.]{yoshikawa-discriminant}. In the mirror symmetry interpretation, the correspondence $\tau \mapsto q$ is the (inverse) mirror map. Equation \eqref{eq:firstKronecker} can be recovered from a standard application of the arithmetic Riemann--Roch theorem. In this vein, we will interpret all results of this shape as generalizations of the Kronecker limit formula. This includes the Theorem \ref{thm:pre-formula} cited above, as well as a Theorem \ref{thm:Kroneckerhypersurf} for Calabi--Yau hypersurfaces in Fano manifolds.

\subsubsection*{Chowla--Selberg formula} 
While being applicable to algebraic varieties over $\CBbb$, the Riemann--Roch theorem in Arakelov geometry has the further advantage of providing arithmetic information when the varieties are defined over $\QBbb$. 
The arithmetic Riemann--Roch theorem is  suited to evaluating the BCOV invariant of certain arithmetically defined Calabi--Yau varieties with additional automorphisms. As an example, for the special fibre $Z_{0}$ of our mirror family \eqref{intro:Dworkpencil}, Theorem \ref{thm:ChowlaSelberg} computes the BCOV invariant as a product of special values of the $\Gamma$ function. This is reminiscent of the Chowla--Selberg theorem \cite{Chowla-Selberg}, which derives from \eqref{eq:firstKronecker} an expression of the periods of a CM elliptic curve as a product of special $\Gamma$ values. Assuming deep conjectures of Gross--Deligne \cite{Gross}, we would be able to write any BCOV invariant of a CM Calabi--Yau manifold in such terms.

\endgroup

\subsection*{Acknowledgements}
The authors extend their heartfelt gratitude to Ken-Ichi Yoshikawa, for generously sharing his ideas and insights into BCOV invariants. Special thanks are also extended to Nicholas Shepherd--Barron, who explained Proposition \ref{prop:NSB} to us and allowed us to include its proof in the article. We thank the referee for the diligent reading and criticism of the article. It in particular helped us to write and expand Section \ref{subsec:distinguished-sections}, and correct many inaccuracies that appeared in the first version. The first author thanks Michael Bj\"orklund and Hjalmar Rosengren for discussions relating to Picard--Fuchs equations and their solutions.

\setcounter{theorem}{\thetmp}

\section{The BCOV invariant and the arithmetic Riemann--Roch theorem}\label{sec:BCOV-ARR}
In this section we describe a general method to express the BCOV invariant of a family of Calabi--Yau varieties in terms of $L^{2}$ norms of rational sections of determinants of Hodge bundles. The approach is based on the arithmetic Riemann--Roch theorem. As an example of application, we consider the case of the universal family of Calabi--Yau hypersurfaces in the projective space.

\subsection{K\"ahler manifolds and $L^2$ norms}
Let $X$ be a compact complex manifold. In this article, a hermitian metric on $X$ means a smooth hermitian metric on the holomorphic vector bundle $T_{X}$. Let $h$ be a hermitian metric on $X$. The Arakelov theoretic K\"ahler form attached to $h$ is given in local holomorphic coordinates by
\begin{equation}\label{eq:normal-omega}
    \omega=\frac{i}{2\pi}\sum_{j,k}h\left(\frac{\partial}{\partial z_{j}},\frac{\partial}{\partial z_{k}}\right)dz_{j}\wedge d\ov{z}_{k}.
\end{equation}

We assume that the complex hermitian manifold $(X,h)$ is  K\"ahler, that is the differential form $\omega$ is closed.
The hermitian metric $h$ induces hermitian metrics on the $\Ccal^{\infty}$ vector bundles of differential forms of type $(p,q)$, that we still denote $h$. Then,  the spaces $A^{p,q}(X)$ of global sections, inherit a $L^{2}$ hermitian inner product
\begin{equation}\label{eq:normal-L2}
    h_{\l2}(\alpha,\beta)=\int_{X}h(\alpha,\beta)\frac{\omega^{n}}{n!}.
\end{equation}

The coherent cohomology groups $H^{q}(X,\Omega^{p}_{X})$ can be computed as Dolbeault cohomology, that in turn can be computed in $A^{p,q}(X)$ by taking $\ov{\partial}$-harmonic representatives. Via this identification, $H^{q}(X,\Omega^{p}_{X})$ inherits a $L^{2}$ inner product. Similarly, the hermitian metric $h$ also induces hermitian metrics on the vector bundles and spaces of complex differential forms of degree $k$. The complex de Rham cohomology $H^{k}(X,\CBbb)$ has an induced $L^{2}$ inner product by taking $d$-harmonic representatives. The canonical Hodge decomposition
\begin{displaymath}
    H^{k}(X,\CBbb)\simeq \bigoplus_{p,q}^\perp H^{q}(X,\Omega_{X}^{p})
\end{displaymath}
is an isometry for the $L^{2}$ metrics. 

\subsection{The BCOV invariant}\label{subsec:BCOV-invariant}
We briefly recall the construction of the BCOV invariant \cite[Sec. 5]{cdg}. 
Let $X$ be a Calabi--Yau manifold of dimension $n$. Fix a K\"ahler metric $h$ on $X$, with K\"ahler form $\omega$ as in \eqref{eq:normal-omega}. 
Let $T(\Omega_{X}^{p},\omega)$ be the holomorphic analytic torsion of the vector bundle $\Omega_{X}^{p}$ of holomorphc differential $p$-forms endowed with the metric induced by $h$, and with respect to the K\"ahler form $\omega$ on $X$. The \emph{BCOV torsion} of $(X,\omega)$ is
\begin{displaymath}
    T(X,\omega)=\prod_{0\leq p\leq n}T(\Omega_{X}^{p},\omega)^{(-1)^{p}p}.
\end{displaymath}
Let $\Delta^{p,q}_{\ov{\partial}}$ be the Dolbeault Laplacian acting on $A^{p,q}(X)$, and $\det\Delta^{p,q}_{\ov{\partial}}$ its $\zeta$-regularized determinant (excluding the zero eigenvalue). Unraveling the definition of holomorphic analytic torsion, we find for the BCOV torsion
\begin{displaymath}
    T(X,\omega)=\prod_{0\leq p,q\leq n}(\det\Delta^{p,q}_{\ov{\partial}})^{(-1)^{p+q}pq}.
\end{displaymath}
It depends on the choice of the K\"ahler metric. A suitable normalization makes it independent of choices. For this purpose, we introduce two real valued quantities. For the first one, let $\eta$ be a basis of $H^{0}(X,K_{X})$, and define as in \cite[Sec. 4]{FLY}
\begin{equation}\label{eq:factorA}
    A(X,\omega)=\exp\left(-\frac{1}{12}\int_{X}(\log\varphi)\mathrm{c}_{n}(T_{X},h)\right),\quad\text{with}\quad \varphi=\frac{i^{n^{2}}\eta\wedge\ov{\eta}}{\|\eta\|_{\l2}^{2}} \frac{n!}{(2\pi\omega)^{n}}.
\end{equation}
For the second one, we consider the largest torsion free quotient of the cohomology groups $H^{k}(X,\ZBbb)$, denoted by $H^{k}(X,\ZBbb)_{\nt}$. These are lattices in the real cohomology groups $H^{k}(X,\RBbb)$. The latter have Euclidean structures induced from the $L^{2}$ inner products on the $H^{k}(X,\CBbb)$. 
We define $\vol_{\l2}(H^{k}(X,\ZBbb),\omega)$ to be the square of the covolume of the lattice $H^{k}(X,\ZBbb)_{\nt}$ with respect to this Euclidean structure, and we put
\begin{equation}\label{eq:factorB}
    B(X,\omega)=\prod_{0\leq k\leq 2n} \vol_{\l2}(H^{k}(X,\ZBbb),\omega)^{(-1)^{k+1}k/2}.
\end{equation}
The \emph{BCOV invariant} of $X$ is then defined to be
\begin{equation}\label{eq:def-BCOV}
    \tau_{\bcov}(X)=\frac{A(X,\omega)}{B(X,\omega)}T(X,\omega)\in\RBbb_{>0}.
\end{equation}
The BCOV invariant depends only on the complex structure of $X$ \cite[Prop. 5.8]{cdg2}. The definition \eqref{eq:def-BCOV} differs from that of  \cite[Def. 5.7]{cdg2} by a factor $(2\pi)^{n^{2}\chi(X)/2}$, due to the different choice of normalization of the $L^2$ metric:
\begin{equation*}
    \langle\alpha,\beta\rangle=
    \int_{X}h(\alpha,\beta)\frac{(2\pi\omega)^{n}}{n!}.
\end{equation*}

\subsection{The arithmetic Riemann--Roch theorem}\label{subsec:ARR}
In this subsection, we work over an arithmetic ring. This means an excellent regular domain $A$ together with a finite set $\Sigma$ of embeddings $\sigma \colon A \hookrightarrow \CBbb$, closed under complex conjugation. For example, $A$ could be a number field with the set of all its complex embeddings, or the complex field $\CBbb$. Denote by $K$ the field of fractions of $A$.

Let $X$ be an arithmetic variety, \emph{i.e.} a regular, integral, flat and quasi-projective scheme over $A$. For every embedding $\sigma\colon A\hookrightarrow\CBbb$, the base change $X_{\sigma}=X \times_{A, \sigma} \CBbb$ is a quasi-projective and smooth complex variety, whose associated analytic space $X_{\sigma}^{\an}$ is therefore a quasi-projective complex manifold. It is convenient to define $X^{\an}$ as the disjoint union of the $X_{\sigma}^{\an}$, indexed by $\sigma$. For instance, when $A$ is a number field, then $X^{\an}$ is the complex analytic space associated to $X$ as an arithmetic variety over $\QBbb$. Differential geometric objects on $X^{\an}$ such as line bundles, differential forms, metrics, etc. may equivalently be seen as collections of corresponding objects on the $X^{\an}_{\sigma}$. The complex conjugation induces an anti-holomorphic involution on $X^{\an}$, and it is customary in Arakelov geometry to impose some compatibility of the analytic data with this action. Let us now recall the definitions of the arithmetic Picard and first Chow groups of $X$.
\begin{definition}
A smooth hermitian line bundle on $X$ consists in a pair $(L,h)$, where
\begin{itemize}
    \item $L$ is a line bundle on $X$.
    \item $h$ is a smooth hermitian metric on the holomorphic line bundle $L^{\an}$ on $X^{\an}$ deduced from $L$, invariant under the action of the complex conjugation. Hence, $h$ is a conjugation invariant collection $\lbrace h_{\sigma}\rbrace_{\sigma\colon A\to\CBbb}$, where $h_{\sigma}$ is a smooth hermitian metric on the holomorphic line bundle $L_{\sigma}^{\an}$ on $X_{\sigma}^{\an}$ deduced from $L$ by base change and analytification. 
\end{itemize}
The set of isomorphism classes of hermitian line bundles $(L,h)$, with the natural tensor product operation, is a commutative group denoted by $\widehat{\Pic}(X)$ and called the arithmetic Picard group of $X$.
\end{definition}

\begin{definition} The first arithmetic Chow group $\ACH^{1}(X)$ of $X$ is the commutative group
\begin{itemize}
    \item generated by arithmetic divisors, \emph{i.e.} couples $(D, g_D)$, where $D$ is a Weil divisor on $X$ and $g_{D}$ is a Green current for the divisor $D^{\an}$, compatible with complex conjugation. Hence, by definition $g_{D}$ is a degree 0 current on $X^{\an}$ that is a $\hess$-potential for the current of integration $\delta_{D^{\an}}$
    \begin{displaymath}
        \hess g_{D}+\delta_{D^{\an}}=[\omega_{D}],
    \end{displaymath}
    up to some smooth differential $(1,1)$ form $\omega_{D}$ on $X^{\an}$.
    \item with relations $\left(\Div(\phi), [-\log|\phi|^2]\right)$, for non-zero rational functions $\phi$ on $X$.
\end{itemize}
\end{definition}

The arithmetic Picard and first Chow groups are related via the first arithmetic Chern class
$$\cHat_{1}: \widehat{\Pic}(X) \to \ACH^{1}(X),$$
which maps a hermitian line bundle $(L,h)$ to the class of the arithmetic divisor $\left(\Div(\ell), [-\log \|\ell\|_h^{2}]\right)$, where $\ell$ is any non-zero rational section of $L$. This is in fact an isomorphism. We refer the reader to \cite[Sec. 2]{GS:characteristic} for a complete discussion.

More generally, Gillet--Soul\'e developed a theory of arithmetic cycles and Chow rings \cite{GS:AIT}, an arithmetic $K$-theory and characteristic classes \cite{GS:characteristic, GS:characteristic-2}, and an arithmetic Riemann--Roch theorem \cite{GS:ARR}. While for the comprehension of the  theorem below only $\ACH^{1}$, $\widehat{\Pic}$ and $\cHat_{1}$ are needed, the proof uses all this background, for which we refer to the above references. 

Let now $f\colon \Xcal\to S$ be a smooth projective morphism of arithmetic varieties of relative dimension $n$, with generic fiber $X_{\infty}$. To simplify the exposition, we assume that $S\to\Spec A$  is surjective and has geometrically connected fibers. In particular, that $S_{\sigma}^{\an}$ is connected for every embedding $\sigma$. More importantly, we suppose that the fibers $X_{s}$ are Calabi--Yau, hence they satisfy $K_{X_{s}}=\Ocal_{X_{s}}$. We define the BCOV line bundle on $S$ as the determinant of cohomology of the virtual vector bundle $\sum_{p}(-1)^{p}p\Omega_{\Xcal/S}^{p}$, that is, in additive notation for the Picard group of $S$
\begin{equation}\label{eq:BCOV-bundle}
         \lambda_{\bcov}(\Xcal/S)=\sum_{p=0}^{n}(-1)^{p}p\lambda(\Omega_{\Xcal/S}^{p})
                            =\sum_{p,q}(-1)^{p+q}p\det R^{q}f_{\ast}\Omega^{p}_{\Xcal/S}.
\end{equation}
If there is no possible ambiguity, we will sometimes write $\lambda_{\bcov}$ instead of $\lambda_{\bcov}(\Xcal/S)$. 

For the following statement, we fix an auxiliary conjugation invariant K\"ahler metric $h$ on $T_{\Xcal^{\an}}$. We denote by $\omega$ the associated K\"ahler form, normalized according to the conventions in Arakelov theory as in \eqref{eq:normal-omega}. We assume that the restriction of $\omega$ to fibers (still denoted by $\omega$) has rational cohomology class. All the $L^{2}$ metrics below are computed with respect to $\omega$ as in \eqref{eq:normal-L2}.
Depending on the K\"ahler metric, the line bundle $\lambda_{\bcov}$ carries a Quillen metric $h_{\quillen}$
\begin{displaymath}
    h_{\quillen, s}=T(X_{s},\omega)\cdot h_{\l2, s}.
\end{displaymath}
Following \cite[Def. 4.1]{cdg}  and \cite[Def. 5.2]{cdg2}, the Quillen-BCOV metric on $\lambda_{\bcov}$ is defined by multiplying $h_{\quillen}$ by the correcting factor $A$ in \eqref{eq:factorA}: for every $s\in S^{\an}$, we put
\begin{displaymath}
    h_{\quillen,\bcov, s}=A(X_{s},\omega)\cdot h_{\quillen, s}.
\end{displaymath}
It is shown in \emph{loc. cit.} that the Quillen-BCOV is actually a smooth hermitian metric, independent of the choice of $\omega$. Besides, according to \cite[Def. 5.4]{cdg} one defines the $L^{2}$-BCOV metric on $\lambda_{\bcov}$ by
\begin{equation}\label{eq:B}
    h_{\l2,\bcov,s}=B(X_{s},\omega)\cdot h_{\l2,s},
\end{equation}
where $h_{\l2}$ stands for the combination of $L^{2}$-metrics on the Hodge bundles and $B$ was introduced in \eqref{eq:factorB}. In \emph{loc. cit.} we showed that the function $s\mapsto B(X_{s},\omega)$ is actually locally constant and that $h_{\l2,\bcov}$ is a smooth hermitian metric, independent of the choice of $\omega$. Notice that the BCOV invariant defined in \eqref{eq:def-BCOV} can then be written as the quotient of the Quillen-BCOV and $L^{2}$-BCOV metrics:
\begin{equation}\label{eq:def-BCOV-2}
    \tau_{\bcov}(X_{s})=\frac{h_{\quillen,\bcov,s}}{h_{\l2, \bcov, s}}.
\end{equation}

\begin{theorem}\label{thm:ARR} Under the above assumptions, there is an equality in $\ACH^{1}(S)_{\QBbb}=\ACH^{1}(S)\otimes\QBbb$
\begin{equation}\label{eq:ARR-BCOV}
    \cHat_{1}(\lambda_{\bcov}, h_{\quillen,\bcov}) = \frac{ \chi(X_\infty)}{12} \cHat_{1}(f_\ast K_{\Xcal/S}, h_{\l2}).
\end{equation}
 Hence, for any complex embedding $\sigma$,
any rational section $\eta$ of $f_\ast K_{\Xcal/S}$,
   any rational section $\eta_{p,q}$ of $\det R^{q}f_{\ast}\Omega^{p}_{\Xcal/S}$,
   we have an equality of functions on $S_{\sigma}^{\an}$
\begin{equation}\label{eq:bcov-ARR}
    \log \tau_{\bcov,\sigma} 
    =\log|\Delta|^{2}_{\sigma} +\frac{\chi(X_\infty)}{12}\log\|\eta\|_{\l2,\sigma}^{2}
    -\sum_{0\leq p,q\leq n} (-1)^{p+q}p\log\|\eta_{p,q}\|_{\l2,\sigma}^{2}
     +\log C_{\sigma},
\end{equation}
where:
\begin{itemize}
\item $\Delta \in K(S)^\times\otimes_{\ZBbb}\QBbb$.
     \item $C_{\sigma}\in\pi^{r}\QBbb_{>0}$, where $r=\frac{1}{2}\sum (-1)^{k+1} k^{2}b_{k}$ and $b_{k}$ is the $k$-th Betti number of $X_{\infty}$.
\end{itemize}
\end{theorem}
\begin{proof}
The proof is a routine application of the arithmetic Riemann--Roch theorem of Gillet--Soul\'e \cite[Thm. 7]{GS:ARR}. We give the details for the convenience of the reader. Consider the virtual vector bundle $\sum(-1)^{p}p\Omega_{\Xcal/S}^{p}$, with virtual hermitian structure deduced from the metric $h$, and denoted $h^{\bullet}$. Its determinant of cohomology $\lambda_{\bcov}$ carries the Quillen metric $h_{\quillen}$. The theorem of Gillet--Soul\'e provides an equality in $\ACH^{1}(S)_{\QBbb}$
\begin{eqnarray}
 \cHat_{1}(\lambda_{\bcov}, h_{\quillen})
    &=&f_{\ast}\left(\chHat(\sum(-1)^{p}p\Omega_{\Xcal/S}^{p}),h^{\bullet})\tdHat(T_{\Xcal/S},h)\right)^{(1)}\nonumber \\
 &&   -a\left(\ch(\sum(-1)^{p}p\Omega_{\Xcal^{\an}/S^{\an}}^{p})\td (T_{\Xcal^{\an}/S^{\an}}) R(T_{\Xcal^{\an}/S^{\an}})\right)^{(1)}\nonumber \\
    &=&\frac{1}{12}f_{\ast}\left(\cHat_{1}(K_{\Xcal/S},h^{\ast})\cHat_{n}(T_{\Xcal/S},h)\right),\label{eq-arr-1}
\end{eqnarray}
where $h^{\ast}=(\det h)^{-1}$ is the hermitian metric on $K_{\Xcal/S}$ induced from $h$. Notice that the topological factor containing the $R$-genus in \emph{loc. cit.} vanishes in our situation, since
\begin{displaymath}
    \ch\left(\sum(-1)^{p}p\Omega_{\Xcal^{\an}/S^{\an}}^{p}\right)\td(T_{\Xcal^{\an}/S^{\an}})=-\mathrm{c}_{n-1}+\frac{n}{2}\mathrm{c}_n-\frac{1}{12}\mathrm{c}_1 \mathrm{c}_n
+\textrm{higher degree terms}
\end{displaymath}
and $R$ has only odd degree terms and $\mathrm{c}_1(T_{\Xcal^{\an}/S^{\an}})=0$.
Now, the evaluation map $f^{\ast}f_{\ast}K_{\Xcal/S}\to K_{\Xcal/S}$ is an isomorphism, but it is in general not an isometry if we equip $f_{\ast}K_{\Xcal/S}$ with the $L^{2}$ metric and $K_{\Xcal/S}$ with the metric $h^{\ast}$. Comparing both metrics yields a relation in $\ACH^{1}(\Xcal)$
\begin{equation}\label{eq-arr-2}
    \cHat_{1}(K_{\Xcal/S},h^{\ast})=f^{\ast}\cHat_{1}(f_{\ast}K_{\Xcal/S},h_{\l2})+[(0,-\log\varphi)].
\end{equation}
Here  $\varphi$ is the smooth function on $X^{\an}$ given by
\begin{displaymath}
\varphi=\frac{i^{n^{2}}\eta\wedge\ov{\eta}}{\|\eta\|_{\l2}^{2}}\frac{n!}{(2\pi\omega)^{n}},
\end{displaymath}
where $\eta$ denotes a local trivialization of $f_{\ast}K_{\Xcal^{\an}/S^{\an}}$, thought of as a section of $K_{\Xcal^{\an}/S^{\an}}$ via the evaluation map. 
Multiplying \eqref{eq-arr-2} by $\cHat_{n}(T_{\Xcal/S},h)$ and applying $f_{\ast}$ and the projection formula for arithmetic Chow groups, we find
\begin{displaymath}
    \begin{split}
        f_{\ast}\left(\cHat_{1}(K_{\Xcal/S},h^{\ast})\cHat_{n}(T_{\Xcal/S},h)\right)=&f_{\ast}\left(f^{\ast}\cHat_{1}(f_{\ast}K_{\Xcal/S},h_{\l2})\cHat_{n}(T_{\Xcal/S},h)\right)+f_{\ast}\left([(0,-\log\varphi)]\cHat_{n}(T_{\Xcal/S},h)\right)\\
        =&\chi(X_{\infty})\cHat_{1}(f_{\ast}K_{\Xcal/S},h_{\l2})
        +\left[\left(0,-\int_{\Xcal^{\an}/S^{\an}}(\log\varphi)\mathrm{c}_{n}(T_{\Xcal^{\an}/S^{\an}},h)\right) \right],
    \end{split}
\end{displaymath}
where $\mathrm{c}_{n}(T_{\Xcal/S},h)$ is the $n$-th Chern--Weil differential form of $(T_{\Xcal^{\an}/S^{\an}},h)$.
Together with \eqref{eq-arr-1}, this shows that the metric
\begin{displaymath}
    h_{\quillen,\bcov}=h_{\quillen}\cdot\exp\left(-\frac{1}{12}\int_{\Xcal^{\an}/S^{\an}}(\log\varphi)\mathrm{c}_{n}(T_{\Xcal^{\an}/S^{\an}},h)\right)
\end{displaymath}
indeed satisfies \eqref{eq:ARR-BCOV}. 

The outcome \eqref{eq:bcov-ARR} is a translation of the meaning of the equality \eqref{eq:ARR-BCOV} in $\ACH^{1}(S)_{\QBbb}$, in terms of the constructions \eqref{eq:def-BCOV-2} and \eqref{eq:B}. 

By \cite[Prop. 4.2]{cdg2} the normalizing factor $B$ is constant on each connected manifold $S_{\sigma}^{\an}$ and would be rational if the $L^2$ inner products on cohomology groups were computed with $h/2\pi$.

With this understood, we find
\begin{equation}\label{eq:volume-not-rational}
    \vol_{\l2}(H^{k}(X_{s},\ZBbb),\omega)
    \in (2\pi)^{-kb_{k}}\QBbb^{\times}_{>0}
\end{equation}
for any $s\in S_{\sigma}^{\an}$. Together with the definition of $B$ \eqref{eq:factorB}, this is responsible for the constants $C_{\sigma}$.
\end{proof}

\begin{remark}
\begin{enumerate}
    \item The use of the arithmetic Riemann--Roch theorem requires an algebraic setting, but directly yields the existence of the rational function $\Delta$. By contrast, previous techniques (cf. \emph{e.g.} \cite[Sections 7 \& 10]{FLY}) rely on subtle integrability estimates of the functions in~\eqref{eq:bcov-ARR}, in order to ensure that the \emph{a priori} pluriharmonic function $\log|\Delta|_{\sigma}^{2}$ is indeed the logarithm of a rational function. The arithmetic Riemann--Roch theorem further provides the field of definition of $\Delta$ and the constants $C_{\sigma}$. 
    \item In the case of a Calabi--Yau 3-fold defined over a number field, similar computations were done by Maillot--R\"ossler \cite[Sec. 2]{MaRo}.
\end{enumerate}
\end{remark}

\subsection{Kronecker limit formulas for families of Calabi--Yau hypersurfaces}\label{subsec:Kronecker-hyp}
In this section we give an example of use of Theorem \ref{thm:ARR} and we determine the BCOV invariant for families of Calabi--Yau hypersurfaces in Fano manifolds. The argument provides a simplified model for the later computation of the BCOV invariant of the mirror family of Calabi--Yau hypersurfaces.  

Let $V$ be a complex Fano manifold, with very ample anti-canonical bundle $-K_V$. We consider the anti-canonical embedding of $V$ into $|-K_V| = \PBbb(H^0(V, -K_V))\simeq\PBbb^{N}$, whose smooth hyperplane sections are Calabi--Yau manifolds.  The dual projective space $\check{\PBbb} = \PBbb(H^0(V, -K_V)^\vee)\simeq\check{\PBbb}^{N}$ parametrizes hyperplane sections, and contains an irreducible subvariety $\Delta \subseteq \check{\PBbb}$ which corresponds to  singular such sections \cite[Chap. 1, Prop. 1.3]{GKZ}. We assume that $\Delta$ is a hypersurface in $\check{\PBbb}$. This is in general not true, and a necessary condition is proven in  \cite[Chap. 1, Cor. 1.2]{GKZ}. Denote by $U$ the quasi-projective complement $U := \check{\PBbb} \setminus \Delta$. Denote by $f \colon \mathcal{X} \to \check{\PBbb}$ the universal family of hyperplane sections. Therefore $f$ is smooth on $U$, and the corresponding BCOV line bundle $\lambda_{\bcov}$ is thus defined on $U$.

\begin{lemma}
For some positive integer $m$, the line bundles $(f_{\ast}K_{\Xcal/U})^{\otimes m}$ and $\lambda_{\bcov}^{\otimes m}$ have trivializing sections. These are unique up to constants.
\end{lemma}
\begin{proof}

A standard computation shows that $\Pic(U) = \ZBbb/\deg \Delta$, providing the first claim of the lemma. For the second assertion, for any of the line bundles under consideration, let $\theta$ and $\theta^{\prime}$ be two trivializations on $U$. Therefore, $\theta=h\theta^{\prime}$ for some invertible function $h$ on $U$. The previous description of $\Pic(U)$ shows that the divisor of $h$, as a rational function on $\check{\PBbb}$, is supported on $\Delta$. As $\Delta$ is irreducible, in the projective space $\check{\PBbb}$ this is only possible if the divisor vanishes. We conclude that $h$ is necessarily constant.
\end{proof}

For the following statement, we need a choice of auxiliary K\"ahler metric on $\Xcal$ (restricted to $U$), whose Arakelov theoretic K\"ahler form has fiberwise rational cohomology class. We compute $L^{2}$ norms on Hodge bundles and on $\lambda_{\bcov}$ with respect to this choice.
\begin{theorem} \label{thm:Kroneckerhypersurf}
For some integer $m>0$ as in the lemma, let $\beta$ be a trivialization of $\lambda_{\bcov}^{\otimes m}$ and $\eta$ a trivialization of $(f_\ast K_{\mathcal{X}/U})^{\otimes m}$. Then there is a global constant $C$ such that, for any Calabi--Yau hyperplane section $X_H = V \cap H$, we have
\begin{displaymath}
    \taubcov{X_H} = C  \|\eta\|_{\l2}^{\chi/6m}
\|\beta\|^{-2/m}_{\l2}.
\end{displaymath}
\end{theorem}

\begin{proof}
We apply Theorem \ref{thm:ARR} to $f\colon\Xcal\to U$ (over $\CBbb$), which in terms of $\beta$ and $\eta$ becomes 
\begin{displaymath}
    m\log \taubcov{X_H}= \log|g|^2+\frac{\chi}{12}\log\|\eta\|^2_{\l2}-\log \|\beta\|^2_{\l2}+\log C
\end{displaymath}
 for some regular invertible function $g$ on $U$ and some constant $C$. By construction, as a rational function on $\check{\PBbb}$, $g$ must have its zeros or poles along $\Delta$. Since $\Delta$ is irreducible this forces $g$ to be constant.
\end{proof}
\begin{remark}
\begin{enumerate}
    \item When $V$ is a toric variety with very ample anti-canonical class, all of the constructions can in fact be done over the rational numbers. The sections $\beta$ and $\eta$ can be taken to be defined over $\QBbb$, and unique up to a rational number. With this choice, the constant $C$ takes the form stated in Theorem \ref{thm:ARR}. 
    \item In the case when the discriminant $\Delta$ has higher codimension, we have $\Pic(U)\simeq\Pic(\check{\PBbb})$. In particular, $\lambda_{\bcov}$ uniquely extends to a line bundle $\check{\PBbb}$. The existence of the canonical (up to constant) trivializations $\beta$ and $\eta$ is no longer true. However, one can propose a variant of the theorem where $\beta$ and $\eta$ are trivializations outside a chosen ample divisor in $\check{\PBbb}$. 
\end{enumerate}
\end{remark}

\section{The Dwork and mirror families, and their Hodge bundles} \label{sec:dworkandmirrorfamilies}
The main object of interest of this section is the mirror family of Calabi--Yau hypersurfaces. It is obtained from the Dwork pencil of Calabi--Yau varieties, by first modding out by a group of generic symmetries, and then by performing a crepant resolution. We study the structure of the Hodge bundles of the mirror family. In even dimension, we show that the primitive Hodge bundles in the middle degree can be decomposed in two direct factors. One will be seen to be constant in Section \ref{sec:degHodgebund}, and the other one is called the minimal part. For the latter, we construct explicit trivializations via Griffiths' residue method.

Throughout, our arguments combine analytic and algebraic aspects of the same geometric objects. Except when there is a risk of confusion, we won't make any distinction in the notations between an algebro-geometric object and its analytification. Likewise, we won't specify the field of definition of various algebraic varieties and schemes. However, we will precisely indicate the category where the statements take place.

\subsection{The geometry of the Dwork family}
We review general facts on the Dwork pencil of Calabi--Yau hypersurfaces, and the construction of an equivariant normal crossings model. Initially, we work with algebraic varieties over the field of complex numbers. Rationality refinements will be made along the way.

Let $n\geq 4$ be an integer. The Dwork pencil $\Xcal\to \PBbb^1$ is defined by the hypersurface of $\PBbb^{n}\times\PBbb^{1}$ of equation
\begin{displaymath}
    F_\psi (x_{0},\ldots,x_{n}) := \sum_{j=0}^{n} x_j^{n+1} - (n+1) \psi x_0 \ldots x_n = 0,\quad [x_0:x_1:\ldots:x_n]\in\PBbb^n, \quad \psi \in \PBbb^1.
\end{displaymath}
The smooth fibers of this family are Calabi--Yau manifolds of dimension $n-1$. The singular fibers are:
\begin{itemize}
    \item fiber at $\psi=\infty$, given by the divisor with normal crossings $x_{0}\cdot\ldots\cdot x_{n}=0$.
    \item the fibers where $\psi^{n+1}=1$. These fibers have ordinary double point singularities. The singular points have projective coordinates $(x_{0},\ldots, x_{n})$ with $x_0=1$ and $x_{j}^{n+1}=1$ for all $j\geq 1$, and $\prod_{j} x_{j}=\psi^{-1}$.
\end{itemize}

Denote by $\mu_{n+1}$ the group of the $(n+1)$-th roots of unity. Let $K$ be the kernel of the multiplication map $\mu_{n+1}^{n+1}\to\mu_{n+1}$. Let also $\Delta$ be the diagonal embedding of $\mu_{n+1}$ in $K$ and $G:=K/\Delta$. The group $G$ acts naturally on the fibers $X_\psi$ of $\Xcal \to \PBbb^1$ by multiplication of the projective coordinates.

The above constructions can be realized as schemes over $\QBbb$. Indeed, $F_{\psi}$ is already defined over $\QBbb$, and the groups $K$, $\Delta$ are finite algebraic groups over $\QBbb$, and hence so does the quotient $G$. The action of $G$ on $F_{\psi}$ is defined over $\QBbb$ as well, as one can see by examining the compatibility with the action of $\Aut(\CBbb/\QBbb)$ on the $\CBbb$ points of $\Xcal$, or alternatively by writing the co-action at the level of algebras. 

The following argument was provided to us by Nicholas Shepherd-Barron, whom we warmly thank for letting us include it in our article.
\begin{proposition}\label{prop:NSB}
The family $\Xcal \to \PBbb^1$  admits a $G$-equivariant projective normal crossings model $\Xcal^{\prime}\to\PBbb^{1}$, with $\Xcal^{\prime}$ non-singular, which is semi-stable at $\psi = \infty$ and defined over $\QBbb$.
\end{proposition}

\begin{proof}
Outside of the singular points there is nothing to modify. The points corresponding to ordinary double point singularities are provided by the affine equations $x_{0}=1$ and $x_j^{n+1} = 1$ for $j\geq 1$, and blowing up along the corresponding locus of ${\Xcal}$ provides a normal crossings model. The locus of ordinary double points is defined over $\QBbb$ and is $G$-equivariant, and so is thus also the blowup.

We next consider our family at the point at infinity. Introduce the divisor $D_0$ in $\PBbb^{n}$ given by $\sum_{j} x_j^{n+1}=0$ and the divisor $D_\infty=\sum_{j} H_{j}$, where $H_{j}$ is the hyperplane cut out by $x_{j}=0$. The axis of the pencil $\Xcal \to \PBbb^1$ is  $D_0 \cap D_\infty$ and hence $\Xcal = \mathrm{Bl}_{D_0 \cap D_\infty}(\PBbb^n)$. We construct another model by blowing up $\PBbb^n$ in $D_0 \cap H_{0}$ to get $\Xcal_1$. Continue to blow up the strict transform of $D_0$ in $\Xcal_1$ intersected by the strict transform of $H_1$, and so on. Each such blowup is a blowup in a smooth center which is $G$-equivariant. The final result is a $G$-equivariant  $\widetilde{\Xcal}$ projective manifold, with an equivariant morphism $\nu: \widetilde{\Xcal} \to \PBbb^{n}$. Denote by $\widetilde{D_0}$ (resp. $\widetilde{D_\infty}$) the strict transforms of $D_0$ (resp. $D_\infty$). By construction they are disjoint, and  computation shows that $\nu^* D_0 \sim \widetilde{D_0} + \sum E_i$ and $\nu^* (D_\infty) \sim \widetilde{D_\infty} + \sum E_i,$ where the $E_i$ denote the strict transforms of the exceptional divisors. Since $D_0-D_\infty$ is the divisor of a rational function, hence linearly equivalent to zero, and $\widetilde{D_\infty}$ is disjoint from $\widetilde{D_0}$, we find a morphism $p: \widetilde{\Xcal} \to \PBbb^1$ such that $p^{-1}(\infty)=\widetilde{D_\infty}$ and $p^{-1}(0)= \widetilde{D_0}$. This is the searched for semi-stable model at infinity. From the local description we also see that $\nu^{-1}(D_0 \cap D_\infty)= \sum E_i$ which is principal, so that $\widetilde{\Xcal} \to \PBbb^1$ factors over $\Xcal \to \PBbb^1$. 

All of the above constructions can be defined over $\QBbb$, and taking them together with the previous considerations with the ordinary double points provides a model $\Xcal^\prime \to \PBbb^1$ as in the statement of the proposition. 
\end{proof}

\subsection{The mirror family}
The first step towards the construction of a mirror family is the formation of the quotient $\Ycal=\Xcal/G$. As the action of $G$ on $\Xcal$ is defined over $\QBbb$, the space $\Ycal$  and the projection map $\Ycal\to\PBbb^{1}$ are also. The following lemma shows that except for the fiber at infinity, this is a family of singular Calabi--Yau varities with mild singularities. 

\begin{lemma}\label{lemma:gorenstein}
The total space of the restricted family $\Ycal\to\ABbb^{1}$ has rational Gorenstein singularities. It has a relative canonical line bundle  $K_{\Ycal/\ABbb^{1}}$, obtained by descent from $K_{\Xcal/\ABbb^{1}}$.
\end{lemma}
\begin{proof}
To lighten notations, let us write in this proof $\Xcal$ and $\Ycal$ for the corresponding restrictions to $\ABbb^{1}$. The total space $\Xcal$ is non-singular, and $\Ycal$ is a quotient of it by the action of a finite group. Therefore, $\Ycal$ has rational singularities. In particular, it is normal and Cohen--Macaulay. Consequently, if $\Ycal^{ns}$ is the non-singular locus of $\Ycal$, and $j\colon\Ycal^{ns}\hookrightarrow\Ycal$ the open immersion, then we have a relation between relative dualizing sheaves $j_{\ast}\omega_{\Ycal^{ns}/\ABbb^{1}}=\omega_{\Ycal/\ABbb^{1}}$. We will use this below.

Now for the Gorenstein property and the descent claim. Notice that since $\ABbb^{1}$ is non-singular, $\Ycal$ is Gorenstein if, and only if, the fibers of $\Ycal\to\ABbb^{1}$ are Gorenstein. We will implicitly confound both the absolute and relative points of view. We introduce $\Xcal^{\circ}$ the complement of the fixed locus of $G$, and $\Xcal^{\ast}$ the smooth locus of $\Xcal\to\ABbb^{1}$. These are $G$-invariant open subschemes of $\Xcal$ and constitute an open cover, because the ordinary double points in the fibers of $\Xcal\to\ABbb^{1}$ are disjoint from the fixed point locus of $G$. Then $\Ycal^{\circ}=\Xcal^{\circ}/G$ and $\Ycal^{\ast}=\Xcal^{\ast}/G$ form an open cover of $\Ycal$, and it is enough to proceed for each one separately. 

Since $G$ acts freely on $\Xcal^{\circ}$, the quotient $\Ycal^{\circ}$ is non-singular, and is therefore Gorenstein. The morphism $\Xcal^{\circ}\to\Ycal^{\circ}$ is \'etale, and hence $K_{\Xcal^{\circ}/\ABbb^{1}}$ descends to $K_{\Ycal^{\circ}/\ABbb^{1}}$.

For $\Ycal^{\ast}$, we observe that $G$ preserves a relative holomorphic volume form on $\Xcal^{\ast}$. Indeed, in affine coordinates $z_k = \frac{x_k}{x_j}$ on the open set $x_j \neq 0$, and where $\partial F_\psi/\partial z_i \neq 0$, the expression 
\begin{equation}\label{eq:relative-volume-form}
    \theta_{0}=\frac{(-1)^{i-1} dz_0 \wedge \ldots  \widehat{dz_i} \wedge \ldots \wedge \widehat{dz_j} \wedge \ldots \wedge dz_n}{\partial F_\psi/\partial z_i} \Big|_{F_{\psi}=0}
\end{equation}
provides such an invariant relative volume form. This entails that $K_{\Xcal^{\ast}/\ABbb^{1}}$ descends to an invertible sheaf $\Kcal$ on $\Ycal^{\ast}$. Now,  the singular locus of $\Ycal^{\ast}$ is contained in the image of the fixed point set of $G$ on $\Xcal^{\ast}$. We infer that $\Kcal$ is an invertible extension of the relative canonical bundle of $(\Ycal^{\ast})^{ns}\to\ABbb^{1}$. But $\Ycal^{\ast}$ is normal so that $\Kcal \simeq j_{\ast} j^{\ast} \Kcal$. Then, as mentioned at the beginning of the proof, $j_{\ast} \omega_{\Ycal^{ns}/\ABbb^{1}}=\omega_{\Ycal/\ABbb^{1}}$ and we conclude, since $\Kcal$ is also an extension of $\omega_{\Ycal^{ns}/\ABbb^{1}}$.  
\end{proof}

Because the BCOV invariant has not been fully developed for Calabi--Yau orbifolds (see nevertheless \cite{Yoshikawa-orbifold} for some three-dimensional cases), we need crepant resolutions of the varieties $Y_{\psi}$. This needs to be done in families, so that the results of \textsection\ref{subsec:ARR} apply. The family of crepant resolutions $\Zcal\to\PBbb^{1}$ that we exhibit will be called the \emph{mirror family}, although it is not unique. We also have to address the rationality of the construction. 

\begin{lemma}\label{lemma:crepant}
There is a projective birational morphism $\Zcal \to \Ycal$ of algebraic varieties over $\QBbb$, such that:
\begin{enumerate}
    \item\label{item:crepant-1} $\Zcal$ is smooth.  
    \item If $\psi^{n+1} = 1$, the fiber $Z_\psi$ has a single ordinary double point singularity.
    \item If $\psi = \infty$, $Z_\infty$ is a simple normal crossings  divisor in $\Zcal$.
    \item\label{item:crepant-4} Otherwise, $Z_\psi \to Y_\psi$ is a crepant resolution of singularities. In particular, $Z_\psi$ is a smooth Calabi--Yau variety.
    \item\label{item:crepant-5} The smooth complex fibers $Z_{\psi}$ are mirror to the $X_{\psi}$, in that their Hodge numbers satisfy $h^{p,q}(Z_{\psi})=h^{n-1-p,q}(X_{\psi})$. In particular, the smooth $Z_{\psi}$ are Calabi--Yau with $\chi(Z_{\psi})=(-1)^{n-1}\chi(X_{\psi})$.
\end{enumerate}
\end{lemma}

\begin{proof}
The proof of \eqref{item:crepant-1}--\eqref{item:crepant-4} is based on \cite[Sec. 8 (v)]{DHZ}, \cite{DHZ-Gorenstein} and \cite[Prop. 3.1]{Dworkpencil}, together with Hironaka's resolution of singularities. We recall the strategy, in order to justify the existence of a model over $\QBbb$.

Introduce $W=\PBbb^{n}/G$. We claim this is a split toric variety over $\QBbb$. First of all, it can be realized as the hypersurface in $\PBbb^{n+1}_{\QBbb}$ of equation
\begin{equation*}\label{eq:Pn/Ghyper}
    W \colon y_{0}^{n+1}=\prod_{j=1}^{n+1}y_{j}.
\end{equation*}
Second, the associated torus is split over $\QBbb$. It is actually given by $\GBbb_{\mathrm{m}\ \QBbb}\times\TBbb$, where $\TBbb$ is the kernel of the multiplication map $\GBbb_{\mathrm{m}\ \QBbb}^{n+1}\to\GBbb_{\mathrm{m}\ \QBbb}$. Finally, the action of the torus on $W$ is defined over $\QBbb$:
\begin{displaymath}
    \left((t_{0},t_{1},\ldots,t_{n+1}), (y_{0},y_{1},\ldots,y_{n+1})\right)\mapsto (t_{0}y_{0},t_{0}t_{1}y_{1},\ldots, t_{0}t_{n+1}y_{n+1}).
\end{displaymath}
Once we know that $W$ is a split toric variety over $\QBbb$, with same equation  as in \cite[Application 5.5]{DHZ-Gorenstein}, the toric and crepant projective resolution exhibited in \emph{loc. cit.} automatically works over $\QBbb$ as well. We write $\widetilde{W}$ for this resolution of $W$. 

We now consider $\Ycal$ as a closed integral $\QBbb$-subscheme of $W \times\PBbb^{1}$. Let $\widetilde{\Ycal}$ be the strict transform of $\Ycal$ in $\widetilde{W}\times\PBbb^{1}$. By \cite[Sec. 8 (v)]{DHZ}, the fibers of $\widetilde{\Ycal}$ at $\psi\in\CBbb\setminus\mu_{n+1}$ are projective crepant resolutions of the fibers $Y_{\psi}$. In particular, $\widetilde{\Ycal}$ is smooth over $\CBbb\setminus\mu_{n+1}$, and in turn this implies smoothness over the complement $U$ of the closed subscheme $V(\psi^{n+1}-1)$ of $\ABbb^{1}_{\QBbb}$. Necessarily, the fibers of $\widetilde{\Ycal}$ over $U$ have trivial canonical bundle as well. For the fibers at $\psi^{n+1}=1$, the claim of the lemma requires  two observations:
\begin{itemize}
    \item the ordinary double points of $X_{\psi}$ are permuted freely and transitively by $G$, and get identified to a single point in the quotient $Y_{\psi}$. This entails that the total space $\Ycal$ is non-singular in a neighborhood of these points, and that they remain ordinary double points of $\Ycal\to\PBbb^{1}$.
    \item the center of the toric resolution is disjoint from the ordinary double points, since it is contained in the locus of $\PBbb^{n}/G$ where two or more projective coordinates vanish. Therefore, the morphism $\widetilde{\Ycal}\to\Ycal$ is an isomorphism in a neighborhood of these points. Finally, on the complement, $\widetilde{\Ycal}_{\psi}$ is a resolution of singularities of $Y_{\psi}$. Indeed, this is a local question in a neighborhood of the fixed points of $G$, so that the above references \cite{DHZ, DHZ-Gorenstein} still apply.
\end{itemize}
Finally, $\widetilde{\Ycal}$ is by construction smooth on the complement of the fiber $\psi=\infty$. After a resolution of singularities given by blowups with smooth centers in $\widetilde{Y}_{\infty}$ (defined over $\QBbb$), we obtain a smooth algebraic variety $\Zcal$ over $\QBbb$, such that $Z_{\infty}$ is a simple normal crossings divisor in $\Zcal$. This sets items \eqref{item:crepant-1}--\eqref{item:crepant-4}. 

For \eqref{item:crepant-5}, we refer for instance to \cite[Thm. 6.9, Conj. 7.5 \& Ex. 8.7]{Batyrev-Dais}. This is specific to the Dwork pencil. More generally, we can cite work of Yasuda, who proves an invariance property of orbifold Hodge structures (and hence orbifold Hodge numbers) under crepant resolutions, for quotient Gorenstein singularities \cite[Thm. 1.5]{Yasuda}. Orbifold Hodge numbers coincide with stringy Hodge numbers of global (finite) quotient orbifolds, whose underlying group respects a holomorphic volume form \cite[Thm. 6.14]{Batyrev-Dais}. Finally, by \cite[Thm. 4.15]{Batyrev-Borisov}, stringy Hodge numbers satisfy the expected mirror symmetry property for the mirror pairs constructed by Batyrev \cite{Batyrev}.
\end{proof}

From the proof of Lemma \ref{lemma:crepant}, we keep the notation $U\subset\PBbb^{1}$ for the smooth locus of the mirror family $f\colon \Zcal\to \PBbb^{1}$. For later use, we record the following lemma.
\begin{lemma}\label{lemma:Hodge-numbers-crepant}
Let $h^{p,q}$ be the rank of the Hodge bundle $R^{q}f_{\ast}\Omega^{p}_{\Zcal/U}$. Then:
\begin{itemize}
    \item $h^{p,q}=1$ if $p+q=n-1$ and $p\neq q$.
    \item $h^{p,p}=\sum_{j=0}^{p}(-1)^{j}{{n+1}\choose j}{{(p+1-j)n+p}\choose{n}}+\delta_{2p,n-1}$.
    \item $h^{p,q}=0$ otherwise.
\end{itemize}
In particular, 
\begin{equation*}\label{chi}
    \chi(Z_\psi) =(-1)^{n-1}\chi(X_{\psi})= (-1)^{n-1} \left(\frac{(-n)^{n+1}-1}{n+1}+n+1\right).
\end{equation*}
\end{lemma}
\begin{proof}
The items are a consequence of the mirror symmetry property for the Hodge numbers in Lemma \ref{lemma:crepant}, and the computation of the cohomology of a hypersurface in projective space (cf. \cite[Ex. 8.7]{Batyrev-Dais}).
\end{proof}
\begin{definition}
The point $\infty\in\PBbb^{1}$ is called the MUM point of the family $f\colon\Zcal\to\PBbb^{1}$. The points $\xi\in\PBbb^{1}$ with $\xi^{n+1}=1$ are called the ODP points.
\end{definition}
The terminology MUM stands for \emph{Maximally Unipotent Monodromy}, and it will be justified later in Lemma \ref{lemma:MHS-infty}. The terminology ODP stands for \emph{Ordinary Double Point}.

\subsection{Generalities on Hodge bundles}\label{subsec:hodge-bundles}
We gather general facts on the Hodge bundles of our families of Calabi--Yau varieties, summarized in the following diagram:
\begin{equation}\label{eq:diagram-mirror-families}
    \xymatrix{
        &               &\Xcal\ar[d]_{\rho}\ar@/^1pc/[ddr]^{h}        &\\
        &\Zcal\ar[r]^{\hspace{-0.7cm}\substack{\text{crepant}\\ \pi}}\ar@/_1pc/[drr]_{f}          &\Ycal=\Xcal/G\ar[dr]^{g}         &\\
        &       &       &\PBbb^{1}.
    }
\end{equation}
Recall the notation $U$ for the Zariski open subset of $\PBbb^{1}$ where $f$ (resp. $h$) is smooth. When it is clear from the context, we will still write $\Xcal$, $\Ycal$ and $\Zcal$ for the total spaces of the fibrations restricted to $U$. Otherwise, we add an index $U$ to mean the restriction to $U$. We let $\Ycal^{\circ}$ be the non-singular locus of $\Ycal_{U}$. It is the \'etale quotient of $\Xcal^{\circ}$, the complement in $\Xcal_{U}$ of the fixed point set of $G$. They are both open subsets whose complements have codimension $\geq 2$. 

In this subsection, most of the arguments take place in the complex analytic category.

\subsubsection*{Hodge bundles in arbitrary degree}
Our discussion is based on a minor adaptation of \cite[Sec. 1]{Steenbrink-mixedonvanishing} to the relative setting.  First of all, we observe that the higher direct images $R^{k}g_{\ast}\CBbb$ are locally constant sheaves, and actually $R^{k}g_{\ast}\CBbb\simeq(R^{k}h_{\ast}\CBbb)^{G}$. Indeed, we have the equality $\CBbb_{\Ycal}=(\rho_{\ast}\CBbb_{\Xcal})^{G}$. Moreover, since $G$ is finite, so is $\rho$ and taking $G$-invariants is an exact functor in the category of sheaves of $\CBbb[G]$-modules. A spectral sequence argument allows us to conclude. Similarly, one has $R^{k}g_{\ast}\QBbb\simeq(R^{k}h_{\ast}\QBbb)^{G}$.

Let now $\widetilde{\Omega}_{\Ycal/U}^{\bullet}$ be the relative holomorphic de Rham complex of $\Ycal\to U$, in the orbifold sense. It is constructed as follows. If $j\colon\Ycal^{\circ}\hookrightarrow\Ycal_{U}$ is the open immersion, then we let $\widetilde{\Omega}_{\Ycal_{U}}^{\bullet}:=j_{\ast}\Omega_{\Ycal^{\circ}}^{\bullet}$, and we derive the relative version $\widetilde{\Omega}_{\Ycal/U}^{\bullet}$ out of it in the usual manner. An equivalent presentation is
\begin{displaymath}
    \widetilde{\Omega}_{\Ycal/U}^{\bullet}=(\rho_{\ast}\Omega_{\Xcal/U}^{\bullet})^{G},
\end{displaymath}
The complex $\widetilde{\Omega}_{\Ycal/U}^{\bullet}$ is a resolution of $g^{-1}\Ocal_{U}$. Hence its $k$-th relative hypercohomology computes $(R^{k}g_{\ast}\CBbb)\otimes\Ocal_{U}$, and satisfies 
\begin{equation}\label{eq:relation-coh-quotient}
    R^{k}g_{\ast}\widetilde{\Omega}_{\Ycal/U}^{\bullet}\simeq (R^{k}h_{\ast}\Omega_{\Xcal/U}^{\bullet})^{G},
\end{equation}
compatibly with $R^{k}g_{\ast}\CBbb\simeq (R^{k}h_{\ast}\CBbb)^{G}$. It has a Hodge filtration and a Gauss--Manin connection defined in the usual way, satisfying a relationship analogous to \eqref{eq:relation-coh-quotient}. Equipped with this extra structure, $R^{k}g_{\ast}\QBbb$ defines a variation of pure rational Hodge structures of weight $k$. 

In \cite[Lemma 1.11]{Steenbrink-mixedonvanishing}, a canonical identification $\widetilde{\Omega}_{\Ycal_{U}}^{\bullet}=\pi_{\ast}\Omega_{\Zcal_{U}}^{\bullet}$ is established. It induces a natural morphism  
\begin{equation}\label{eq:morphism-dR-complexes}
    \widetilde{\Omega}_{\Ycal/U}^{\bullet}\longrightarrow \pi_{\ast}(\Omega_{\Zcal/U}^{\bullet}).
\end{equation}
The restriction of \eqref{eq:morphism-dR-complexes} to $\Ycal^{\circ}$ is given by pulling back differential forms. We derive a natural map
\begin{equation}\label{eq:morphism-hodge-crepant}
    (R^{k}h_{\ast}\Omega_{\Xcal/U}^{\bullet})^{G}\simeq R^{k}g_{\ast}\widetilde{\Omega}_{\Ycal/U}^{\bullet}\longrightarrow R^{k}f_{\ast}\Omega_{\Zcal/U}^{\bullet},
\end{equation}
which is an injective morphism of variations of pure Hodge structures of weight $k$, cf. \cite[Cor. 1.5]{Steenbrink-mixedonvanishing}. It is in particular compatible with restricting to the fibers, and remains injective on those. It can be checked to be compatible with the topological $\QBbb$-structures, and hence we have an injective morphism of variations of rational Hodge structures over $U$
\begin{equation}\label{eq:morphism-hodge-crepant-bis}
    (R^{k}h_{\ast}\QBbb)^{G}\hookrightarrow R^{k}f_{\ast}\QBbb.
\end{equation}
Notice that, at this stage, the compatibility of \eqref{eq:morphism-hodge-crepant} with the algebraic geometric $\QBbb$-structure has not been addressed. This will be studied in later subsections. 

\subsubsection*{Hodge bundles in the middle degree} 
In the case $k=n-1$, considering the isotypical components of the action of $G$ on $R^{n-1}h_{\ast}\CBbb$, we have a direct sum decomposition, 
\begin{equation}\label{eq:decomposition-isotypical}
    R^{n-1}h_{\ast}\CBbb=(R^{n-1}h_{\ast}\CBbb)^{G}\oplus\EBbb_{\CBbb},\quad\text{where}\quad\EBbb_{\CBbb}=\bigoplus_{\substack{\chi\colon G\to\CBbb^{\times}\\ \chi\not\equiv 1}}(R^{n-1}h_{\ast}\CBbb)_{\chi}.
\end{equation}
This decomposition is easily seen to be orthogonal for the intersection form on $R^{n-1}h_{\ast}\CBbb$. In particular, the restriction of the intersection form to $(R^{n-1}h_{\ast}\CBbb)^{G}$ is non-degenerate, and Poincar\'e duality holds for $R^{n-1}g_{\ast}\CBbb\simeq (R^{n-1}h_{\ast}\CBbb)^{G}$. Notice that the orthogonal of $(R^{n-1}h_{\ast}\QBbb)^{G}$ in $ R^{n-1}h_{\ast}\QBbb$ defines a rational structure on $\EBbb_{\CBbb}$, and hence \eqref{eq:decomposition-isotypical} can be refined rationally. 

We next relate the intersection forms of $(R^{n-1}h_{\ast}\QBbb)^{G}$ and $R^{n-1}f_{\ast}\QBbb$. Before the first statement in this direction, we recall from Lemma \ref{lemma:gorenstein} that $\Ycal_{U}$ is Gorenstein, and $K_{\Xcal/U}$ descends to the relative canonical bundle  $K_{\Ycal/U}$.
\begin{lemma}\label{lemma:cup-prod}
\begin{enumerate}
    \item $\widetilde{\Omega}_{\Ycal/U}^{n-1}$ is the relative canonical bundle $K_{\Ycal/U}$.
    \item The natural morphism $R^{n-1}g_{\ast}\widetilde{\Omega}_{\Ycal/U}^{\bullet}\longrightarrow R^{n-1}f_{\ast}\Omega_{\Zcal/U}^{\bullet}$ induces a commutative diagram
    \begin{displaymath}
    \xymatrix{
       & R^{q}g_{\ast}\widetilde{\Omega}^{p}_{\Ycal/U}\otimes R^{n-1-q}g_{\ast}\widetilde{\Omega}^{n-1-p}_{\Ycal/U}\ar[r]\ar[dd]  &R^{n-1}g_{\ast}K_{\Ycal/U}\ar@{=}[dd]\ar[rd]^{\tr}\\
        &                                                                                   &       &\Ocal_{U}\\
        &R^{q}f_{\ast}\Omega^{p}_{\Zcal/U}\otimes R^{n-1-q}f_{\ast}\Omega^{n-1-p}_{\Zcal/U}\ar[r]     &R^{n-1}f_{\ast}K_{\Zcal/U}\ar[ru]_{\tr}\\
    }
    \end{displaymath}
    \item The natural isomorphism $R^{n-1}g_{\ast}\widetilde{\Omega}_{\Ycal/U}^{\bullet}\simeq (R^{n-1}h_{\ast}\Omega_{\Xcal/U}^{\bullet})^{G}$ induces a commutative diagram
    \begin{displaymath}
    \xymatrix{
        &R^{q}g_{\ast}\widetilde{\Omega}^{p}_{\Ycal/U}\otimes R^{n-1-q}g_{\ast}\widetilde{\Omega}^{n-1-p}_{\Ycal/U}\ar[r]\ar@{^{(}->}[d]  &R^{n-1}g_{\ast}K_{\Ycal/U}\ar@{^{(}->}[d]\ar[r]^{\hspace{0.5cm}\tr}
        &\Ocal_{U}\ar[d]^{|G|\cdot}\\
        &R^{q}h_{\ast}\Omega^{p}_{\Xcal/U}\otimes R^{n-1-q}h_{\ast}\Omega^{n-1-p}_{\Xcal/U}\ar[r]  &R^{n-1}h_{\ast}K_{\Xcal/U}\ar[r]^{\hspace{0.5cm}\tr}
        &\Ocal_{U}
    }
\end{displaymath}
\end{enumerate}
\end{lemma}
\begin{proof}
For the first property, we notice that $\rho^{\ast}K_{\Ycal/U}=K_{\Xcal/U}$, since both coincide outside a codimension $\geq 2$ closed subset and $\Xcal_{U}$ is smooth. Then we have the string of equalities
\begin{displaymath}
    \widetilde{\Omega}^{n-1}_{\Ycal/U}=(\rho_{\ast}K_{\Xcal/U})^{G}=(K_{\Ycal/U}\otimes\rho_{\ast}\Ocal_{\Xcal_{U}})^{G}
    =K_{\Ycal/U}\otimes(\rho_{\ast}\Ocal_{\Xcal_{U}})^{G}=K_{\Ycal/U}.
\end{displaymath}
For the first diagram, only the commutativity of the triangle requires a justification. For this, we rely on general facts in duality theory. Our references are stated in the algebraic category. Corresponding complex analytic properties are obtained by analytification. With this understood, the commutativity of the triangle is a consequence of the three following facts: i) the transitivity of trace maps with respect to composition of morphisms \cite[Thm. 10.5 (TRA1)]{Hartshorne}; ii) the crepant resolution property $\pi^{\ast}K_{\Ycal/U}=K_{\Zcal/U}$ and iii) $\Ycal_{U}$ has rational singularities, so that $R\pi_{\ast}\Ocal_{{\Zcal}_{U}}=\Ocal_{\Ycal_{U}}$. The argument is similar for the second diagram. Briefly, one combines: i) the transitivity of trace maps; ii) the duality $\rho_{\ast}K_{\Xcal_{U}/\Ycal_{U}}=\Hom_{\Ocal_{\Ycal_{U}}}(\rho_{\ast}\Ocal_{\Xcal_{U}},\Ocal_{\Ycal_{U}})$ and iii) the trace $\tr\colon\rho_{\ast}K_{\Xcal_{U}/\Ycal_{U}}\to\Ocal_{\Ycal_{U}}$ is given by $\varphi\mapsto\varphi(1)$ \cite[proof of Prop. 6.5]{Hartshorne}, and the composite map
\begin{displaymath}
    K_{\Ycal/U}\longrightarrow\rho_{\ast}K_{\Xcal/U}=K_{\Ycal/U}\otimes \rho_{\ast}K_{\Xcal_{U}/\Ycal_{U}}\overset{\id\otimes\tr}\longrightarrow
    K_{\Ycal/U}
\end{displaymath}
is the multiplication by $|G|$. This is clear over $\Ycal^{\circ}$, since it is the \'etale quotient of $\Xcal^{\circ}$ by $G$. It is then necessarily true everywhere.
\end{proof}

\begin{proposition}\label{prop:forme-intersection-restreinte}
Let $Q$ be the intersection form on $R^{n-1}f_{\ast}\QBbb$ and $Q^{\prime}$ the intersection form on $R^{n-1}h_{\ast}\QBbb$. Then, via the injection \eqref{eq:morphism-hodge-crepant-bis}, we have $Q=\frac{1}{|G|}Q^{\prime}$ on $(R^{n-1}h_{\ast}\QBbb)^{G}$.
\end{proposition}
\begin{proof}
It is enough to check the relationship after extending the scalars to $\CBbb$, in which case we can use the Hodge decomposition. The proposition then follows from Lemma \ref{lemma:cup-prod} and the fact that in the middle degree the intersection form is induced by the cohomological cup product and the trace map. We notice that in dimension $n-1$, the topological and complex geometric trace maps differ by a factor $(2\pi i)^{n-1}$, but this is inconsequential for the problem at hand.
\end{proof}

\begin{remark}\label{rmk:cup-product}
\begin{enumerate}
    \item\label{item:cup-product-1} In the case of direct images of relative canonical sheaves, the discussion in the proof of Lemma \ref{lemma:cup-prod} reduces to the chain of isomorphisms of line bundles
\begin{equation}\label{eq:iso-canonical}
       (h_{\ast}K_{\Xcal/U})^{G}\overset{\sim}{\longrightarrow}g_{\ast}K_{\Ycal/U}\overset{\sim}{\longrightarrow} f_{\ast}K_{\Zcal/U}. 
\end{equation}
We leave to the reader to check that these are the natural morphisms already defined in the algebraic category over $\QBbb$.
    \item\label{item:cup-product-2} Because of Proposition \ref{prop:forme-intersection-restreinte}, and for the purposes of this article, it is natural to scale the intersection form on $(R^{n-1}h_{\ast}\QBbb)^{G}$ as $\frac{1}{|G|}Q^{\prime}$. This will be of minor importance below.
\end{enumerate}
\end{remark}

\subsection{The Kodaira--Spencer maps and the Yukawa coupling}\label{subsec:KSYukawa}

Recall that for a general variation of Hodge structures $(\Hcal,\Fcal^{\bullet})$ on a complex manifold $X$, Griffiths transversality entails that the Gauss--Manin connection factors as an $\Ocal_X$-linear morphism $\mathcal{F}^p/\mathcal{F}^{p+1} \to \left(\mathcal{F}^{p-1}/\mathcal{F}^{p}\right) \otimes \Omega^1_{X}$. This is the Kodaira--Spencer map, and in the setting of $R^{n-1}f_{\ast}\Omega^{\bullet}_{\Zcal/U}$ we also write it in the form
\begin{equation}\label{def:KS}
    \KS^{(q)}\colon T_{U}\longrightarrow\Hom_{\Ocal_{U}}(R^{q}f_{\ast}\Omega^{n-1-q}_{\Zcal/U}, R^{q+1}f_{\ast}\Omega^{n-2-q}_{\Zcal/U}). \end{equation}
A repeated application of the Kodaira--Spencer maps gives a morphism
\begin{equation}\label{eq:Yukawa-def}
    Y\colon \Sym^{n-1} T_{U}\longrightarrow\Hom_{\Ocal_{U}}(f_{\ast}K_{\Zcal/U}, R^{n-1}f_{\ast}\Ocal_{\Zcal})\simeq (f_{\ast}K_{\Zcal/U})^{\otimes -2}.
\end{equation}
We can explicitly evaluate the morphism $Y$ in terms of the sections  $\psi d/d\psi$ of $T_{U}$ and the section $\theta_{0}$ of $(h_{\ast}K_{\Xcal/U})^{G}\simeq f_{\ast}K_{\Zcal/U}$ (cf. \eqref{eq:iso-canonical}) constructed in \eqref{eq:relative-volume-form}. Then the morphism $Y$ identifies with a rational function on $U$, denoted $Y(\psi)$. This is the definition of the so-called (unnormalized) \emph{Yukawa coupling}.  

Working with $(R^{n-1}h_{\ast}\Omega_{\Xcal/U}^{\bullet})^{G}$ instead, one similarly defines a function $\widetilde{Y}(\psi)$. Via the morphism \eqref{eq:morphism-hodge-crepant}, the functions $\widetilde{Y}(\psi)$ and $Y(\psi)$ can be compared. The only subtle point to bear in mind is the use of Serre duality in the definition of the Yukawa coupling. For Hodge bundles of complementary bi-degree, Serre duality is induced by the cup-product and the trace morphism. Hence, an application of Lemma \ref{lemma:cup-prod} shows that $Y(\psi)$ and $\widetilde{Y}(\psi)$ are equal up to the order of $G$. With this understood, we can invoke the computation of the Yukawa coupling in \cite[Cor. 4.5.6 \& Ex. 4.5.7]{batyrev-straten}, which summarizes to
\begin{equation}\label{eq:yukawa}
    Y( \psi)=\int_{X_{\psi}} \left (\theta_0 \wedge\nabla_{\psi d/d\psi}^{n-1}  \theta_0 \right)=c\frac{\psi^{n-1}}{1-\psi^{n+1}},
\end{equation}
for some irrelevant constant $c\neq 0$. To ease the comparison with the expression in \emph{loc. cit.}, we make the following observations. First, their factor $\lambda z$ is $1/\psi^{n+1}$. Secondly, their evaluation of $Y$ amounts to working with the section $\psi \theta_0$ instead of $\theta_{0}$.

\subsection{The middle degree Hodge bundles}\label{subsec:Griffiths-sections}
We now further compare the middle degree Hodge bundles of the Dwork pencil $h\colon\Xcal\to U$ and that of the mirror $f\colon\Zcal\to U$, by drawing on specific features of these families. We introduce primitivity notions for the relative Hodge bundles, induced by any projective factorization of $f$ and the natural projective embedding of $h$. Observe the latter is $G$-equivariant and defined over $\QBbb$. We also require the polarization for $\Zcal\to U$ to be defined over $\QBbb$. Then the primitive Hodge bundles are defined in the algebraic category over $\QBbb$.  

\subsubsection*{Construction of sections} We begin by constructing explicit sections of the middle degree Hodge bundles of $h\colon\Xcal\to U$, via Griffiths' residue method \cite{Griffiths-residues}. 

Our reasoning starts in the complex analytic category. Denote by $H = x_0 \cdot x_1\cdot  \ldots \cdot x_n$ and $\Omega = \sum (-1)^i x_i dx_0\wedge \ldots \wedge\widehat{dx_i}  \wedge\ldots \wedge dx_n \in H^{0}(\PBbb^n,\Omega^{n}_{\PBbb^{n}}(n+1))$. For $\psi\in U$, the residue along $X_\psi$
\begin{equation*}\label{eq:residue-theta}
    \theta_k = {\res}_{X_\psi}\left( \frac{k! H^{k} \Omega}{F_\psi^{k+1}} \right)
\end{equation*}
defines a $G$-invariant element of $H^{n-1}(X_{\psi})$, still denoted $\theta_k$. 
For $k=0$, this indeed agrees with the holomorphic volume form \eqref{eq:relative-volume-form}.
Varying $\psi$ gives us sections of $R^{n-1}h_{\ast}\Omega_{\Xcal/U}^{\bullet}$, also denoted by $\theta_k$. The constructed sections are primitive by \cite[Thm. 8.3]{Griffiths-residues}.

From the definition of the sections $\theta_{k}$, one can check the following recurrence:
\begin{equation}\label{eq:recurrence-theta}
    \nabla_{d/d\psi}\ \theta_{k}={{\res}_{X_\psi}\left(\frac{\partial}{\partial \psi} \left(\frac{k! H^{k} \Omega}{F_\psi^{k+1}} \right)\right)}=(n+1)\theta_{k+1}.
\end{equation}

\begin{lemma}\label{lemma:thetakconstruction}
\begin{enumerate}
    \item\label{item:thetakconstruction-1} For $k=0,\ldots,n-1$, we have
\begin{displaymath}
   \theta_k\in F^{n-1-k}H^{n-1}(X_{\psi})^{G}_{\prim}.
\end{displaymath}
Moreover, the spaces $H^{n-1-k,k}(X_{\psi})^{G}_{\prim}$ are all one-dimensional and the image of $\theta_k$ in $H^{n-1-k,k}(X_{\psi})^{G}_{\prim}$ is a basis for $\psi \in U $. In particular, the local system $(R^{n-1} h_\ast \QBbb)_{\prim}^G$ is of rank $n$.
    \item The sections $\theta_{k}$ trivialize $(R^{n-1}h_{\ast}\Omega_{\Xcal/U}^{\bullet})^{G}_{\prim}$ outside of $0$, and are algebraic and defined over $\QBbb$.
\end{enumerate}
\end{lemma}
\begin{proof}
For the first item, the spaces $H^{n-1-k,k}(X_{\psi})^{G}_{\prim}$ are necessarily one-dimensional, which follows from a computation in the case of Fermat hypersurfaces, see \cite[p. 82, Rmk. 7.5]{deligne:hodge-cycles}. For the  rest of \eqref{item:thetakconstruction-1}, we use Griffiths's description of the Hodge filtration of a hypersurface, in terms of residues of rational forms, reviewed in \cite[Chapter 6]{Voisin-II}. 

By \cite[Thm. 6.10]{Voisin-II}, we indeed have for $k=0,\ldots,n-1$, $\theta_k\in F^{n-1-k}H^{n-1}(X_{\psi})^{G}_{\prim}$. We need to verify that the projections of the sections $\theta_k$ onto $H^{n-1-k, k}(X_\psi)_\prim^G$ are everywhere non-zero on $U$. The following argument was suggested by the anonymous referee, whom we thank for allowing us to include it. A detailed study of Griffiths' residue map, see e.g. \cite[Corollary 6.12]{Voisin-II},  provides an isomorphism $[\CBbb[x_0, \ldots, x_n]/J]_{(n+1)\cdot k} \to H^{n-1-k, k}(X_\psi)_\prim$, where $J$ denotes the Jacobian ideal of $X_\psi$ in $\PBbb^n$ and the index $(n+1)k$ refers to the homogeneous part of the corresponding degree. Recall the notation $H=x_{0}\cdot\ldots\cdot x_{n}$. Since $\CBbb[x_0, \ldots, x_n]^G = \CBbb[x_0^{n+1}, \ldots, x_n^{n+1}, H]$, we find that 
\begin{displaymath}
    \left[\CBbb[x_0^{n+1}, \ldots, x_n^{n+1}, H]/J^G\right]_{(n+1)\cdot k} \simeq  H^{n-1-k, k}(X_\psi)_\prim^G,
\end{displaymath}
where $J^{G}=J\cap\CBbb[x_0, \ldots, x_n]^G$. A straightforward computation shows that $x_i^{n+1} \equiv \frac{\psi}{n+1} H$ modulo $J^G$, so that in fact \begin{displaymath}
    \CBbb[x_0^{n+1}, \ldots, x_n^{n+1}, H]/J^G \simeq \CBbb[H]/J^G.
\end{displaymath}
Now the image of $\theta_k$ in $H^{n-1-k, k}(X_\psi)_\prim^G$ corresponds to the image of $k!H^{k}$ in $[\CBbb[H]/J^G]_{(n+1)k}$ through the above isomorphisms, and the latter is a generator of $[\CBbb[H]/J^G]_{(n+1)k}$, hence non-zero. This, the projection of $\theta_{k}$ gives a basis of $H^{n-1-k, k}(X_\psi)_\prim^G$. 

For the second item, we just need to address the second half of the statement. We observe that the section $\theta_{0}$ of $(h_{\ast}K_{\Xcal/U})^{G}$ is algebraic and defined over $\QBbb$. By the algebraic theory of the Gauss--Manin connection \cite{Katz-Oda}, we know that the latter preserves the algebraic de Rham cohomology $(R^{n-1}h_{\ast}\Omega_{\Xcal/U}^{\bullet})^{G}_{\prim}$, and is defined over $\QBbb$. Because the vector field $d/d\psi$ is algebraic and defined over $\QBbb$, the claim follows from the recurrence \eqref{eq:recurrence-theta}.
\end{proof}

\begin{remark}
An alternative approach to the non-vanishing of the projection of the sections $\theta_{k}$ onto $H^{n-1-k, k}(X_\psi)_\prim^G$, is based on the explicit expression of the Yukawa coupling \eqref{eq:yukawa} and the realization of the sections $\theta_k$ as iterated Gauss--Manin derivatives via \eqref{eq:recurrence-theta}. If either of $\theta_k$ have zero projection for some $\psi$, applying the Kodaira--Spencer map in \eqref{def:KS} and the recurrence in \eqref{eq:recurrence-theta}, we see that all the projections of $\theta_{k'}$ with $k' \geq k$ are also zero at $\psi$. This implies that the Yukawa coupling, divided by $\psi^{n-1}$ in order to work with the tangent vector $d/d\psi$ instead of $\psi d/d\psi$, also has a zero at $\psi$. But the expression in \eqref{eq:yukawa} divided by $\psi^{n-1}$ has no zeros on $U$, from which we  conclude. 

\end{remark}

\subsubsection*{The minimal component of the cohomology}
Below we show that the image of the primitive middle cohomology of the Dwork family under \eqref{eq:morphism-hodge-crepant-bis} is a direct factor of the cohomology of the mirror. Later in Lemma \ref{lemma:trivial-V-local} we will see that the complement is in fact irrelevant for most considerations.

\begin{lemma}\label{prop:iso-prim-coh} 
\begin{enumerate}
    \item\label{item:iso-prim-coh-1} The natural morphism \eqref{eq:morphism-hodge-crepant-bis} induces an injective morphism of variations of polarized Hodge structures over $U^{\an}$
\begin{equation}\label{eq:iso-prim-coh}
    (R^{n-1}h_{\ast}\QBbb)^{G}_{\prim}{\hookrightarrow} (R^{n-1}f_{\ast}\QBbb)_{\prim}.
\end{equation}
    \item\label{item:iso-prim-coh-2} The natural morphism
    \begin{equation}\label{eq:iso-prim-coh-bis-bis}
        (R^{n-1}h_{\ast}\Omega_{\Xcal/U}^{\bullet})^{G}_{\prim}{\hookrightarrow} (R^{n-1}f_{\ast}\Omega_{\Zcal/U}^{\bullet})_{\prim}
    \end{equation}
    deduced from \eqref{eq:iso-prim-coh} $\otimes\ \Ocal_{U^{\an}}$, exists in the algebraic category over $\QBbb$. 
\end{enumerate}
\end{lemma}
\begin{proof}
For the proof of \eqref{item:iso-prim-coh-1}, it is enough to show that \eqref{eq:morphism-hodge-crepant} restricts to a map between the primitive cohomologies. It will automatically be compatible with the polarizations, by Proposition \ref{prop:forme-intersection-restreinte}. See Remark \ref{rmk:cup-product} \eqref{item:cup-product-2} regarding the scaling of the intersection forms. By Lemma \ref{lemma:thetakconstruction}, it suffices to check that the sections $\theta_{k}$ of $(R^{n-1}h_{\ast}\Omega_{\Xcal/U}^{\bullet})^{G}_{\prim}$ map into primitive classes. Let $\theta_{k}^{\prime}$ be the image of $\theta_{k}$ under \eqref{eq:morphism-hodge-crepant}. As \eqref{eq:morphism-hodge-crepant} is compatible with Gauss-Manin connections, the $\theta_{k}^{\prime}$ satisfy the analogous recurrence to \eqref{eq:recurrence-theta}. Because $f_{\ast}K_{\Zcal/U}$ is primitive and the Gauss--Manin connection preserves primitive cohomology, we see that the $\theta_{k}^{\prime}$ land in the primitive cohomology. 

The claim in \eqref{item:iso-prim-coh-2} is addressed in a similar manner. By Lemma \ref{lemma:thetakconstruction}, we already know that the sections $\theta_{k}$ constitute an algebraic trivialization of $(R^{n-1}h_{\ast}\Omega_{\Xcal/U}^{\bullet})^{G}_{\prim}$, defined over $\QBbb$. We need to prove that their images $\theta_{k}^{\prime}$ in $(R^{n-1}f_{\ast}\Omega_{\Zcal/U}^{\bullet})_{\prim}$ are algebraic and defined over $\QBbb$ as well. This is the case of $\theta_{0}^{\prime}$, because the natural isomorphism $(h_{\ast}K_{\Xcal/U})^{G}\simeq f_{\ast}K_{\Zcal/U}$ (cf. \eqref{eq:iso-canonical}) is algebraic and defined over $\QBbb$. In this respect, see Remark \ref{rmk:cup-product} \eqref{item:cup-product-1}. Because the $\theta_{k}^{\prime}$ satisfy the analogous recurrence to \eqref{eq:recurrence-theta}, and the Gauss--Manin connection and the vector field $d/d\psi$ are algebraic and defined over $\QBbb$, we conclude. 
\end{proof}

Notice that the image of $(R^{n-1} h_\ast \QBbb)_{\prim}^G$ under \eqref{eq:iso-prim-coh} is the smallest subvariation of Hodge structures of $R^{n-1} f_\ast \QBbb$ whose Hodge filtration contains $f_\ast K_{\Zcal/U}$ (see \eqref{eq:iso-canonical}). This motivates the following definition:
\begin{definition}\label{definition:min}
The image of $(R^{n-1} h_\ast \QBbb)_{\prim}^G$ in $(R^{n-1} f_\ast \QBbb)_{\prim}$ under the morphism \eqref{eq:iso-prim-coh}, is denoted by $(R^{n-1} f_\ast \QBbb)_{\min}$, and called the minimal component or minimal part. Likewise, we decorate algebraic variants (cf. \eqref{eq:iso-prim-coh-bis-bis}) and associated objects by ${\min}$. For example, this applies to Hodge bundles and homology constructions. 
\end{definition}

The next step consists in isolating the complement of the minimal component. In preparation for the statement, we recall that the topological intersection form on $R^{n-1}f_{\ast}\CBbb$ has a counterpart on the de Rham cohomology $R^{n-1}f_{\ast}\Omega^{\bullet}_{\Zcal/U}$, which is already defined in the algebraic category over $\QBbb$. Indeed, the construction of the latter involves the cohomological cup-product, the graded product structure on the complex $\Omega_{\Zcal/U}^{\bullet}$ and the algebraic geometric trace map:
\begin{displaymath}
    R^{n-1}f_{\ast}\Omega_{\Zcal/U}^{\bullet}\otimes R^{n-1}f_{\ast}\Omega_{\Zcal/U}^{\bullet}\overset{\cup}{\longrightarrow} R^{2(n-1)}f_{\ast}\Omega^{\bullet}_{\Zcal/U}=R^{n-1}f_{\ast}K_{\Zcal/U}\overset{\tr}{\longrightarrow}\Ocal_{U}.
\end{displaymath}
After forming $(R^{n-1}f_{\ast}\CBbb)\otimes\Ocal_{U^{\an}}$, the topological and algebraic intersection pairings agree up to a factor $(2\pi i)^{n-1}$, which accounts for the comparison of the trace maps. We are now ready for the next result.

\begin{proposition}[Minimal decomposition]\label{prop:orthogonal-decompositions}
\begin{enumerate}
    \item\label{item:orth-dec-1} Let $\VBbb$ be the orthogonal of $(R^{n-1}f_{\ast}\QBbb)_{\min}$ in $(R^{n-1}f_{\ast}\QBbb)_{\prim}$ for the topological intersection form. Then, there is an orthogonal decomposition of variations of polarized rational Hodge structures over $U^{\an}$
    \begin{equation}\label{eq:iso-prim-coh-bis}
        (R^{n-1}f_{\ast}\QBbb)_{\prim}= (R^{n-1}f_{\ast}\QBbb)_{\min}\oplus\VBbb.
    \end{equation}
   Furthermore, we have
   \begin{displaymath}
        \VBbb=\begin{cases}
            0 &\text{if }\ n-1\ \text{is odd},\\
            \text{of pure type } \left(\frac{n-1}{2},\frac{n-1}{2}\right) &\text{ if }\ n-1\ \text{is even}.
        \end{cases}
   \end{displaymath}
   \item\label{item:orth-dec-2} Let $\Vcal$ be the orthogonal of $(R^{n-1}f_{\ast}\Omega_{\Zcal/U}^{\bullet})_{\min}$ in $(R^{n-1}f_{\ast}\Omega_{\Zcal/U}^{\bullet})_{\prim}$ for the algebraic geometric intersection form. Then, there is a direct sum decomposition of locally free coherent sheaves with connection over $U$, in the algebraic category over $\QBbb$,
   \begin{equation}\label{eq:iso-prim-coh-bis-alg}
        (R^{n-1}f_{\ast}\Omega_{\Zcal/U}^{\bullet})_{\prim}=(R^{n-1}f_{\ast}\Omega_{\Zcal/U}^{\bullet})_{\min}\oplus\Vcal.
   \end{equation}
   Furthermore, the analytification of \eqref{eq:iso-prim-coh-bis-alg} is naturally identified with \eqref{eq:iso-prim-coh-bis}$\ \otimes\ \Ocal_{U^{\an}}$.
\end{enumerate}

\end{proposition}
\begin{proof}
We first deal with \eqref{item:iso-prim-coh-1}. In the case of $n-1$ being odd, $(R^{n-1}f_{\ast}\QBbb)_{\prim}=(R^{n-1}f_{\ast}\QBbb)_{\min}$, by the very definition of the minimal component and by Lemma \ref{lemma:Hodge-numbers-crepant} and Lemma \ref{lemma:thetakconstruction}. In the case $n-1$ is even, we first notice that the intersection pairing is flat for the Gauss--Manin connection, and that the orthogonal complement of a subvariation of rational Hodge structures in a variation of polarized rational Hodge structures, is also a variation of polarized rational Hodge structures.  Thus, $\VBbb$ is a variation of polarized rational Hodge structures. To obtain the decomposition \eqref{eq:iso-prim-coh-bis} with the required properties, we can reduce to the following general fact.  Let $(H,Q)$ be a polarized Hodge structure over $\QBbb$, of weight $2d$, and $(E,Q)$ a sub-Hodge structure, such that $E^{(p,q)}=H^{(p,q)}$ for $p\neq q$. Let $V=E^{\perp}$ be the orthogonal of $E$ for the intersection form $Q$. Then  $H=E\oplus V$ and $V$ is a Hodge structure over $\QBbb$, of pure type $(d,d)$. To prove this fact, by linear algebra and the non-degeneracy of the intersection form,  it is enough to verify that $E \cap V$  is trivial. Take any element  $x$ in the intersection, and decompose it in $H_\CBbb$ according to the bidegree as $x = \sum x^{p,q}$. Then $\overline{x^{p,q}} \in E^{q,p}\subset E_\CBbb$. On the other hand, $i^{p-q} Q(x , \overline{x^{p,q}}) = i^{p-q} Q(x^{p,q}, \overline{x^{p,q}}) \geq 0$, with equality only if $x^{p,q} = 0$. But this is the case since $x\in E^{\perp}$, proving the decomposition. It follows from the assumption $E^{(p,q)}=H^{(p,q)}$ for $p\neq q$ that the complement is of pure type $(d,d)$.

For item \eqref{item:orth-dec-2}, we first notice that since $(R^{n-1}f_{\ast}\Omega_{\Zcal/U}^{\bullet})_{\min}$ and $(R^{n-1}f_{\ast}\Omega_{\Zcal/U}^{\bullet})_{\prim}$ are locally free coherent sheaves, so is $\Vcal$. Besides, the algebraic Gauss--Manin connection preserves $\Vcal$, since it preserves the minimal component and the algebraic intersection form is flat. By the compatibility of the topological and algebraic intersection forms, the analytification of $\Vcal$ is canonically identified with $\VBbb\otimes\Ocal_{U^{\an}}$. For the validity of the direct sum decomposition, we can reduce to the analytic setting, in which case it follows from \eqref{eq:iso-prim-coh-bis}$\ \otimes\ \Ocal_{U^{\an}}$.
\end{proof}

\begin{remark}\label{rmk:homology-min}
After Proposition \ref{prop:orthogonal-decompositions}, and with the conventions adopted in Definition \ref{definition:min}, for the homology local systems we have 
\begin{equation}\label{eq:prim-min-decomposition-homology}
    (R^{n-1}f_{\ast}\QBbb)^\vee_\prim= (R^{n-1}f_{\ast}\QBbb)_{\min}^\vee\oplus\VBbb^\vee.
\end{equation}
We can thus consider $ (R^{n-1}f_{\ast}\QBbb)_{\min}^\vee$ as a subsystem of $(R^{n-1}f_{\ast}\QBbb)^\vee_{\prim}$, which in turn can be seen as a subsystem of the homology local system $(R^{n-1}f_{\ast}\QBbb)^\vee$. This allows us to interpret $(R^{n-1}f_{\ast}\QBbb)_{\min}^\vee$ in terms of homology classes and Poincar\'e duals of these in terms of integration. 
\end{remark}

In the application of the arithmetic Riemann--Roch theorem to the BCOV conjecture, we will need sections of the Hodge bundles rather than the Hodge filtration, cf. Theorem \ref{thm:ARR}. This is the reason behind the following definition.
\begin{definition}\label{def:eta-k}
We define $\eta_{k}^{\circ}$ as the trivializing section of $(R^{k}f_{\ast}\Omega_{\Zcal/U}^{n-1-k})_{\min}$, deduced from $\theta_{k}$ via the morphism \eqref{eq:iso-prim-coh} and by projecting to the Hodge bundle.  We also define $\eta_{k}=-(n+1)^{k+1}\psi^{k+1}\eta_{k}^{\circ}$.
\end{definition}
\begin{remark}\label{rmk:vanishing-eta-0}
\begin{enumerate}
    \item\label{item:vanishin-eta-1} By construction, the section $\eta_{k}$ vanishes at order $k+1$ at $\psi=0$.
    \item\label{item:vanishin-eta-2} The sections $\eta_{k}$ are algebraic and defined over $\QBbb$ by Lemma \ref{lemma:thetakconstruction} and Lemma \ref{prop:iso-prim-coh}. 
\end{enumerate}
\end{remark}

\begin{lemma}\label{lemma:recurrence-eta}
The sections $\eta_{k}^{\circ}$ satisfy the recurrence
\begin{equation}\label{eq:recurrence-eta}
    \KS^{(k)}\left(\frac{d}{d\psi}\right)\eta_{k}^{\circ}=(n+1)\eta_{k+1}^{\circ}.
\end{equation}
Consequently, 
\begin{equation}\label{eq:recurrence-eta-norm}
    \KS^{(k)}\left(\psi\frac{d}{d\psi}\right)\eta_{k}=\eta_{k+1}.
\end{equation}
\end{lemma}
\begin{proof}
The first recurrence follows from \eqref{eq:recurrence-theta}, Lemma~\ref{prop:iso-prim-coh}, the link between the Gauss--Manin connection $\nabla$ and the Kodaira--Spencer maps $\KS^{(q)}$, and the definition of $\eta_{k}^{\circ}$. The second recurrence follows from the first one, by the very definition of the sections $\eta_{k}$ and the $\Ocal_{U}$-linearity of the Kodaira--Spencer maps.
\end{proof}

\section{The degeneration of the Hodge bundles of the mirror family}\label{sec:degHodgebund}
In the previous section we exhibited explicit trivializing sections of the minimal part of the middle degree Hodge bundles of the mirror family $\Zcal\to U$. The next goal is to extend these sections to the whole compactification $\PBbb^{1}$. We also address the trivialization of the Hodge bundles other than the minimal part and in any degree. For these goals, we exploit the approach to degenerating Hodge structures via relative logarithmic de Rham cohomology. 

\subsection{Generalities on geometric degenerations of Hodge structures}\label{subsec:generalities-hodge}
We recall some background from Steenbrink \cite{Steenbrink-limits, Steenbrink-mixedonvanishing} and our previous work \cite[Sec. 2 \& Sec. 4]{cdg2}. We also refer to Illusie's survey \cite[Sec. 2.2 \& Sec. 2.3]{illusie} Let $f\colon \Xcal\to\DBbb$ be a projective morphism of reduced analytic spaces, over the unit disc $\DBbb$. We suppose that the fibers $X_{t}$ with $t\neq 0$ are smooth and connected. We consider the variation of Hodge structures associated to $R^{k}f_{\ast}\QBbb$ over the punctured disc $\DBbb^{\times}$. Let $T$ be its monodromy operator and $\nabla$ the Gauss--Manin connection on the holomorphic vector bundle $(R^{k}f_{\ast}\QBbb)\otimes\Ocal_{\DBbb^{\times}}=R^{k}f_{\ast}\Omega_{\Xcal/\DBbb^{\times}}^{\bullet}$. Recall that $T$ is a quasi-unipotent transformation of the cohomology of the general fiber. The flat vector bundle $(R^{k}f_{\ast}\Omega_{\Xcal/\DBbb^{\times}}^{\bullet}, \nabla)$  has a unique extension to a vector bundle on $\DBbb$, such that $\nabla$ extends to a regular singular connection, whose residue $\Res_{0}\nabla$ is an endomorphism with eigenvalues in $[0,1)\cap\QBbb$. This is the Deligne (lower) canonical extension, denoted by $^{\ell}R^{k}f_{\ast}\Omega_{\Xcal/\DBbb^{\times}}^{\bullet}$. Occasionally, we may simply refer to it as the Deligne extension of $R^{k}f_{\ast}\CBbb$. It can be realized as the hypercohomology $R^{k}f^{\prime}\Omega_{\Xcal^{\prime}/\DBbb}^{\bullet}(\log)$ of the logarithmic de Rham complex of a normal crossing model $f'\colon \Xcal'\to\DBbb$. The Hodge filtration $\Fcal^{\bullet}$ on $R^{k}f_{\ast}\Omega_{\Xcal/\DBbb^{\times}}^{\bullet}$ extends to a filtration by vector sub-bundles, still denoted by $\Fcal^{\bullet}$. Its locally free graded quotients are of the form $R^{k-p}f^{\prime}\Omega_{\Xcal^{\prime}/\DBbb}^{p}(\log)$. If the monodromy operator is unipotent, then the fiber of $R^{k}f^{\prime}\Omega_{\Xcal^{\prime}/\DBbb}^{\bullet}(\log)$ at $0$, together with the restricted Hodge filtration, can be identified with the cohomology of the generic fiber $H^{k}_{\lim}$ with the limiting Hodge filtration $F^{\bullet}_{\infty}$. The identification depends on the choice of a holomorphic coordinate on $\DBbb$. There is also the monodromy weight filtration $W_{\bullet}$ on $H^{k}_{\lim}$, attached to the nilpotent operator $N=-2\pi i\Res_{0}\nabla$. The triple $(H^{k}_{\lim}, F^{\bullet}_{\infty}, W_{\bullet})$ is called the limiting mixed Hodge structure. It is isomorphic to Schmid's limiting mixed Hodge structure \cite{schmid} on the cohomology of the general fiber. In particular, $W_{\bullet}$ admits a rational structure. This structure is not needed in the current section, but it will be used later in Section \ref{sec:conjectures}, actually in the greater generality of higher dimensional parameter spaces. In the general quasi-unipotent case, one first performs a semi-stable reduction and then constructs the limiting mixed Hodge structure.

More generally, for a subvariation of Hodge structures $\EBbb$ of $R^{k}f_{\ast}\QBbb$, which is a direct summand, the previous constructions can also be carried out, and relate to those of $R^{k}f_{\ast}\QBbb$ as follows. For concreteness, let us comment on the case of $f\colon \Xcal\to\DBbb$ as above, with normal crossings model $f^{\prime}\colon \Xcal^{\prime}\to\DBbb$. Denote by $j\colon\DBbb^{\times}\hookrightarrow\DBbb$ the open immersion. Then, the Deligne extension of $\Ecal=\EBbb\otimes\Ocal_{\DBbb^{\times}}$ equals $j_{\ast}\Ecal\cap\ {^\ell}R^{k}f_{\ast}\Omega_{\Xcal/\DBbb^{\times}}^{\bullet}$, or equivalently  $j_{\ast}\Ecal\cap R^{k}f_{\ast}^{\prime}\Omega_{\Xcal^{\prime}/\DBbb}^{\bullet}(\log)$, where the intersection is taken in $j_{\ast}R^{k}f_{\ast}\Omega_{\Xcal/\DBbb^{\times}}^{\bullet}$. Let us denote it by $^{\ell}\Ecal$. To construct the limiting mixed Hodge structure of $\EBbb$, we may first perform a ramified base change and suppose that $f^{\prime}$ is semi-stable. Secondly, we intersect the limiting mixed Hodge structure $(H^{k}_{\lim}, F^{\bullet}_{\infty}, W_{\bullet})$ of $R^{k}f_{\ast}\QBbb$ with $^{\ell}\Ecal(0)$, the fiber at $0$ of $^{\ell}\Ecal$. In our work, we will encounter this setting for the standard case of the primitive cohomology, but also for the decompositions \eqref{eq:decomposition-isotypical} ($G$-invariants) and the minimal decomposition \eqref{eq:iso-prim-coh-bis}. Accordingly, the resulting objects will be decorated with the symbols $\prim$, $G$ or ${\min}$. For example, we will have notations such as $R^{n-1}f_{\ast}^{\prime}\Omega_{\Xcal^{\prime}/\DBbb}^{\bullet}(\log)_{\min}$.

Analogously, for a projective normal crossings degeneration $f\colon \Xcal\to S$ between complex algebraic manifolds, with one-dimensional $S$, there are algebraic counterparts of all the above: logarithmic de Rham cohomology, Gauss--Manin connection, Hodge filtration, etc. This is compatible with the analytic theory after localizing to a holomorphic coordinate neighborhood of a given point $p\in S$. We will in particular speak of the limiting mixed Hodge structure at $p$, and simply write $H^{k}_{\lim}$ if there is no danger of confusion. 

Finally, we will also need the limiting mixed Hodge structure $(H_{k})_{\lim}$ on the homology, and in particular the dual weight filtration $W_\bullet^{\prime}$ defined as $W_{-r}^\prime = (H^{k}_{\lim}/W_{r-1})^\vee$. See \cite[(4.2.2)]{DeligneHodge2} or \cite[(3.1.3.1) and (3.2.2.7)]{ElZein} for more information about dual filtrations.

\subsection{Triviality of some variations of Hodge structures}\label{subsec:triviality-hodge-bundles}
We return to the geometric setting of Section \ref{sec:dworkandmirrorfamilies}, and maintain the notations therein. For the mirror family, we prove that outside of $(R^{n-1} f_{\ast} \QBbb)_{\min}$, all the variations of Hodge structures appearing in our work in fact correspond to trivial local systems. In particular, the local systems outside of the middle degree and the local system $\VBbb$ from Proposition \ref{prop:orthogonal-decompositions} are all trivial. We also derive consequences for the associated Hodge bundles, in the algebraic category.

We fix the normal crossings model $f^{\prime}\colon\Zcal^{\prime}\to\PBbb^{1}$ obtained by blowing up the locus of the ordinary double points of $f\colon\Zcal\to\PBbb^{1}$, which is defined over $\QBbb$. We also introduce a polarization, induced by a projective factorization of $f^{\prime}$ defined over $\QBbb$. The corresponding logarithmic Hodge bundles and their primitive parts are locally free sheaves over $\PBbb^{1}$, already defined in the algebraic category and over $\QBbb$.

By Lemma \ref{lemma:Hodge-numbers-crepant} we have $R^{d}f_{\ast}^{\prime}\Omega_{\Zcal^{\prime}/\PBbb^{1}}^{\bullet}(\log)=0$ for $d$ odd, not equal to $n-1$, while if $d = 2p \neq n-1$, $R^{d}f_{\ast}^{\prime}\Omega_{\Zcal^{\prime}/\PBbb^{1}}^{\bullet}(\log)=R^{p}f_{\ast}^{\prime}\Omega_{\Zcal^{\prime}/\PBbb^{1}}^{p}(\log)$. We then have the following result outside of middle degrees: 

\begin{lemma}\label{lemma:trivial-hodge}
For $2p\neq n-1$, the following hold:
\begin{enumerate}
    \item The local system $R^{2p}f_{\ast}\QBbb$ on $U^{\an}=\PBbb^{1}\setminus(\mu_{n+1}\cup\lbrace\infty\rbrace)$ is trivial. 
    \item\label{lemma:trivial-hodge-2} The Hodge bundle $R^{p}f_{\ast}^{\prime}\Omega_{\Zcal^{\prime}/\PBbb^{1}}^{p}(\log)$ is trivial in the algebraic category over $\QBbb$.
\end{enumerate}
\end{lemma}
\begin{proof}
We first prove that the local system $R^{2p}f_{\ast}\QBbb$ is trivial. Take a base point $b\in U^{\an}$, and let $\rho\colon \pi_{1}(U^{\an},b)\to\GL(H^{2p}(Z_{b},\QBbb))$ be the monodromy representation determining the local system. The fundamental group $\pi_{1}(U^{\an},b)$ is generated by loops $\gamma_{\xi}$ circling around $\xi\in\mu_{n+1}$, and a loop $\gamma_{\infty}$ circling around $\infty$, with a relation $\prod_{\xi}\gamma_{\xi}=\gamma_{\infty}$. Because the singularities of $\Zcal\to\PBbb^{1}$ at the points $\xi$ are ordinary double points, and $2p\neq n-1$, the local monodromies $\rho(\gamma_{\xi})$ are trivial. Therefore $\rho(\gamma_{\infty})$ is trivial as well, and so is $\rho$.

Now, the first claim implies the triviality of $R^{p}f_{\ast}^{\prime}\Omega_{\Zcal^{\prime}/\PBbb^{1}}^{p}(\log)=R^{2p}f_{\ast}^{\prime}\Omega_{\Zcal^{\prime}/\PBbb^{1}}^{\bullet}(\log)$ in the analytic category, since the latter realizes the Deligne extension of $R^{2p}f_{\ast}\Omega_{\Zcal/U}^{\bullet}$. By the GAGA principle, $R^{p}f_{\ast}^{\prime}\Omega_{\Zcal^{\prime}/\PBbb^{1}}^{p}(\log)$ is algebraically trivial as a complex vector bundle. This already implies the second claim. Indeed, let $E$ be a vector bundle over $\PBbb^{1}_{\QBbb}$, which is trivial after base change to $\CBbb$. Then the natural morphism $H^{0}(\PBbb^{1}_{\QBbb},E)\otimes\Ocal_{\PBbb^{1}_{\QBbb}}\to E$ is necessarily an isomorphism, since it is an isomorphism after a flat base change.
\end{proof}

\begin{lemma}\label{lemma:trivial-V-local}
With the same notations as in Proposition \ref{prop:orthogonal-decompositions}, we have:
\begin{enumerate}
    \item The local system $\VBbb$ on $U^{\an}$ is trivial. 
    \item The locally free coherent sheaf with connection $\Vcal$ over $U$ is trivial in the algebraic category over $\QBbb$.
\end{enumerate}
\end{lemma}
\begin{proof}
We first show that if the local monodromy of $(R^{n-1}f_\ast \CBbb)_{\min}$ around one ODP point is trivial, then it is so around all the ODP points.
Since $(R^{n-1}f_\ast \CBbb)_{\min}$ is isomorphic to $(R^{n-1}h_{\ast}\CBbb)^{G}_{\prim}$ as a local system, it is enough to show that the latter descends along the natural projection $(U-\{0\})\to (U-\{0\})/\mu_{n+1}$, where $\mu_{n+1}$ acts by multiplication on $(U-\{0\})\subset\PBbb^{1}$. Notice that for any $\zeta \in \mu_{n+1}$, the automorphism $\psi \mapsto \zeta \cdot \psi$ lifts to an automorphism of the family $g\colon\Ycal \to (U-\{0\})$, via the formula $[x_0:x_1:\ldots:x_n] \mapsto [x_0':x_1':\ldots:x_n']$, where $x_i'=x_i$ except for one $i$, for which $x_i' = \zeta^{-1} \cdot x_i$. Since we work in the quotient by the group $G$, all the choices of $i$ correspond to the same action. We conclude that the local systems $(R^{k}h_{\ast}\CBbb)^{G}\simeq R^{k}g_{\ast}\CBbb$ descend for all $k$. Observe that $R^{2}h_{\ast}\CBbb$ is actually constant with fiber $H^{2}(\PBbb^{n},\CBbb)$, by Lefschetz, with $G$ acting trivially. Therefore, the polarization necessarily descends. We conclude that $(R^{n-1}h_{\ast}\CBbb)^{G}_{\prim}=\ker \left(L\colon (R^{n-1}h_{\ast}\CBbb)^{G}\to (R^{n+1}h_{\ast}\CBbb)^{G}\right)$ descends too, as was to be shown.

We now show that $\VBbb$ is a trivial local system. It is enough to argue for $\VBbb_{\CBbb}$. In the odd dimensional case, there is nothing to prove. In the even dimensional case, we first recall that by the Picard--Lefschetz formula, the local monodromies on $(R^{n-1}f_{\ast}\CBbb)_{\prim}$ around the ODP points are semi-simple with a single non-trivial eigenvalue $-1$ of multiplicity one. It follows that around each ODP point, exactly one of the sub-local systems $(R^{n-1}f_\ast \CBbb)_{\min}$ and $\VBbb_{\CBbb}$ has trivial local monodromy.  By the argument in the previous lemma, if the monodromies around one hence all the ODP points on $(R^{n-1}f_\ast \CBbb)_{\min}$ were trivial, it would follow that the monodromy around $\infty$ would also be trivial. As this is excluded by Lemma \ref{lemma:MHS-infty} below, we infer that $\VBbb$ is a trivial local system.

We next address the triviality of $\Vcal$ asserted by the second point. We will now make use of the $G$-equivariant normal crossings model $h^{\prime}\colon\Xcal^{\prime}\to\PBbb^{1}$ of Proposition \ref{prop:NSB}. We summarize the current geometric setting in the following diagram, which builds upon \eqref{eq:diagram-mirror-families}:
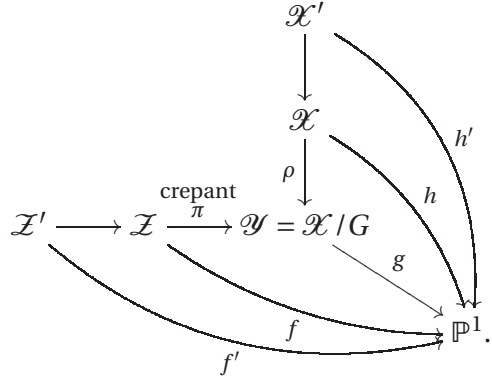
\begin{displaymath}
    \xymatrix{  &               &\Xcal^{\prime}\ar[d]\ar@/^2pc/[dddr]^{h'}        &\\
        &               &\Xcal\ar[d]_{\rho}\ar@/^1pc/[ddr]^{h}        &\\
     \Zcal^{\prime}\ar[r] \ar@/_2pc/[drrr]^{f'}   &\Zcal\ar[r]^{\hspace{-0.7cm}\substack{\text{crepant}\\ \pi}}\ar@/_1pc/[drr]^{f}          &\Ycal=\Xcal/G\ar[dr]^{g}         &\\
        &       &       &\PBbb^{1}.
    }
\end{displaymath}
We first argue analytically. By the minimal decomposition (Proposition \ref{prop:orthogonal-decompositions}), the morphism \eqref{eq:iso-prim-coh} induces a morphism between Deligne extensions
\begin{equation}\label{eq:morphism-Deligne-Vcal}
    ^{\ell}(R^{n-1}h_{\ast}\Omega_{\Xcal/U}^{\bullet})^{G}_{\prim}{\hookrightarrow}\ ^{\ell}(R^{n-1}f_{\ast}\Omega_{\Zcal/U}^{\bullet})_{\prim}.
\end{equation}
We can reformulate \eqref{eq:morphism-Deligne-Vcal} as a morphism
\begin{displaymath}
    \psi\colon(R^{n-1}h_{\ast}^{\prime}\Omega_{\Xcal^{\prime}/\PBbb^{1}}^{\bullet}(\log))^{G}_{\prim}{\hookrightarrow} (R^{n-1}f_{\ast}^{\prime}\Omega_{\Zcal^{\prime}/\PBbb^{1}}^{\bullet}(\log))_{\prim},
\end{displaymath}
 extending \eqref{eq:iso-prim-coh-bis-bis} to $\PBbb^{1}$. By the GAGA principle, this morphism is algebraic. Notice that the coherent sheaves involved in $\psi$ are locally free and defined over $\QBbb$. By Lemma \ref{prop:iso-prim-coh} \eqref{item:iso-prim-coh-2}, $\psi_{\mid U}$ is defined over $\QBbb$. Therefore, $\psi_{\mid U}$ is invariant under the action of $\Aut(\CBbb/\QBbb)$. Because $U$ is a non-empty Zariski open subset of $\PBbb^{1}$, which is an integral scheme, we infer that $\psi$ is invariant under $\Aut(\CBbb/\QBbb)$. Therefore, $\psi$ is defined over $\QBbb$, and so is its cokernel. We denote by $\widetilde{\Vcal}$ this cokernel of $\psi$ modulo its torsion part. Then $\widetilde{\Vcal}$ is a vector bundle.
 
 By Proposition \ref{prop:orthogonal-decompositions} \eqref{item:orth-dec-2}, $\widetilde{\Vcal}_{\mid U}$ is canonically isomorphic to $\Vcal$, over $\QBbb$, and in particular inherits a connection from $\Vcal$. Also, by the same proposition, we know that the analytification of \eqref{eq:iso-prim-coh-bis-alg} is canonically identified with the tensor product of \eqref{eq:iso-prim-coh-bis} with $\Ocal_{U^{\an}}$. By taking Deligne extensions, we deduce that $\widetilde{\Vcal}^{\an}$ is a vector bundle with regular singular connection, canonically isomorphic to the Deligne extension of $\VBbb\otimes\Ocal_{U^{\an}}$. By the first part of the lemma, we thus infer that $\widetilde{\Vcal}^{\an}$ is a trivial vector bundle with connection, and in particular any trivialization over $\PBbb^{1}$ is flat. As in the proof of Lemma \ref{lemma:trivial-hodge} \eqref{lemma:trivial-hodge-2}, we deduce that $\widetilde{\Vcal}$ is a trivial vector bundle over $\PBbb^{1}$, defined over $\QBbb$. From all the above, we conclude that the restriction to $U$ of any trivialization of $\widetilde{\Vcal}$, defined over $\QBbb$, induces a flat trivialization of $\Vcal\simeq\widetilde{\Vcal}_{\mid U}$, defined over $\QBbb$. This concludes the proof.
 \end{proof}

\subsection{Behaviour of $\eta_{k}$ at the MUM point}
For the mirror family $f\colon\Zcal\to\PBbb^{1}$, let $\DBbb_{\infty}$ be a holomorphic disc neighborhood at infinity, with parameter $t=1/\psi$. To lighten notations, we still denote by $f\colon\Zcal\to\DBbb_{\infty}$ the restricted family. To simplify notation, we write $H^{n-1}_{\lim}$ for the limiting mixed Hodge structure at infinity of $(R^{n-1}f_{\ast}\QBbb)_{\min}$. 
\begin{lemma}\label{lemma:MHS-infty}
\begin{enumerate}
    \item The monodromy $T$ of $(R^{n-1}f_{\ast}\QBbb)_{\min}$ at $\infty$ is maximally unipotent. In particular, the nilpotent operator $N$ on $H^{n-1}_{\lim}$ satisfies $N^{n-1}\neq 0$. 
    \item The graded pieces $\Gr^{W}_{k}H^{n-1}_{\lim}$ are one-dimensional if $k$ is even, and trivial otherwise. For all $1\leq k\leq n-1$, $N$ induces isomorphisms 
    \begin{displaymath}
       \Gr^{W}_{k} N\colon\Gr^{W}_{k}H^{n-1}_{\lim}\overset{\sim}{\longrightarrow} \Gr^{W}_{k-2}H^{n-1}_{\lim}.
    \end{displaymath}
    \item\label{lemma:MHS-infty-3} For all $1\leq p\leq n-1$, $N$ induces isomorphisms
    \begin{displaymath}
        \Gr_{F_{\infty}}^{p} N\colon \Gr_{F_{\infty}}^{p}H^{n-1}_{\lim}\overset{\sim}{\longrightarrow} \Gr_{F_{\infty}}^{p-1}H^{n-1}_{\lim}.
    \end{displaymath}
\end{enumerate}
\end{lemma}
\begin{proof}
The maximally unipotent property for $(R^{n-1}f_{\ast}\QBbb)_{\min}\simeq (R^{n-1}h_{\ast}\QBbb)^{G}_{\prim}$ is proven in odd relative dimension in \cite[Cor. 1.7]{HSBT}. Exactly the same argument as in \emph{loc. cit.} yields the claim in even relative dimension.
In particular $N^{n-1}\neq 0$. This settles the first point.  Because moreover $N^{n-1}$ induces an isomorphism $\Gr^{W}_{2(n-1)}H^{n-1}_{\lim}\overset{\sim}{\to} \Gr^{W}_{0}H^{n-1}_{\lim}$ we deduce that $\Gr^{W}_{0}H^{n-1}_{\lim}\neq 0$. By Lemma \ref{lemma:thetakconstruction}, $H^{n-1}_{\lim}$ is $n$-dimensional, and the second item follows for dimension reasons. Finally, we use that $\Gr_{F_{\infty}}^{p}H^{n-1}_{\lim}$ is one-dimensional again by Lemma \ref{lemma:thetakconstruction} and then necessarily $\Gr_{F_{\infty}}^{p}H^{n-1}_{\lim}=\Gr_{F_{\infty}}^{p}\Gr^{W}_{2p}H^{n-1}_{\lim}=\Gr^{W}_{2p}H^{n-1}_{\lim}$. Hence the second point implies the third one.
\end{proof}

By the maximally unipotent monodromy and for dimension reasons, the $T$-invariant classes of the minimal cohomology of a general fiber span a rank one trivial subsystem of $(R^{n-1}f_{\ast}\CBbb)_{\min}$ on $\DBbb^{\times}_{\infty}$. We fix a basis $\gamma^{\prime}$ of this trivial system. It extends to a nowhere vanishing holomorphic section of the Deligne extension of $(R^{n-1}f_{\ast}\CBbb)_{\min}$. The fiber at $0$ is then a basis for $W_{0}$, which identifies with $\ker N$ by the above lemma. We still write $\gamma^{\prime}$ for this limit element. Similarly, $(R^{n-1}f_{\ast}\CBbb)_{\min}^{\vee}$ has a rank one trivial subsystem, spanned by the class of a $T$-invariant homological cycle $\gamma$. We may choose $\gamma$ to correspond to $\gamma^{\prime}$ by Poincar\'e duality.\footnote{Recall, from Proposition \ref{prop:orthogonal-decompositions} and Remark \ref{rmk:homology-min}, that classes in $(R^{n-1}f_{\ast}\CBbb)_{\min}^{\vee}$ can be seen as homological cycles, and Poincar\'e duality can be used on the minimal component.} Hence, for any $\eta\in H^{n-1}(Z_t)$, $t\in\DBbb^{\times}_{\infty}$, the period $\langle\gamma,\eta\rangle$ equals the intersection pairing $Q(\gamma^{\prime},\eta)$. It is possible to explicitly construct an invariant cycle. Although we will need this in a moment, we postpone the discussion to \textsection \ref{subsubsec:periods}, where a broader study of homological cycles is delivered. 

In preparation for the following lemma, we recall from the preliminaries in \textsection \ref{subsec:generalities-hodge} that $f_{\ast}K_{\Zcal/\DBbb_{\infty}}(\log)$ is isomorphic to $\Fcal^{n-1}R^{n-1}\Omega^{\bullet}_{\Zcal/\DBbb_{\infty}}(\log)$, and that $R^{n-1}\Omega^{\bullet}_{\Zcal/\DBbb_{\infty}}(\log)$ realizes the Deligne extension of $(R^{n-1}f_{\ast}\CBbb)\otimes\Ocal_{\DBbb_{\infty}^{\times}}$ to $\DBbb_{\infty}$.
\begin{lemma}\label{lemma:non-vanishing-period}
Let $\eta$ be a holomorphic trivialization of $f_{\ast}K_{\Zcal/\DBbb_{\infty}}(\log)$. Then the period $\langle\gamma,\eta\rangle$ defines a holomorphic function on $\DBbb_{\infty}$, non-vanishing at the origin. 
\end{lemma}
\begin{proof}
The argument is well-known, see \emph{e.g.} \cite[Prop.]{morrison-mirror} and \cite[Lemma 3.10]{VoisinMirror}, but we sketch it due to its relevance. 

The pairing $\langle\gamma,\eta\rangle=Q(\gamma^{\prime},\eta)$ is clearly a holomorphic function on $\DBbb^{\times}_{\infty}$, since both $\gamma^{\prime}$ and $\eta$ are holomorphic sections of $(R^{n-1}f_{\ast}\CBbb)\otimes\Ocal_{\DBbb^{\times}_{\infty}}$. Moreover, they are both global sections of the Deligne extension. This ensures that $|Q(\gamma^{\prime},\eta)|$ has at most a logarithmic singularity at $0$. It follows that $Q(\gamma^{\prime},\eta)$ is actually a holomorphic function. 

For the non-vanishing property, we make use of the interplay between the intersection pairing seen on $H^{n-1}_{\lim}$ and the monodromy weight filtration \cite[Lemma 6.4]{schmid}, together with Lemma \ref{lemma:MHS-infty}. Let $\eta^{\prime}\in H^{n-1}_{\lim}$ be the fiber of $\eta$ at $0$. We need to show that $Q(\gamma^{\prime},\eta^{\prime})\neq 0$. Suppose the contrary. Since $\gamma^{\prime}$ is a basis of $W_{0}=\ker N=\Imag N^{n-1}$, we have $\eta^{\prime}\in (\Imag N^{n-1})^{\perp}$. The intersection pairing is non-degenerate and satisfies $Q(Nx,y)+Q(x,Ny)=0$. Therefore, we find that $\eta^{\prime}\in (\Imag N^{n-1})^{\perp}=\ker N^{n-1}=W_{2n-3}$. But $\eta^{\prime}$ is a basis of $F^{n-1}H^{n-1}_{\lim}=F^{n-1}\Gr^{W}_{2n-2}H^{n-1}_{\lim}$, and therefore $\eta^{\prime}\not\in W_{2n-3}$. We thus have reached a contradiction.
\end{proof}

Before the next theorem, we consider the logarithmic extension of the Kodaira--Spencer maps \eqref{def:KS}: if $D$ is the divisor $[\infty]+\sum_{\xi^{n+1}=1}[\xi]$, then 
\begin{equation}\label{eq:KS-log}
    \KS^{(q)}\colon T_{\PBbb^{1}}(-\log D)\longrightarrow\Hom_{\Ocal_{\PBbb^{1}}}(R^{q}f_{\ast}\Omega^{n-1-q}_{\Zcal^{\prime}/\PBbb^{1}}(\log), R^{q+1}f_{\ast}\Omega^{n-2-q}_{\Zcal^{\prime}/\PBbb^{1}}(\log)).
\end{equation}
They preserve the minimal and primitive components. 
\begin{theorem}\label{thm:eta-triv-infty}
The section $\eta_{k}$ is a holomorphic trivialization of $R^{k}f_{\ast}\Omega^{n-1-k}_{\Zcal/\DBbb_{\infty}}(\log)_{\min}$.
\end{theorem}
\begin{proof}
First of all, we prove that $\eta_{0}$ is a meromorphic section of $f_{\ast}K_{\Zcal/\DBbb_{\infty}}(\log)$. Indeed, $\eta_{0}$ is an algebraic section of $f_{\ast}K_{\Zcal/U}$ (see Lemma \ref{lemma:recurrence-eta}), hence a rational section of $f_{\ast}K_{\Zcal^{\prime}/\PBbb^{1}}(\log)$ and thus a meromorphic section of $f_{\ast}K_{\Zcal/\DBbb_{\infty}}(\log)$.

Second, we establish the claim of the theorem for $\eta_{0}$. By Lemma \ref{lemma:non-vanishing-period}, we need to show that the holomorphic function $\langle\gamma,\eta_{0}\rangle$ on $\DBbb^{\times}_{\infty}$ extends holomorphically to $\DBbb_{\infty}$, and does not vanish at the origin. This property can be checked by a standard explicit computation reproduced below \eqref{eq:period-I0}. 

Finally, for the sections $\eta_{k}$, we use the recurrence \eqref{eq:recurrence-eta-norm} and the logarithmic extension of the Kodaira--Spencer maps \eqref{eq:KS-log}. It follows that the sections $\eta_{k}$ are global sections of the sheaves $R^{k}f_{\ast}\Omega^{n-1-k}_{\Zcal/\DBbb_{\infty}}(\log)_{\min}$. Let us denote by $\eta^{\prime}_{k}$ the fiber at 0 of the sections $\eta_{k}$. Specializing \eqref{eq:recurrence-eta-norm} at $0$, we find $(\Gr_{F_{\infty}}^{n-1-k} N)\eta_{k}^{\prime}=\eta_{k+1}^{\prime}$. By Lemma \ref{lemma:MHS-infty} \eqref{lemma:MHS-infty-3} and because $\eta_{0}^{\prime}\neq 0$, we see that $\eta_{k}^{\prime}\neq 0$ for all $k$. This concludes the proof.
\end{proof}

\subsection{Behaviour of $\eta_{k}$ at the ODP points}\label{subsec:etak-odp}
Recall the normal crossings model $f^{\prime}\colon\Zcal^{\prime}\to\PBbb^{1}$. We restrict it to a disc neighborhood $\DBbb_{\xi}$ of some $\xi\in\mu_{n+1}$. Concretely, we fix the coordinate $t=\psi-\xi$. We  write $f^{\prime}\colon\Zcal^{\prime}\to\DBbb_{\xi}$ for the restricted family. We now deal with the limiting mixed Hodge structure $H^{n-1}_{\lim}$ at $\xi$ of $(R^{n-1} f_\ast \QBbb)_{\min}$. Since the monodromy around $\xi$ is not unipotent in general, the construction of $H^{n-1}_{\lim}$ requires a preliminary semi-stable reduction. This can be achieved as follows:
\begin{equation}\label{eq:semi-stable-reduction}
    \xymatrix{
        &\widetilde{\Zcal}\ar[rr]^{\text{normalization}}\ar[drr]_{\widetilde{f}}    &   &\Zcal^{\prime\prime}\ar[r]^{r}\ar[d]\ar@{}[rd]|\square           &\Zcal^{\prime}\ar[d]^{f^{\prime}}\\
       &         &   &\DBbb_{\xi}\ar[r]_{\rho(u)=u^{2}=t}                  &\DBbb_{\xi}\\
    }
\end{equation}
Hence $\widetilde{f}: \widetilde{\Zcal} \to \DBbb_{\xi}$ is the normalized base change of $f'$ by $\rho$. An explicit computation in local coordinates shows it is indeed semi-stable. The special fiber $\widetilde{f}^{-1}(0)$ consists of two components intersecting transversally. One is the strict transform $\widetilde{Z}$ of $Z_\xi$. We denote by $E$ the other component. Then $E$ is a non-singular quadric of dimension $n-1$, and $\widetilde{Z} \cap E$ is a non-singular quadric of dimension $n-2$. In terms of this data, the monodromy weight filtration is computed as follows. 
\begin{lemma}\label{lemma:MHS-xi}
 The graded pieces of the weight filtration on $H^{n-1}_{\lim}$ are given by:
\begin{itemize}
    \item if $n-1$ is odd, then
    \begin{displaymath}
        \Gr^{W}_{k}H^{n-1}_{\lim}=\begin{cases}
            \QBbb\left(-\frac{n-2}{2}\right),    &\text{if } k=n-2,\\
            \text{a direct factor of $H^{n-1}(\widetilde{Z})$}  , &\text{if } k=n-1,\\
            \QBbb\left(-\frac{n}{2}\right),    &\text{if } k=n,\\
            0,  &\text{otherwise}.
        \end{cases}
    \end{displaymath}
    \item if $n-1$ is even, then 
    \begin{displaymath}
        \Gr^{W}_{k}H^{n-1}_{\lim}=\begin{cases}
             \text{a direct factor of } &H\left(H^{n-3}(\widetilde{Z}\cap E)(-1) \to H^{n-1}(\widetilde{Z}) \oplus H^{n-1}(E) \to H^{n-1}(\widetilde{Z}\cap E)\right), \\
              &\hspace{7.5cm}\text{if } k=n-1,\\
              & \\
            0, &\hspace{7.5cm}\text{if } k\neq n-1.
        \end{cases}
    \end{displaymath}
    Hence, $H^{n-1}_{\lim}$ is a pure Hodge structure of weight $n-1$.
\end{itemize}
\end{lemma}
\begin{proof}
The proof follows from  \cite[Ex. 2.15]{Steenbrink-mixedonvanishing}, noticing that $(R^{n-1}\widetilde{f}_{\ast}\QBbb)_{\min}=\rho^{\ast}(R^{n-1}f_{\ast}\QBbb)_{\min}$ is a direct factor of $R^{n-1}\widetilde{f}_{\ast}\QBbb$, whose complement is a trivial variation of Hodge structures by Lemma \ref{lemma:trivial-hodge} and Lemma \ref{lemma:trivial-V-local}. For the case $n-1$ is even, we moreover recall that  $\VBbb$ as in Proposition \ref{prop:orthogonal-decompositions} has pure bidegree $((n-1)/2,(n-1)/2)$.
\end{proof}

We will need the comparison of the middle degree minimal Hodge bundles between before and after semi-stable reduction. We follow \cite[Sec. 2 \& Prop. 3.10]{cdg2}. There are natural morphisms 
\begin{equation}\label{eq:injection-hodge-bundles}
    \varphi^{p,q}\colon\rho^{\ast} R^{q}f^{\prime}_{\ast}\Omega^{p}_{\Zcal^{\prime}/\DBbb_{\xi}}(\log)_{\min}\hookrightarrow R^{q}\widetilde{f}_{\ast}\Omega^{p}_{\widetilde{\Zcal}/\DBbb_{\xi}}(\log)_{\min}.
\end{equation}
\begin{lemma}\label{lemma:cokernel}
Suppose that $p+q=n-1$. Let $Q^{p,q}$ be the cokernel of $\varphi^{p,q}$ in \eqref{eq:injection-hodge-bundles}. 
\begin{itemize}
    \item If $p\neq q$, then $Q^{p,q}=0$.
    \item If $p=q=\frac{n-1}{2}$, then $Q^{p,p}=\Ocal_{\DBbb_{\xi},0}/u\Ocal_{\DBbb_{\xi},0}$.
\end{itemize}
\end{lemma}
\begin{proof}
The results in \cite[Sec. 2 \& Prop. 3.10]{cdg2} are explicitly stated for the whole Hodge bundles, and describe the cokernels in terms of the semi-simple part of the monodromy acting on the limiting Hodge structure. For their minimal components, see however Remark 2.7 (iii) in \emph{loc. cit.}, together with Proposition \ref{prop:orthogonal-decompositions} and  Lemma \ref{lemma:trivial-V-local}.
\end{proof}

We are now fully equipped for the proof of:

\begin{theorem}\label{thm:order-eta-xi}
The sections $\eta_{k}$ extend to meromorphic sections of the logarithmic Hodge bundles $R^{k}f^{\prime}_{\ast}\Omega^{n-1-k}_{\Zcal^{\prime}/\ABbb^{1}}(\log)_{\min}$. Furthermore, denote by $\ord_{\xi}\eta_{k}$ the order of zero or pole of $\eta_{k}$ at $\xi$, as a rational section of $R^{k}f^{\prime}_{\ast}\Omega^{n-1-k}_{\Zcal^{\prime}/\ABbb^{1}}(\log)_{\min}$.
\begin{itemize}
        \item[\textbullet] If $n-1$ is odd, then $\ord_{\xi}\eta_{k} = 0 $ for $k\leq n/2-1$ and $\ord_{\xi}\eta_{k} = -1 $ otherwise.
        \item[\textbullet] If $n-1$ is even, then $\ord_{\xi}\eta_{k}=0$ for $k\leq\frac{n-3}{2}$ and $\ord_{\xi}\eta_{k} = -1 $ otherwise.
    \end{itemize}
\end{theorem}
\begin{proof}
Throughout the proof, we write $\Xcal$, $\Ycal$ and $\Zcal$ for the respective total spaces over $\ABbb^{1}$. We begin by showing that $\eta_{0}$ extends to a global section of $f^{\prime}_{\ast}K_{\Zcal^{\prime}/\ABbb^{1}}(\log)$, non-vanishing at $\xi$. Since the singular fibers of $\Zcal\to\ABbb^{1}$ present only ordinary double points, there is an equality
\begin{displaymath}
    f_{\ast}K_{\Zcal/\ABbb^{1}}=f^{\prime}_{\ast}K_{\Zcal^{\prime}/\ABbb^{1}}(\log).
\end{displaymath}
This can be seen as the coincidence of the upper and lower extensions of $f_{\ast}K_{\Zcal/U}$ to $\ABbb^{1}$ (apply \cite[Cor. 2.8 \& Prop. 2.10]{cdg2} and the Picard--Lefschetz formula for the monodromy). Since $\Ycal$ has rational singularities (cf. Lemma \ref{lemma:gorenstein}), the natural morphism $g_{\ast}K_{\Ycal/\ABbb^{1}}\to f_{\ast}K_{\Zcal/\ABbb^{1}}$ is an isomorphism. Also $g_{\ast}K_{\Ycal/\ABbb^{1}}=(h_{\ast}K_{\Xcal/\ABbb^{1}})^{G}$. Indeed, let $\Xcal^{\circ}$ be the complement of the fixed point locus of $G$ in $\Xcal$ and similarly for $\Ycal^{\circ}$, so that $\Ycal\setminus\Ycal^{\circ}$ has codimension $\geq 2$. Then, because $\Ycal$ is normal Gorenstein and $\Ycal^{\circ}=\Xcal^{\circ}/G$ is an \'etale quotient, and $\Xcal$ is non-singular, we find
\begin{displaymath}
    g_{\ast}K_{\Ycal/\ABbb^{1}}=g_{\ast}K_{\Ycal^{\circ}/\ABbb^{1}}=(h_{\ast}K_{\Xcal^{\circ}/\ABbb^{1}})^{G}
    =(h_{\ast}K_{\Xcal/\ABbb^{1}})^{G}.
\end{displaymath}
By construction of $\eta_{0}$ (cf. Definition \ref{def:eta-k}), it is enough to prove that $\theta_{0}$ defines a trivialization of $h_{\ast}K_{\Xcal/\ABbb^{1}}$ around $\xi$.
Denote by $\Xcal^{\ast}$ the complement in $\Xcal$ of the ordinary double points, so that $\Xcal\setminus\Xcal^{\ast}$ has codimension $\geq 2$. Because $\Xcal$ is non-singular, we have $h_{\ast}K_{\Xcal/\ABbb^{1}}=h_{\ast}K_{\Xcal^{\ast}/\ABbb^{1}}$. Now, the expression \eqref{eq:relative-volume-form} for $\theta_{0}$ defines a relative holomorphic volume form on the whole $\Xcal^{\ast}$, and hence a trivialization of $h_{\ast}K_{\Xcal^{\ast}/\ABbb^{1}}$ as desired. 

That the sections $\eta_{k}$ define meromorphic sections of the sheaves $R^{k}f^{\prime}_{\ast}\Omega^{n-1-k}_{\Zcal^{\prime}/\ABbb^{1}}(\log)_{{\min}}$ follows from the corresponding property for $\eta_{0}$, plus the recurrence \eqref{eq:recurrence-eta-norm} and the existence of the logarithmic extension of the Kodaira--Spencer maps \eqref{eq:KS-log}. From the same recurrence, we reduce the computation of $\ord_{\xi}\eta_{k}$ to the computation of the orders at $\xi$ of the rational morphisms $\KS^{(j)}(\psi d/d\psi)$, with respect to the logarithmic extension of the Hodge bundles:
\begin{equation*}\label{eq:ordetak}
         \ord_{\xi} \eta_k =\ord_{\xi} \eta_0 +  \sum_{j=0}^{k-1} \ord_{\xi}\KS^{(j)}\left(\psi \frac{d}{d\psi}\right)
         =\sum_{j=0}^{k-1} \ord_{\xi}\KS^{(j)}\left(\psi \frac{d}{d\psi}\right).
\end{equation*}
Let us define $M^{(j)}=\ord_{\xi}\KS^{(j)}\left(\psi \frac{d}{d\psi}\right)$. Because $\eta_{0}$ trivializes $f_{\ast}K_{\Zcal/\ABbb^{1}}$ at $\xi$, formula \eqref{eq:yukawa} shows that
\begin{equation}\label{eq:sumofMJ}
    \sum_{j=0}^{n-2}M^{(j)}=\ord_{\xi}Y(\psi)=-1.
\end{equation}
We argue that all but one of the $M^{(j)}$ are  in fact zero. For this, we relate $M^{(j)}$ to the action of the nilpotent operator $N$ on the limiting mixed Hodge structure at $\xi$. Recall we defined the coordinate $t=\psi-\xi$ on a disc neighborhood $\DBbb_{\xi}$ of $\xi$. The first observation is
\begin{equation}\label{eq:Mj>-1}
    \ord_{t=0}\KS^{(j)}\left(t\frac{d}{dt}\right)=\ord_{\xi}\KS^{(j)}\left((\psi-\xi)\frac{d}{d\psi}\right)=M^{(j)}+1\geq 0,
\end{equation}
since the Kodaira--Spencer maps along logarithmic tangent vectors preserve the logarithmic Hodge bundles (cf. \eqref{eq:KS-log}). Hence, we see that $M^{(j)}\geq -1$. We now need to distinguish two cases, depending on the parity of $n-1$.\\

\noindent\emph{Odd case}: If $n-1$ is odd, then the monodromy is unipotent and the fiber of $\KS^{(p)}(td/dt)$ at $t=0$ is already $\Gr^{p}_{F_\infty}N\colon \Gr_{F_\infty}^p H^{n-1}_{\lim}\to \Gr_{F_\infty}^{p-1} H^{n-1}_{\lim}$. From Lemma \ref{lemma:MHS-xi}, we deduce that unless $p=n/2$, $\Gr^p_{F_\infty}N=0$ so that $\ord_{t=0} \KS^{(p)}(td/dt) > 0$ and hence $M^{(p)} \geq 0.$ By \eqref{eq:sumofMJ} we necessarily have $M^{(n/2)}=-1$ and the other $M^{(j)}=0$.\\

\noindent\emph{Even case}: If $n-1$ is even, the nilpotent operator $N$ is in fact trivial, but the monodromy is no longer unipotent. The construction of the limiting mixed Hodge structure thus involves a semi-stable reduction. Choose a square root $u$ of $t$ as in \eqref{eq:semi-stable-reduction}. Then since $ u \frac{d}{du} = 2 t \frac{d}{dt}$ we get, comparing the Gauss-Manin connection before and after semi-stable reduction, a commutative diagram of maps of line bundles 
\begin{displaymath}
    \xymatrix{\rho^\ast R^{q}f^{\prime}_{\ast}\Omega^{p}_{\Zcal^{\prime}/\DBbb_{\xi}}(\log)_{{\min}} \ar[r]^-{\KS^{(q)} \left( u \frac{d}{du}\right) }\ar[d]^{\varphi^{p,q}} &\rho^\ast R^{q+1}f^{\prime}_{\ast}\Omega^{p-1}_{\Zcal^{\prime}/\DBbb_{\xi}}(\log)_{{\min}} \ar[d]^{\varphi^{p-1, q+1}} \\ 
    R^{q}\widetilde{f}_{\ast}\Omega^{p}_{\widetilde{\Zcal}/\DBbb_{\xi}}(\log)_{{\min}} \ar[r]^-{\KS^{(q)} \left( 2 t \frac{d}{dt}\right) } &  R^{q+1}\widetilde{f}_{\ast}\Omega^{p-1}_{\widetilde{\Zcal}/\DBbb_{\xi}}(\log)_{{\min}}.  }
\end{displaymath}
Together with $\ord_{u=0}= 2 \ord_{t=0}$, we conclude that:

\begin{equation}\label{eq:KSsemistable}
    \ord_{u=0} \varphi^{p,q}+\ord_{u=0} \KS^{(q)}\left(u \frac{d}{du}\right) = \ord_{u=0} (\varphi^{p-1,q+1})+2 \ord_{t=0}  \KS^{(q)}\left(t \frac{d}{dt}\right).
\end{equation}
By Lemma \ref{lemma:cokernel},  $\ord_{u=0} (\varphi^{p,q})=0$ except for the case $(p,q)=((n-1)/2, (n-1)/2),$ where in fact $\ord_{u=0} (\varphi^{p,q})=1$. From \eqref{eq:KSsemistable} we then conclude that 
\begin{equation}\label{eq:KSsemistable1}
     \ord_{u=0} \KS^{((n-3)/2)}\left(u \frac{d}{du}\right) = 1+2 \ord_{t=0}  \KS^{((n-3)/2)}\left(t \frac{d}{dt}\right) 
\end{equation}
\begin{equation}\label{eq:KSsemistable2}
     1+\ord_{u=0} \KS^{((n-1)/2)}\left(u \frac{d}{du}\right) =2 \ord_{t=0}  \KS^{((n-1)/2)}\left(t \frac{d}{dt}\right).
\end{equation}
In both cases \eqref{eq:KSsemistable1}--\eqref{eq:KSsemistable2} the order of vanishing of Kodaira--Spencer along the vector field $u\frac{d}{du}$ is strictly positive, since the restriction to 0 is the nilpotent operator $N=0$.  It follows that 
\begin{displaymath}
 \ord_{t=0}  \KS^{((n-3)/2)}\left(t \frac{d}{dt}\right) \geq 0,\quad \text{\emph{i.e.}}\quad M^{((n-3)/2)} \geq -1,
\end{displaymath}
and
\begin{displaymath}
 \ord_{t=0}  \KS^{((n-1)/2)}\left(t \frac{d}{dt}\right) \geq 1, \quad \text{\emph{i.e.}}\quad M^{((n-1)/2)} \geq 0.
\end{displaymath}
Since all other $M^{(j)} \geq 0$ as in the odd case, we conclude from \eqref{eq:sumofMJ}  that all these inequalities are in fact equalities.
\end{proof}

\section{The BCOV invariant of the mirror family}
In this section we prove the first part of the Main Theorem in the introduction, to the effect that the BCOV invariant of the mirror family encapsulates the Gromov--Witten invariants of a general Calabi--Yau hypersurface. The proof proceeds by applying the arithmetic Riemann--Roch theorem as in Section \ref{sec:BCOV-ARR}, by choosing the algebraic trivializations of the Hodge bundles studied in Section \ref{sec:dworkandmirrorfamilies}. This is then worked out in terms of canonical sections of the Hodge bundles, whose existence is tied to the limiting Hodge structure $H^{n-1}_{\lim}$ at the MUM point. In the process, a transcendental expression built out of periods arises, which matches Zinger's formula for the sought generating function of Gromov--Witten invariants. 

\subsection{The Kronecker limit formula for the mirror family}\label{subsec:Kronecker-mirror} 
For the mirror family $f\colon\Zcal\to U$, we proceed to prove an expression for the BCOV invariant $\tau_{\bcov}(Z_{\psi})$ in terms of the $L^{2}$ norms of the sections $\eta_{k}$ (cf. Definition \ref{def:eta-k}). The strategy follows the same lines as for families of Calabi--Yau hypersurfaces \textsection \ref{subsec:Kronecker-hyp}. 

We fix a polarization and a projective factorization of $f$, defined over $\QBbb$. We denote by $L$ the corresponding \emph{algebraic} Lefschetz operator, that is the cup-product against the algebraic cycle class of a hyperplane section. We will abusively confound $L$ with the algebraic cycle class of a hyperplane section. With this choice of $L$, the primitive decomposition of the Hodge bundles $R^{p}f_{\ast}\Omega_{\Zcal/U}^{q}$ holds over $\QBbb$. Let $h$ be a K\"ahler metric and $\omega$ the K\"ahler form normalized as in \eqref{eq:normal-omega}, and assume that the fiberwise cohomology class is in the topological hyperplane class. Hence, under the correspondence between algebraic and topological cycle classes, $L$ is sent to $(2\pi i)[\omega]\in R^{2}f_{\ast}\QBbb(1)$.

Below, all the $L^{2}$ norms are computed with respect to $\omega$ as in \eqref{eq:normal-L2}.

\begin{theorem}\label{thm:pre-formula}
There exists a real positive constant $C\in\pi^{c}\ov{\QBbb}^{\times}$ such that
\begin{displaymath}
    \taubcov{Z_{\psi}}=C\left|\frac{(\psi^{n+1})^{a}}{(1-\psi^{n+1})^{b}}\right|^{2}
    \frac{\|\eta_{0}\|^{\chi/6}_{\l2}}{\left(\prod_{k=0}^{n-1}\|\eta_{k}\|_{\l2}^{2(n-1-k)}\right)^{(-1)^{n-1}}}
\end{displaymath}
where $\chi=\chi(Z_{\psi})$ and
\begin{eqnarray*}
    a & = & (-1)^{n-1}\frac{n(n-1)}{6}-\frac{\chi}{12(n+1)},\\
    b & = & (-1)^{n-1}\frac{n (3n-5)}{24}\\
    c & = &\frac{1}{2}\sum_{k}(-1)^{k+1}k^{2}b_{k}.
\end{eqnarray*}
\end{theorem}

\begin{proof}
We apply the version of the arithmetic Riemann--Roch theorem formulated in Theorem \ref{thm:ARR}, to the family $f\colon\Zcal\to U$ as being defined over $\QBbb$.\bigskip

\noindent\emph{Choices of sections.} We need to specify the section $\eta$ and the sections $\eta_{p,q}$ in equation \eqref{eq:bcov-ARR}. The section $\eta$ is chosen to be $\eta_0$ as defined in Definition~\ref{def:eta-k}. We next describe our choices of $\eta_{p,q}$:
\begin{itemize}
    \item If $p+q \neq n-1$ and $p\neq q$, then the corresponding Hodge bundle vanishes by Lemma \ref{lemma:Hodge-numbers-crepant}, and thus gives no contribution.
    \item  For $2p\neq n-1$, Lemma \ref{lemma:trivial-hodge} guarantees that $\det R^{p}f_{\ast}^{\prime}\Omega_{\Zcal^{\prime}/\PBbb^{1}}^{p}(\log)=\det R^{2p}f_{\ast}^{\prime}\Omega_{\Zcal^{\prime}/\PBbb^{1}}^{\bullet}(\log)$ is trivial, in the algebraic category over $\QBbb$, and any trivialization is flat for the Gauss--Manin. We choose $\eta_{p,p}$ to be any trivialization defined over $\QBbb$, and then restrict it to $U$. Notice that the $L^{2}$ norm  $\|\eta_{p,p}\|_{\l2}$ is then constant.
    \item For $p+q=n-1$ and $p\neq q$, the $(p,q)$ Hodge bundle is primitive and has rank one. Then we take $\eta_{p,q}=\eta_{\quillen}$ in Definition \ref{def:eta-k}. By Lemma \ref{lemma:recurrence-eta}, $\eta_{\quillen}$ is defined over $\QBbb$. 
    \item For $p+q=n-1$ and $p=q$, which can only occur when $n-1$ is even, the $(p,q)$ Hodge bundle is no longer primitive of rank one. We first employ the algebraic primitive decomposition, and then the minimal decomposition of Proposition \ref{prop:orthogonal-decompositions} \eqref{item:orth-dec-2}: 
\begin{equation}\label{eq:prim-dec-for-eta-p}
    \begin{split}
    \det R^{p}f_{\ast}\Omega^{p}_{\Zcal/U}=&\det (R^{p}f_{\ast}\Omega^{p}_{\Zcal/U})_{\prim}\otimes\det L R^{p-1}f_{\ast}\Omega^{p-1}_{\Zcal/U}\\
    \simeq &\det (R^{p}f_{\ast}\Omega^{p}_{\Zcal/U})_{\prim}\otimes\det  R^{p-1}f_{\ast}\Omega^{p-1}_{\Zcal/U} \\
    \simeq & \det (R^{p}f_{\ast}\Omega^{p}_{\Zcal/U})_{\min} \otimes \det \Vcal \otimes\det  R^{p-1}f_{\ast}\Omega^{p-1}_{\Zcal/U}. 
    \end{split}
\end{equation}
We define  $\eta_{\frac{n-1}{2},\frac{n-1}{2}}$ as the element corresponding to $\eta_\frac{n-1}{2} \otimes v \otimes  \eta_{\frac{n-3}{2},\frac{n-3}{2}}$ under this isomorphism, for any algebraic flat trivialization $v \in \det \Vcal$, defined over $\QBbb$, provided by Lemma \ref{lemma:trivial-V-local}. We claim that
\begin{equation}\label{eq:je-manque-didees}
   \|\eta_{\frac{n-1}{2},\frac{n-1}{2}}\|_{\l2}^{2}\sim_{\QBbb^{\times}} \|\eta_\frac{n-1}{2}\|_{\l2}^{2} \|v\|_{\l2}^{2}\|\eta_{\frac{n-3}{2},\frac{n-3}{2}}\|_{\l2}^{2},
\end{equation}
where $\sim_{\QBbb^{\times}}$ denotes equality up to a rational number. For this, we bring together several facts. The first one is that the Lefschetz decomposition is orthogonal for the $L^{2}$ metrics, regardless of the normalization of the K\"ahler forms. The second one is that the algebraic cycle class of $L$ corresponds to $(2\pi i)[\omega]$ in analytic de Rham cohomology. The third fact is that the operator $[2\pi\omega]\wedge\cdot$ is an isometry up to a rational constant, since $2\pi\omega$ is the Hodge theoretic K\"ahler form (see for instance \cite[Prop. 1.2.31]{Huy}). The last fact is that the minimal component decomposition of Proposition \ref{prop:orthogonal-decompositions} is also orthogonal for the $L^{2}$ norm, since it is orthogonal for the intersection form by construction. This settles \eqref{eq:je-manque-didees}. Furthermore, we notice that as for $\eta_{\frac{n-3}{2},\frac{n-3}{2}}$, the $L^{2}$ norm of $v$ is constant, since it is flat by construction and it is the wedge product of a collection of sections of pure Hodge bidegree $((n-1)/2,(n-1)/2)$. Therefore, the norm $\|\eta_{\frac{n-1}{2},\frac{n-1}{2}}\|_{\l2}^{2}$ equals $\|\eta_\frac{n-1}{2}\|_{\l2}^{2}$ up to a constant.
\end{itemize}
\bigskip

\noindent\emph{Determining the rational function $\Delta$.} To establish the theorem we need to specify the element $\Delta\in\QBbb(\psi)^{\times}\otimes\QBbb$ in \eqref{eq:bcov-ARR} (formal rational power of a rational function), which satisfies: 
\begin{equation}\label{DefinitionDelta}
   \log \tau_{\bcov}= \log|\Delta|^{2}+\frac{\chi}{12}\log\|\eta\|_{\l2}^{2}
   -\sum_{p,q} (-1)^{p+q}p\log \|\eta_{p,q}\|_{\l2}^{2}
   +\log C_{\sigma}.
\end{equation} 
We will determine $\Delta$ up to an algebraic number. To this end, it suffices to know its divisor. Unless $\psi = 0$ or $\psi = \xi$ where $\xi^{n+1}=1$, $\Delta$ has no zeroes or poles by construction, since the sections $\eta_{p,q}$ are holomorphic and non-vanishing, and $\log \tau_{\bcov}$ is smooth. Hence we are lead to consider the logarithmic behaviour of the right hand side of \eqref{DefinitionDelta} at these points. Since for $2p\neq n-1$ the sections $\eta_{p,p}$ have constant $L^{2}$ norm, we only need to examine the functions $\log\|\eta_{p,q}\|_{\l2}$ with $p+q=n-1$.\\

\noindent\emph{Behaviour at $\psi=0$.} This corresponds to a smooth fiber of $f\colon\Zcal\to U$. Hence $\log \tau_{\bcov}$ is smooth at $\psi=0$, as are the $L^2$ metrics.  However, the sections $\eta_{p,q}$ with $p+q=n-1$ admit zeros at $\psi = 0 $ (see Remark \ref{rmk:vanishing-eta-0}), with $\ord_0 \eta_{p,q} = q+1 = n-p$. This means that $a$ in the theorem is given by
\begin{displaymath}
(n+1)a = (-1)^{n-1}\sum_{p=0}^{n-1} p (n-p)- \frac{\chi}{12} 
=(-1)^{n-1}\frac{(n-1)n(n+1)}{6}- \frac{\chi}{12}.
\end{displaymath}

\noindent\emph{Behaviour at $\psi=\xi\in\mu_{n+1}$.} This corresponds to a singular fiber of $f\colon\Zcal\to\PBbb^{1}$, which has a unique ordinary double point. By Theorem \ref{thm:order-eta-xi} we control $\ord_\xi \eta_k$ according to the parity of $n-1$. Here we encounter the additional problem that the $L^2$ norms might have contributions from the semi-simple part of the monodromy $T_{s}$. More precisely, consider the  local parameter $t = \psi - \xi$ around $\xi$, and write $\eta_{p,q} =  t^{b_{p,q}}\sigma_{p,q}$ where $\sigma_{p,q}$ trivializes $\det R^q f_{\ast} \Omega^{p}_{\Zcal^{\prime}/\PBbb^{1}}(\log)$. Then by construction of $\eta_{p,q}$ and by \cite[Thm. C]{cdg2}, we have
\begin{displaymath}
\log\|\eta_{p,q}\|^2_{\l2} =(b_{p,q}+ \alpha_{p,q}) \log|t|^2 + o(\log|t|^2)
\end{displaymath}
with
\begin{displaymath}
    \alpha_{p,q}=-\frac{1}{2\pi i}\tr\left(^{\ell}\log T_{s}\mid \Gr_{F_{\infty}}^{p}H^{n-1}_{\lim}\right)\in\QBbb.
\end{displaymath}
Here $^{\ell}\log$ refers to the lower branch of the logarithm, \emph{i.e.} with argument in $2\pi (-1,0]$. Let us combine all this information:\\
\noindent\emph{Odd case}: If $n-1$ is odd, 
   according to Theorem \ref{thm:order-eta-xi} , if  $k \leq \frac{n}{2}-1, \ord_\xi \eta_k = 0$ and $\ord_\xi \eta_k = -1$ otherwise. In this case the monodromy is unipotent, so that $\alpha_{p,q}=0$ for all $p+q=n-1$. Moreover, by \cite[Thm. B]{cdg2}, we have that $\log \tau_{\bcov}=  \frac{n}{24}\log|t|^2 + o(\log|t|^2).$ Putting all these contributions together we find that
\begin{displaymath}
  -b=\frac{n}{24}+(-1)^{n-1}\sum_{k = n/2}^{n-1} (n-1-k) \cdot (-1)  = \frac{n(3n-5)}{24}.
\end{displaymath}

\noindent\emph{Even case}: If $n-1$ is even, according to Theorem \ref{thm:order-eta-xi} , if  $k \leq \frac{n-3}{2}, \ord_\xi \eta_k = 0$ and $\ord_\xi \eta_k = -1$ otherwise. 
Also, unless $p = q= (n-1)/2$, $\alpha_{p,q}=0$. In the remaining case $p=q=(n-1)/2$, by \cite[Prop. 3.10]{cdg2} we have $\alpha_{p,p}= 1/2$. Finally, from \cite[Thm. B]{cdg2}, we have that $\log \tau_{\bcov}=  \frac{3-n}{24}\log|t|^2 + o(\log|t|^2).$ Putting all these contributions together we find that 
\begin{displaymath}
   \begin{split}
       -b=&\frac{3-n}{24}+(-1)^{n-1}\left((n-1)/2(-1 + 1/2)+\sum_{k = (n+1)/2}^{n-1} (n-1-k) \cdot (-1)\right) = -\frac{n(3n-5)}{24}.
    \end{split}
\end{displaymath}
\bigskip

\noindent\emph{Rationality considerations.} To complete the proof of the theorem, we still need to tackle the constant $C$. There are two sources that contribute: i) for $2p\neq n-1$, the $L^{2}$ norms $\|\eta_{p,p}\|_{\l2}$ are constant and ii) if $n-1=2p$, after \eqref{eq:je-manque-didees}, there might be extra contributions from $\|\eta_{\frac{n-3}{2},\frac{n-3}{2}}\|_{\l2}$ and from $\|v\|_{\l2}$. 

First for $2p\neq n-1$. Let $\psi_{0}\in\QBbb$ be in the smooth locus, so that we have the period isomorphism
\begin{displaymath}
    H^{2p}(Z_{\psi_{0}},\Omega_{Z_{\psi_{0}}/\QBbb}^{\bullet})\otimes_{\QBbb}\CBbb\overset{\sim}{\longrightarrow} H^{2p}(Z_{\psi_{0}},\QBbb)\otimes\CBbb.
\end{displaymath}
Taking rational bases on both sides, the determinant can be defined in $\CBbb^{\times}/\QBbb^{\times}$. It equals $(2\pi i)^{pb_{2p}}$. Since $\|\eta_{p,p}\|_{\l2}$ is constant, it can be evaluated at $\psi=\psi_{0}$. We find
\begin{equation}\label{eq:period-isom-L2-aux}
    \|\eta_{p,p}\|_{\l2}^{2}\sim_{\QBbb^{\times}}(2\pi)^{2pb_{2p}}\vol_{\l2}(H^{2p}(Z_{\psi_{0}},\ZBbb), \omega).
\end{equation}
Now recall from \eqref{eq:volume-not-rational} that with the Arakelov theoretic normalization of the K\"ahler form, and under the integrality assumption on its cohomology class, we have $\vol_{\l2}(H^{2p}(Z_{\psi_{0}},\ZBbb), \omega)\sim_{\QBbb^{\times}}(2\pi)^{-2pb_{2p}}$. All in all, we arrive at the pleasant  
\begin{equation}\label{eq:L2-eta-pp-1}
    \|\eta_{p,p}\|_{\l2}^{2}\sim_{\QBbb^{\times}}1.
\end{equation}

If $n-1=2p$ is even, we will show that in fact
\begin{equation}\label{eq:L2-eta-pp-2}
\|\eta_{\frac{n-1}{2},\frac{n-1}{2}}\|^{2}_{\l2}\sim_{\QBbb^{\times}}\|\eta_{\frac{n-1}{2}}\|^{2}_{\l2},
\end{equation}
 namely that both $\|\eta_{\frac{n-3}{2},\frac{n-3}{2}}\|_{\l2}^{2}$ and $\|v\|_{\l2}^{2}$ are rational. For $\eta_{\frac{n-3}{2},\frac{n-3}{2}}$, this is already known after \eqref{eq:L2-eta-pp-1}. We will now see that formally the same argument yields the case of $v$. Since the norm of $v$ is constant, it is enough to consider the value at any $\psi_{0}\in\QBbb$ in the smooth locus. The results in \textsection \ref{subsec:hodge-bundles} and \textsection \ref{subsec:Griffiths-sections} show that $(\VBbb_{\psi_{0}},\Vcal_{\psi_{0}})$ behaves as the pair formed by the rational Betti and algebraic de Rham primitive cohomologies in degree $n-1$ of a smooth projective algebraic variety defined over $\QBbb$, of dimension $n-1$. In particular, we have a period isomorphism and a Poincar\'e type duality induced by the intersection form. From this, one derives the analog of \eqref{eq:period-isom-L2-aux} for $v$: for any rational basis $v^{\prime}$ of $\det\VBbb_{\psi_{0}}$, we have
\begin{displaymath}
    \|v\|^{2}_{\l2}\sim_{\QBbb^{\times}} (2\pi)^{(n-1)d}\|v^{\prime}\|^{2}_{\l2},\quad d=\dim \VBbb_{\psi_{0}}.
\end{displaymath}
Now, we use that the Hodge structure on $\VBbb_{\psi_{0}}$ is concentrated in bidegree $((n-1)/2,(n-1)/2)$, by Proposition \ref{prop:orthogonal-decompositions}, and we take into account the Arakelov theoretic normalization of the K\"ahler form. We readily deduce $\|v^{\prime}\|^{2}_{\l2}\sim_{\QBbb^{\times}} (2\pi)^{-(n-1)d}$. All in all, we conclude that $\|v\|^{2}_{\l2}\sim_{\QBbb^{\times}} 1$ as desired.

Finally, plug into \eqref{DefinitionDelta} the equations \eqref{eq:L2-eta-pp-1} in the cases $2p\neq n-1$, and \eqref{eq:L2-eta-pp-2} in the case $2p=n-1$. Plug as well the value of $C_{\sigma}$ furnished by Theorem \ref{thm:ARR}, and recall that $\Delta$ was determined only up to algebraic number. We conclude that $C$ has the asserted shape.
\end{proof}

\begin{corollary}\label{cor:kappa-infty}
As $\psi\to\infty$, $\log\taubcov{Z_{\psi}}$ behaves as
\begin{equation}\label{eq:ass-bcov-infty}
    \log \taubcov{Z_\psi} = \kappa_\infty \log\left|\psi\right|^{-2} + \varrho_\infty \log \log|\psi|^{-2} + continuous,
\end{equation}
where 
\begin{eqnarray*}
    \kappa_\infty & = & (-1)^{n}\frac{n+1}{12}\left( \frac{(n-1)(n+2)}{2}+\frac{1-(-n)^{n+1}}{(n+1)^{2}}\right),\\
    \varrho_\infty & = & (-1)^{n-1}\frac{(n-1)(n+1)}{12}\left(\frac{(-n)^{n+1}-1}{(n+1)^2} -2n +1\right) .
\end{eqnarray*}
\end{corollary}
\begin{proof}
The general shape \eqref{eq:ass-bcov-infty} was proven in \cite[Prop. 6.8]{cdg2}. The precise value of $\kappa_{\infty}$ is $(n+1)(b-a)$ entirely due to the term $\left|\frac{(\psi^{n+1})^{a}}{(1-\psi^{n+1})^{b}}\right|$ in Theorem \ref{thm:pre-formula}. Indeed, by Theorem \ref{thm:eta-triv-infty} the sections $\eta_{k}$ trivialize $R^{k}f_{\ast}\Omega^{n-1-k}_{\Zcal^{\prime}/\PBbb^{1}}(\log)_{\min}$ at infinity, and moreover the monodromy is unipotent there (Lemma \ref{lemma:MHS-infty}). This entails that the functions $\log\|\eta_{k}\|_{\l2}^{2}$ are $O(\log\log|\psi|^{-2})$ at infinity, and hence do not contribute to $\kappa_{\infty}$. For the subdominant term, the expression of \cite[Prop. 6.8]{cdg2} can be explicitly evaluated for the mirror family, thanks to the complete understanding of the limiting Hodge structure at infinity (again Lemma \ref{lemma:MHS-infty}), and the known value of $\chi$ (Lemma~\ref{lemma:Hodge-numbers-crepant}). 
\end{proof}

\subsection{Canonical trivializations of the Hodge bundles at the MUM point}\label{subsec:norm-sect-GW}

\subsubsection*{The Picard--Fuchs equation of the mirror}\label{subsubsec:periods}
For the mirror family $f\colon\Zcal\to U$, we review classical facts on the Picard--Fuchs equation of the local system of middle degree cohomologies. The discussion serves as the basis for the construction of canonical trivializing sections of the middle degree Hodge bundles, close to the MUM point, which differ from the $\eta_{k}$ by some periods. 

The starting point is the construction of an invariant $(n-1)$-homological cycle at infinity for the mirror family $f\colon\Zcal\to\PBbb^{1}$. Recall the Dwork pencil $h\colon\Xcal \to \PBbb^1$, which comes with a natural embedding in $\PBbb^{n}\times\PBbb^{1}$. We obtain a "physical" $n$-cycle $\Gamma$ in $\PBbb^{n}$ as follows: we place ourselves in the affine piece $x_0\neq 0$ and define $\Gamma$ by the condition $|x_{i}/x_{0}|=1$ for all $i$. If $\psi\in\CBbb$ and $|\psi|^{-1}$ is small, then the fiber $X_{\psi}$ does not encounter $\Gamma$. Therefore, $\Gamma$ induces a constant family of cycles in $H_{n}(\PBbb^{n}\setminus X_{\psi},\QBbb)$. Notice that these are clearly $G$-invariant cycles. The tube map $H_{n-1}(X_\psi,\QBbb) \to H_{n}(\PBbb^n \setminus X_\psi, \QBbb)$ is surjective and $G$-equivariant, and induces an isomorphism $H_{n-1}(X_\psi,\QBbb)_{\prim} \simeq H_{n}(\PBbb^n \setminus X_\psi, \QBbb)$ by \cite[Prop. 3.5]{Griffiths-residues}. Therefore, we can find a $T$-invariant cycle $\widetilde{\gamma}_{0}\in H_{n-1}(X_\psi, \QBbb)_{\prim}^G$ corresponding to $\Gamma$. Finally, through the isomorphism $H_{n-1}(X_{\psi},\QBbb)^{G}_{\prim}\simeq H_{n-1}(Z_\psi, \QBbb)_{{\min}}$ deduced by duality from Lemma \ref{prop:iso-prim-coh} and Proposition \ref{prop:orthogonal-decompositions}, $|G|\cdot \widetilde{\gamma}_{0}$ maps to a $T$-invariant cycle on $Z_{\psi}$, denoted $\gamma_{0}$. The convenience of multiplication by $|G|$ will be clear in a moment.  

The period integral $I_{0}(\psi):=\int_{\gamma_{0}}\eta_{0}$ can be written as an absolutely convergent power series in $\psi^{-1}$. Indeed, taking into account the relationship between the cup-product on $X_{\psi}$ and $Z_{\psi}$ (see \emph{e.g.} Lemma \ref{lemma:cup-prod}), and the definition of $\eta_{0}$ (cf. Definition \ref{def:eta-k}) we find
\begin{displaymath}
    I_{0}(\psi)=\int_{\gamma_{0}}\eta_{0}=-\frac{(n+1)\psi}{|G|}\int_{|G|\cdot \widetilde{\gamma}_{0}}\theta_{0}
    =-(n+1)\psi\int_{\widetilde{\gamma}_{0}}\theta_{0}.
\end{displaymath}
For the latter integral, we use that the residue map and the tube map are mutual adjoint, and then perform an explicit computation: if $\DBbb\subset\CBbb$ is the unit disc around $0$, we have
\begin{equation}\label{eq:period-I0}
   \begin{split}
        I_{0}(\psi)&=\frac{1}{(2\pi i)^{n}}\int_{(\partial \DBbb)^{n}}\frac{-(n+1)\psi dz_{1}\wedge\ldots\wedge dz_{n}}{F_{\psi}(1,z_{1},\ldots, z_{n})}\\
        &=\sum_{j\geq 0}\frac{1}{((n+1)\psi)^{j}}\frac{1}{(2\pi i)^{n}}\int_{(\partial \DBbb)^{n}}\left(1+\sum_{l=1}^{n}z_{l}^{n+1}\right)^{j}\frac{dz_{1}}{z_{1}^{j+1}}\wedge\ldots\wedge\frac{dz_{n}}{z_{n}^{j+1}}\\
        &=\sum_{k\geq 0}\frac{1}{((n+1)\psi)^{(n+1)k}}\frac{((n+1)k)!}{(k!)^{n+1}}.
    \end{split}
\end{equation}
In these integrals, the parameters $z_{i}$ are the affine coordinates $x_{i}/x_{0}$. To obtain the last equality, we expand the integrands in the second line with Newton's multinomial formula, and then evaluate the resulting Cauchy integrals. We conclude that those with $j\neq (n+1)k$ for any $k$ vanish, while those with $j=(n+1)k$ for some $k$ equal $((n+1)k)!/(k!)^{n+1}$. Equation \eqref{eq:period-I0} is the period integral used in Theorem \ref{thm:eta-triv-infty}, to prove that $\eta_{0}$ trivializes $f_{\ast}K_{\Zcal/\DBbb_{\infty}}(\log)$.

To the local system $R^{n-1}f_* \CBbb$ there is an associated Picard--Fuchs equation, which coincides with that of $(R^{n-1}f_* \CBbb)_{\min}$, since the associated Hodge bundles of type $(n,0)$ are equal. We make the change of variable $z=\psi^{-(n+1)}$, so that $I_{0}$ becomes
\begin{displaymath}
    I_{0}(z)=\sum_{k\geq 0}\frac{z^{k}}{(n+1)^{(n+1)k}}\frac{((n+1)k)!}{(k!)^{n+1}}.
\end{displaymath}
Define the differential operators $\delta = z\frac{d}{dz}$ and  
\begin{equation}\label{eq:Picard-Fuchs}
    D=\delta^{n} - z \prod_{j=1}^{n} \left(\delta +\frac{j}{n+1}\right).
\end{equation}
 Differentiating $I_{0}(z)$ term by term and repeatedly, one checks $DI_{0}(z)=0$. It is known (cf. \cite[Thm. 6]{Gahrs} and \cite[Sec. 1]{Corti-Golyshev}) that this is in fact the Picard--Fuchs equation of $(R^{n-1} h_{\ast} \CBbb)_{\prim}$, which necessarily coincides with that of $(R^{n-1} h_{\ast} \CBbb)_{\prim}^{G}$, and hence $R^{n-1} f_{\ast} \CBbb$.

We now exhibit all the solutions of the Picard--Fuchs equation. For dimension reasons, these will determine a multivalued basis of homology cycles. Following Zinger (see \emph{e.g.} \cite[pp. 1214--1215]{Zingerstvsred}), for $q = 0, \ldots, n-1$ we define an \emph{a priori} formal series $I_{0,q}$ by 
\begin{displaymath}
    \sum_{q=0}^\infty I_{0,q}(t) w^q = e^{wt} \sum_{d=0}^\infty e^{dt}\frac{\prod_{r=1}^{(n+1)d}((n+1)w+r)}{\prod_{r=1}^d (w+r)^{n+1}}=:R(w,t).
\end{displaymath}
Let us also define $F(w,t)$ for the infinite sum on the right hand side, so that $R(w,t)=e^{wt}F(w,t)$. Under the change of variable \begin{equation}\label{eq:change-variable}
    e^{t}=(n+1)^{-(n+1)}z=((n+1)\psi)^{-(n+1)},
\end{equation}
the series $I_{0,0}(t)$ becomes $I_{0}(z)=I_{0}(\psi)$ \cite[eq.  (2--17)]{Zingerstvsred}. 

\begin{proposition}\label{prop:solutions-PF}
Under the change of variable \eqref{eq:change-variable}, the functions $I_{0,q}(z)$, $q=0,\ldots, n-1$, define a basis of multivalued holomorphic solutions of the Picard--Fuchs equation for the local system $R^{n-1}f_{\ast} \CBbb$ on $0<|z|<1$. 
\end{proposition}
\begin{proof}
We first recall that the Picard--Fuchs equations of $R^{n-1}f_{\ast} \CBbb$ and $(R^{n-1}f_{\ast} \CBbb)_{\min}$ coincide, and the latter is a local system of rank $n$. 

After the change of variable, one checks that $F(w,z)$ is absolutely convergent on compact subsets in the region $|w|<1$ and $|z|<1$. This implies that the functions $I_{0,q}(z)$ are multivalued holomorphic functions on $0<|z|<1$. Again taking into account the change of variable, it is formal to verify that $R(w,t)$ solves the Picard-Fuchs equation \eqref{eq:Picard-Fuchs}, and hence so do the functions $I_{0,q}(z)$. To see that they form a basis of solutions, it is enough to notice that each $I_{0,q}(z)$ has a singularity of the form $(\log z)^{q}$ as $z\to 0$. 
\end{proof}

\subsubsection*{An adapted basis of homological cycles} By Proposition \ref{prop:solutions-PF}, and because $(R^{n-1}f_{\ast}\CBbb)_{\min}$ has rank $n$, the functions $I_{0,q}(z)$ determine a flat multivalued basis of sections $\gamma_{q}$ of $(R^{n-1}f_{\ast}\CBbb)_{\min}^{\vee}$ on $0<|z|<1$, by the recipe 
\begin{equation*}\label{eq:def-gamma-q}
    I_{0,q}(z) = \int_{\gamma_q (z)} \eta_{0}.
\end{equation*}
See for instance \cite[Sec. 3.4 \& Lemme 3.12]{VoisinMirror} for a justification in an analogous situation. The notation is compatible with the invariant cycle $\gamma_{0}$ constructed above, as we already observed that $I_{0,0}(z)=I_{0}(z)$. The flat multivalued basis elements $\gamma_{q}$ provide a basis of $(H_{n-1})_{\lim}$, whose underlying vector space is seen here as the (minimal) homology of the general fiber. We still denote by $\gamma_{q}$ this basis of $(H_{n-1})_{\lim}$. We next prove that it is adapted to the homological weight filtration, recalled at the end of \textsection \ref{subsec:generalities-hodge}.
\begin{proposition}\label{prop:gamma-weight}
Let $W_{\bullet}^{\prime}$ be the weight filtration of the limiting mixed Hodge structure on $(H_{n-1})_{\lim}$. Then $\gamma_{q}\in W_{2q-2(n-1)}^{\prime}\setminus W_{2q-1-2(n-1)}^{\prime}$. 
\end{proposition}
\begin{proof} 
By \cite[Lemma (6.4)]{schmid}, the Poincar\'e duality induces an isomorphism between the weight filtration $W_{r}$ on $H^{n-1}_{\lim}$ to the dual weight filtration $W_{r-2(n-1)}^{\prime}$ on $(H_{n-1})_{\lim}$. Therefore, it is enough to establish $\gamma^{\prime}_{q}\in W_{2q}\setminus W_{2q-1}$ for the Poincar\'e duals $\gamma^{\prime}_{q}\in H^{n-1}_{\lim}$.

On each fiber $Z_{z}$, the Hodge decomposition and the Cauchy--Schwarz inequality imply 
\begin{displaymath}
    |I_{0,q}(z)|=\left|\int_{Z_{z}}\gamma^{\prime}_{q}(z)\wedge\eta_{0}\right|\leq (2\pi)^{n-1}\|\gamma^{\prime}_{q}(z)\|_{\l2}\|\eta_{0}\|_{\l2}.
\end{displaymath}
Now $|I_{0,q}(z)|$ grows like $(\log |z|^{-1})^{q}$ as $z\to 0$ along angular sectors (cf. proof of Proposition \ref{prop:solutions-PF}). Because the monodromy is maximally unipotent at infinity and $\eta_{0}$ is a basis of $f_{\ast}K_{\Zcal/\DBbb_{\infty}}(\log)$, the $L^{2}$ norm $\|\eta_{0}\|_{\l2}$ grows like $(\log |z|^{-1})^{(n-1)/2}$ (see \cite[Thm. A]{cdg} or the more general \cite[Thm. 4.4]{cdg2}). We infer that as $z\to 0$, along angular sectors,
\begin{displaymath}
    \|\gamma_{q}^{\prime}(z)\|_{\l2}\gtrsim (\log|z|^{-1})^{\frac{2q-(n-1)}{2}}.
\end{displaymath}
By Schmid's metric characterization of the limiting Hodge structure \cite[Thm. 6.6]{schmid}, we then see that $\gamma_{q}^{\prime}\not\in W_{2q-1}$. 

It remains to show that $\gamma_{q}^{\prime}\in W_{2q}$. First of all, starting with $q = n-1$, we already know $\gamma_{n-1}^{\prime}\in W_{2n-2}\setminus W_{2n-3}$. We claim that $\gamma_{n-2}^{\prime}\in W_{2n-4}$. Otherwise  $\gamma_{n-2}^{\prime}\in W_{2n-2}\setminus W_{2n-4}$. But the weight filtration has one-dimensional graded pieces in even degrees, and zero otherwise (cf. Lemma \ref{lemma:MHS-infty}). It follows that $W_{2n-4}=W_{2n-3}$ and $\gamma_{n-1}^{\prime}=\lambda\gamma_{n-2}^{\prime}+\beta$, for some constant $\lambda$ and some $\beta\in W_{2n-4}$. Integrating against $\eta_{0}$, this relation entails
\begin{displaymath}
    I_{0,n-1}(z)=\lambda I_{0,n-2}(z)+\int_{Z_{z}}\beta(z)\wedge\eta_{0},
\end{displaymath}
where $\beta(z)$ is the flat multivalued section corresponding to $\beta$. Let us examine the asymptotic behaviour of the right hand side of this equality, as $z\to 0$, along angular sectors. We know that $|I_{0,n-2}(z)|$ grows like $(\log |z|^{-1})^{n-2}$. By the Hodge decomposition and the Cauchy--Schwarz inequality, and Schmid's theorem, the integral grows at most like $(\log |z|^{-1})^{n-2}$. This contradicts that $|I_{0,n-1}(z)|$ grows like $(\log |z|^{-1})^{n-1}$. Hence $\gamma_{n-2}^{\prime}\in W_{2n-4}$. Continuing inductively in this fashion, we conclude that $\gamma_{q}^{\prime}\in W_{2q}$ for all $q$, as desired.
\end{proof}

\subsubsection*{A normalized basis of $R^{n-1}f_{\ast}\Omega^{\bullet}_{\Zcal/\DBbb_{\infty}}(\log)_{\min}$}\label{sub:normalised} We construct a basis of holomorphic sections of $R^{n-1}f_{\ast}\Omega^{\bullet}_{\Zcal/\DBbb_{\infty}}(\log)_{\min}$ close to infinity, which correspond to the period integrals $I_{p,q}(z)$. We proceed inductively:
\begin{enumerate}
    \item set $\widetilde{\vartheta}_{0}=\eta_{0}$;
    \item for $p\geq 1$, suppose that $\widetilde{\vartheta}_{0},\ldots, \widetilde{\vartheta}_{p-1}$ have been constructed. Define
    \begin{displaymath}
        I_{p-1,q}(z)=\int_{\gamma_{q}(z)}\widetilde{\vartheta}_{p-1}.
    \end{displaymath}
    This notation is consistent with the previous definition of $I_{0,q}$;
    \item as by \cite[Prop. 3.1]{Zingerreduced}, in turn based on \cite{ZaZi}, the integral $I_{p-1,p-1}(z)$ is holomorphic and non-vanishing at $z=0$, we can define $\widetilde{\vartheta}_{p}$ by
    \begin{equation}\label{construction-theta}
        \widetilde{\vartheta}_{p}=\nabla_{zd/dz}\left( \frac{\widetilde{\vartheta}_{p-1} }{I_{p-1,p-1}(z)}\right);
    \end{equation}
\end{enumerate}
One verifies integrating \eqref{construction-theta} over $\gamma_{q}(z)$ that the period integrals $I_{p,q}(z):=\int_{\gamma_{q}(z)}\widetilde{\vartheta}_{p}$ satisfy the following recursion:
\begin{equation}\label{recursion-ipq}
    I_{p,q}(z)=z\frac{d}{dz}\left(\frac{I_{p-1,q}(z)}{I_{p-1,p-1}(z)}\right).
\end{equation}
Taking into account the change of variable \eqref{eq:change-variable}, we see that this is the same recurrence relation as in \cite[eq. (2--18)]{Zingerstvsred} (see also \cite[eq. (0.16)]{Zingerreduced}). Hence the $I_{p,q}(z)$ above coincides with the $I_{p,q}(t)$ in \emph{loc. cit.}
We further normalize:
    \begin{displaymath}
        \vartheta_{p}=\frac{\widetilde{\vartheta}_{p}}{I_{p,p}(z)}.
    \end{displaymath}
\begin{proposition}\label{prop:properties-vartheta}
\begin{enumerate}
    \item\label{item:prop-varth-1} For all $k$, the sections $\lbrace\vartheta_{j}\rbrace_{j=0,\ldots, k}$, constitute a holomorphic basis of the filtered piece $\Fcal^{n-1-k}R^{n-1}f_{\ast}\Omega^{\bullet}_{\Zcal/\DBbb_{\infty}}(\log)_{\min}$.
    \item\label{item:prop-varth-2} The periods of $\vartheta_{k}$ satisfy
    \begin{displaymath}
        \int_{\gamma_{k}}\vartheta_{k}=1
        \quad\textrm{and}\quad
        \int_{\gamma_{q}}\vartheta_{k}=0\quad\textrm{if}\quad q<k.
    \end{displaymath}
    \item\label{item:prop-varth-3} The projection of $\vartheta_{k}$ to $R^{k}f_{\ast}\Omega_{\Zcal/\DBbb_{\infty}}^{n-1-k}(\log)_{\min}$ relates to $\eta_{k}$ by
    \begin{displaymath}
        (\vartheta_{k})^{n-1-k,k}=\frac{(-1)^k}{(n+1)^k}\frac{\eta_{k}}{\prod_{p=0}^{k}I_{p,p}(z)}.
    \end{displaymath}
    \item\label{item:prop-varth-4} The sections $\lbrace\vartheta_{j}\rbrace_{j=0,\ldots, n-1}$ are uniquely determined by properties \eqref{item:prop-varth-1}--\eqref{item:prop-varth-2} above.
\end{enumerate}
\end{proposition}

\begin{proof}
We noticed that the period integrals $I_{p,p}(z)$ are holomorphic in $z$ and non-vanishing at $z=0$. With this observation at hand, the claims \eqref{item:prop-varth-1}--\eqref{item:prop-varth-2} then follow from properties of the Gauss--Manin connection and Kodaira--Spencer maps, Lemma \ref{lemma:recurrence-eta} and Theorem \ref{thm:eta-triv-infty}.
From $\widetilde{\vartheta}_0=\eta_0=-(n+1)\psi\theta_0$ and the recursion~\eqref{eq:recurrence-theta} for $\theta_k$, the definition~\eqref{construction-theta} further normalized gives 
\begin{displaymath}
    \vartheta_k=(-1)^{k-1}(n+1)\psi^{k+1}\frac{\theta_k}{\prod_{p=0}^{k}I_{p,p}(z)}\mod \Fcal^{n-1-(k-1)}R^{n-1}f_{\ast}\Omega^{\bullet}_{\Zcal/\DBbb_{\infty}}(\log)_{\min}.
\end{displaymath}
As $\theta_k$ maps to $\eta^\circ_k=-\frac{\eta_k}{(n+1)^{k+1}\psi^{k+1}}$ in the Hodge bundle $(R^{k}f_{\ast}\Omega_{\Zcal/U}^{n-1-k})_{\min}$, this proves \eqref{item:prop-varth-3}.
The uniqueness property \eqref{item:prop-varth-4}is obtained by comparing two such bases adapted to the Hodge filtration as in \eqref{item:prop-varth-1}, and then imposing the period relations \eqref{item:prop-varth-2}. 
\end{proof}

Actually, the basis $\vartheta_{\bullet}=\lbrace\vartheta_{j}\rbrace_{j=0,\ldots,n-1}$ is determined by the limiting Hodge structure $H^{n-1}_{\lim}$, up to constant, as we now show:
\begin{proposition}\label{prop:vartheta-LHS}
\begin{enumerate}
    \item Let $\gamma_{\bullet}^{\prime}$ be an adapted basis of the weight filtration on $(H_{n-1})_{\lim}$, as in Proposition \ref{prop:gamma-weight}. Then there exists a unique holomorphic basis $\vartheta_{\bullet}^{\prime}$ of $R^{n-1}f_{\ast}\Omega^{\bullet}_{\Zcal/\DBbb_{\infty}}(\log)_{\min}$ satisfying the conditions analogous to \eqref{item:prop-varth-1}--\eqref{item:prop-varth-2} with respect to $\gamma_{\bullet}^{\prime}$. 
    \item There exist non-zero constants $c_{k}\in\CBbb$ such that $\vartheta_{k}^{\prime}=c_{k}\vartheta_{k}$.
\end{enumerate}
\end{proposition}
\begin{proof}
We prove both assertions simultaneously. We write $\gamma_{\bullet}$ and $\gamma_{\bullet}^{\prime}$ as column vectors. Since the graded pieces of the weight filtration on $(H_{n-1})_{\lim}$ are all one-dimensional, there exists a lower triangular matrix $A\in\GL_{n}(\CBbb)$ with $\gamma_{\bullet}^{\prime}=A\gamma_{\bullet}$. If we decompose $A=D+L$, where $D$ is diagonal and $L$ is lower triangular, we see that the entries of the column vector $\vartheta_{\bullet}^{\prime}:=D^{-1}\vartheta_{\bullet}$ fulfill the requirements.
\end{proof}

\begin{definition}\label{eq:norm-eta-periods}
We define the canonical trivializing section of $R^{k}f_{\ast}\Omega_{\Zcal/\DBbb_{\infty}}^{n-1-k}(\log)_{\min}$ to be
\begin{displaymath}
    \widetilde{\eta}_{k}=(\vartheta_{k})^{n-1-k,k}=\frac{(-1)^k}{(n+1)^k}\frac{\eta_{k}}{\prod_{p=0}^{k}I_{p,p}(z)}.
\end{displaymath}
\end{definition}
By the previous proposition, up to constants, the sections $\widetilde{\eta}_{k}$ depend only on $(H_{n-1})_{\lim}$, or equivalently $H^{n-1}_{\lim}$ by Poincar\'e duality. These constructions are part of a wider framework about distinguished sections for degenerations of Hodge--Tate type. It is discussed in more detail in \textsection \ref{subsec:distinguished-sections}.

\subsection{Generating series of Gromov--Witten invariants and Zinger's theorem}
In order to state Zinger's theorem on generating series of Gromov--Witten invariants of genus one, and for coherence with the notations of this author, it is now convenient to work in the $t$ variable instead of $z$. The mirror map in Zinger's normalizations is the change of variable
\begin{equation}\label{eq:Zinger-mirror-map}
    t\mapsto T=\frac{I_{0,1}(t)}{I_{0,0}(t)}=\frac{\int_{\gamma_{1}(t)}\eta_{0}}{\int_{\gamma_{0}(t)}\eta_{0}}.
\end{equation}
Notice that this differs by a factor $2\pi i$ from the more standard Morrison's mirror map \cite{morrison-mirror} used in the introduction. The Jacobian of the mirror map is computed from \eqref{recursion-ipq}
\begin{displaymath}
    \frac{dT}{dt}=I_{1,1}(t).
\end{displaymath}
Let us introduce some last notations: 
\begin{itemize}
    \item $X_{n+1}$ denotes a general degree $n+1$ hypersurface in $\PBbb^n$.
    \item $N_1(0) = - \left( \frac{(n-1)(n+2)}{48}+\frac{1-(-n)^{n+1}}{24(n+1)^2}\right)=\frac{1}{24}\left(-\frac{n(n+1)}{2}+\frac{\chi(X_{n+1})}{n+1}\right)$.
    \item $N_1(d)$ is the genus 1 and degree $d$ Gromov-Witten invariant of $X_{n+1}$ ($d\geq 1$).
\end{itemize}
From these invariants we build a generating series:
\begin{equation}\label{eq:F1A}
    F_{1}^{A}(T)=N_1(0)T  + \sum_{d=1}^\infty N_{1}(d) e^{dT}.
\end{equation}
It follows from \cite[Thm. 2]{Zingerstvsred} that this generating series satisfies
\begin{eqnarray*}
F_{1}^{A}(T) & = & N_1(0) t + \frac{(n+1)^2 -1 + (-n)^{n+1}}{24 (n+1)} \log I_{0,0}(t) \\
&  & -\begin{cases}
\frac{n}{48} \log (1-(n+1)^{n+1} e^t) + \sum_{p=0}^{(n-2)/2} \frac{(n-2p^2)}{8}\log I_{p,p}(t), \hbox{ if } n \hbox{ even }\\
\frac{n-3}{48} \log (1-(n+1)^{n+1} e^t) + \sum_{p=0}^{(n-3)/2)} \frac{(n+1-2p)(n-1-2p)}{8} \log I_{p,p}(t) \hbox{ if } n \hbox{ odd }. \end{cases}
\end{eqnarray*}

This identity has to be understood in the sense of formal series. As an application of relations between the hypergeometric series $I_{p,p}(t)$, studied in detail in \cite{ZaZi}, 
the following identity holds (for a version of this particular identity, see \cite[eq. (3.2)]{Zingerreduced}):
\begin{multline*}
\frac{n(3n-5)}{48}\log (1-(n+1)^{n+1} e^t) + \frac{1}{2}\sum_{p=0}^{n-2} {n-p \choose 2} \log I_{p,p}(t) = \\ 
   \begin{cases}
\frac{n}{48} \log (1-(n+1)^{n+1} e^t) + \sum_{p=0}^{(n-2)/2} \frac{(n-2p^2)}{8}\log I_{p,p}(t), \hbox{ if } n \hbox{ even }\\
\frac{n-3}{48} \log (1-(n+1)^{n+1} e^t) + \sum_{p=0}^{(n-3)/2)} \frac{(n+1-2p)(n-1-2p)}{8} \log I_{p,p}(t) \hbox{ if } n \hbox{ odd } \end{cases}
\end{multline*}
Consequently, Zinger's theorem takes the following pleasant form, that we will use to simplify the task of recognizing $F_{1}^{A}(T)$ in our expression for the BCOV invariant (cf. Theorem \ref{thm:pre-formula}). 
\begin{theorem}[Zinger] \label{theo:Zinger}
Under the change of variables $t \mapsto T$, the series $F_{1}^{A}(T)$ takes the form
\begin{equation}\label{eq:Zinger-expression}
    \begin{split}
         F_{1}^{A}(T)\ =\ &N_1(0) t + \frac{\chi (X_{n+1})}{24}\log I_{0,0}(t)\\
         &-\frac{n(3n-5)}{48}\log (1-(n+1)^{n+1} e^t) - \frac{1}{2}\sum_{p=0}^{n-2} {n-p \choose 2} \log I_{p,p}(t).
    \end{split}
\end{equation}
\end{theorem}
A final remark on the holomorphicity of $F_{1}^{A}(T)$ is in order. While Theorem \ref{theo:Zinger} is \emph{a priori} an identity of formal series, the right hand side of \eqref{eq:Zinger-expression} is actually a holomorphic function in $t$, for $\Real t\ll 0$. Then, via the mirror map, $F_{1}^{A}(T)$ acquires the structure of a holomorphic function in $T$. One can check that the domain of definition is a half-plane $\Real T\ll 0$.  

\subsection{Genus one mirror symmetry and the BCOV invariant}
We are now in position to show that the BCOV invariant of the mirror family $f\colon\Zcal\to U$ realizes genus one mirror symmetry for Calabi--Yau hypersurfaces in projective space. That is, one can extract the generating series $F_{1}^{A}(T)$ from the function $\psi\mapsto\tau_{\bcov}(Z_{\psi})$ . The precise recipe by which this is accomplished goes through expressing $\tau_{\bcov}$ in terms of the $L^{2}$ norms of the canonical sections $\widetilde{\eta}_{k}$ (cf. Definition \ref{eq:norm-eta-periods}). But first we need to make $\tau_{\bcov}(Z_{\psi})$ and $F_{1}^{A}(T)$ depend on the same variable. To this end, we let
\begin{equation}\label{eq:F1B}
    F_{1}^{B}(\psi)= F_{1}^{A}(T),\quad\text{for}\quad T=\frac{I_{0,1}(t)}{I_{0,0}(t)}\quad\text{and}\quad e^{t}=((n+1)\psi)^{-(n+1)}.
\end{equation}

\begin{theorem}\label{thm:BCOV-mirror}
In a neighborhood of $\psi=\infty$, there is an equality
\begin{displaymath}
    \tau_{\bcov}(Z_{\psi})=C\left|\exp\left((-1)^{n-1}F_{1}^{B}(\psi)\right)\right|^{4} \frac{\|\widetilde{\eta}_{0}\|^{\chi/6}_{\l2}}{\left(\prod_{k=0}^{n-1}\|\widetilde{\eta}_{k}\|_{\l2}^{2(n-1-k)}\right)^{(-1)^{n-1}}},
\end{displaymath}
where $\chi=\chi(Z_{\psi})$ and $C\in\pi^{c}\QBbb^{\times}_{>0}$,  $c=\frac{1}{2}\sum_{k}(-1)^{k+1}k^{2}b_{k}$.
\end{theorem}
\begin{proof}
The proof is a simple computation, which consists in changing the variable $T$ to $\psi$, using \eqref{eq:F1B}, in the expression for $F_{1}^{A}(T)$ provided by Theorem \ref{theo:Zinger}. For the computation, recall that for a smooth hypersurface $X_{n+1}$ in $\PBbb^{n}$, $\chi(X_{n+1}) = (-1)^{n-1} \chi$. Modulo $\log$ of rational numbers, we find
\begin{displaymath}
    \begin{split}
    4F_{1}^{A}(T)=&\left(-\frac{n(n+1)}{12}+\frac{\chi(X_{n+1})}{6(n+1)}\right)t + 
         \frac{\chi (X_{n+1})}{6}\log I_{0,0}(t)\\
         &-\frac{n(3n-5)}{12}\log (1-(n+1)^{n+1} e^t) - 2\sum_{p=0}^{n-2} {n-p \choose 2} \log I_{p,p}(t)\\
         =&\left(\frac{n(n+1)}{12}-\frac{\chi(X_{n+1})}{6(n+1)}+\frac{n(3n-5)}{12}\right) \log (\psi^{n+1})\\
         &-\frac{n(3n-5)}{12}\log (\psi^{n+1}-1)
         +\frac{\chi (X_{n+1})}{6}\log I_{0,0}(t)
         - 2\sum_{p=0}^{n-2} {n-p \choose 2} \log I_{p,p}(t)\\
         =&(-1)^{n-1}\log\frac{(\psi^{n+1})^{2a}}{(\psi^{n+1}-1)^{2b}}
         +(-1)^{n-1}\frac{\chi}{6}\log I_{0,0}(t)
         - 2\sum_{p=0}^{n-2} {n-p \choose 2} \log I_{p,p}(t).\\
    \end{split}
\end{displaymath}
Now, in terms of the canonical trivializing sections $\widetilde{\eta}_{k}$ given in Definition \ref{eq:norm-eta-periods}, Theorem \ref{thm:pre-formula} becomes:
\begin{displaymath}
\taubcov{Z_{\psi}}   = C\left|\frac{(\psi^{n+1})^{a}}{(1-\psi^{n+1})^{b}}\right|^{2}
    \frac{|I_{0,0}(t)|^{\chi/6}}{\left(\prod_{p=0}^{n-2}|I_{p,p}(t)|^{2{n-p \choose 2}}\right)^{(-1)^{n-1}}}
\frac{ \|\widetilde{\eta}_{0}\|^{\chi/6}_{\l2}}{\left(\prod_{k=0}^{n-1}\|\widetilde{\eta}_{k}\|_{\l2}^{2(n-1-k)}\right)^{(-1)^{n-1}}}.\qedhere
\end{displaymath}
\end{proof}
\begin{remark}\label{rem:apresmainthm}
\begin{enumerate} 
    \item In relative dimension 3, we recover the main theorem of Fang--Lu--Yoshikawa \cite[Thm 1.3]{FLY}. Their result is presented in a slightly different form. The first formal discrepancy is in the choice of the trivializing sections. Their trivializations can be related to ours via Kodaira--Spencer maps. The second discrepancy is explained by a different normalization of $F_{1}^{A}$: they work with two times Zinger's generating series. This justifies why their expression for the BCOV invariant contains $|\exp(-F_{1}^{B}(\psi))|^{2}$, while our formula in dimension 3 specializes to $|\exp(-F_{1}^{B}(\psi))|^{4}$.
    
    \item The norms of the sections $\widetilde{\eta}_k$ are independent of the choice of crepant resolution. It follows that the expression on the right hand side in Theorem \ref{thm:BCOV-mirror} is independent of the crepant resolution, except possibly for the constant $C$. In \cite[Conj. B]{cdg2} we conjectured that the BCOV invariant is a birational invariant. A proof of this conjecture has been announced in \cite{Yeping-2, Yeping-3}. Thus $C$ should in fact be independent of the choice of crepant resolution.
\end{enumerate}
\end{remark}

\begin{corollary}\label{cor:determined}
\begin{enumerate}
    \item The invariant $N_{1}(0)$ satisfies
    \begin{displaymath}
        N_{1}(0)=\frac{-1}{24}\int_{X_{n+1}}\mathrm{c}_{n-2}(X_{n+1})\wedge H,
    \end{displaymath}
    where $H$ is the hyperplane class in $\PBbb^{n}$.
    \item As $\psi\to\infty$, $\log\tau_{\bcov}(Z_{\psi})$ behaves as
    \begin{equation}\label{eq:asymp-bcov-mirror}
        \log\tau_{\bcov}(Z_{\psi})=\left(\frac{(-1)^{n}}{12}\int_{X_{n+1}}\mathrm{c}_{n-2}(X_{n+1})\wedge H\right) \ \log\left|\psi^{-(n+1)}\right|^{2}+O(\log\log |\psi|).
    \end{equation}
\end{enumerate}
\end{corollary}
\begin{proof} 
The sought for interpretation of $N_{1}(0)$, or equivalently for the coefficient $\kappa_{\infty}$ in Corollary \ref{cor:kappa-infty}, is obtained by an explicit computation of, and comparison to $\int_{X_{n+1}}\mathrm{c}_{n-1}(\Omega_{X_{n+1}})\wedge H$. Indeed, by the cotangent exact sequence for the immersion of $X_{n+1}$ into $\PBbb^{n}$, this reduces to
\begin{displaymath}
    \int_{X_{n+1}}\mathrm{c}_{n-2}(\Omega_{X_{n+1}})\wedge H=\frac{(-1)^{n-1}}{n+1}\chi(X_{n+1})
    -\int_{\PBbb^{n}}\mathrm{c}_{n-1}(\Omega_{\PBbb^{n}})\wedge H,
\end{displaymath}
and we have explicit formulas for both terms on the right. This settles both the first and second claims.
\end{proof}
\begin{remark}
The asymptotic expansion \eqref{eq:asymp-bcov-mirror} has been written in the variable $\psi^{-(n+1)}$ on purpose, since this is the natural parameter in a neighborhood of the MUM point in the moduli space. In this form, the equation agrees with the predictions of genus one mirror symmetry (cf. \cite[Sec. 1.4]{cdg2} for a discussion). 
\end{remark}

\section{The refined BCOV conjecture}\label{sec:conjectures}

In this section, we propose an alternative approach to genus one mirror symmetry for Calabi--Yau manifolds, which bypasses spectral theory and is closer in spirit to the genus zero picture. The counterpart of the Yukawa coupling on the mirror side will now be a Grothendieck--Riemann--Roch isomorphism (GRR) of line bundles, built out of Hodge bundles. As in the case of the Yukawa coupling, one seeks canonical trivializations of these Hodge bundles, and the expression of the GRR isomorphism in these trivializations should then encapsulate the genus one Gromov--Witten invariants of the original Calabi--Yau manifold. This is our interpretation of the holomorphic limit of the BCOV invariant. We refer to this conjectural program as \emph{the refined BCOV conjecture}.

\subsection{The Grothendieck--Riemann--Roch isomorphism} Let $f\colon \Xcal\to S$ be a projective morphism of connected complex manifolds, whose fibers are Calabi--Yau manifolds. Recall from \eqref{eq:BCOV-bundle} that the BCOV bundle $\lambda_{\bcov}(\Xcal/S)$ is defined as a combination of determinants of Hodge bundles. Its formation commutes with arbitrary base change.
\begin{conjecture}\label{conj:GRR-iso}
For every projective family of Calabi--Yau manifolds $f\colon \Xcal\to S$ as above, there exists a natural isomorphism of line bundles, compatible with any base change,
\begin{equation}\label{eq:conjecture-functorial-GRR}
    \GRR(\Xcal/S) \colon\lambda_{\bcov}(\Xcal/S)^{\otimes 12 \kappa}\overset{\sim}{\longrightarrow} (f_{\ast} K_{\Xcal/S})^{\otimes\chi \kappa}.
\end{equation}
Here $\chi$ is the Euler characteristic of any fiber of $f$ and $\kappa$ only depends on the relative dimension of~$f$.

\end{conjecture}
Below we present some arguments in favour of the conjecture.
\begin{itemize}
    \item Applying this to the universal elliptic curve, the right hand side becomes trivial in view of $\chi=0$. This suggests that the left hand side is trivial. It is indeed trivialized by the discriminant modular form $\Delta$, with  $\kappa=1$. For higher dimensional abelian varieties both sides are trivial and the identity provides a natural isomorphism.
   
    \item For $K3$ surfaces both sides are identical, and the identity provides a  natural isomorphism. See in particular Proposition \ref{prop:refined-BCOV-K3}. The referee kindly communicated to us a proof of the analogue of the conjecture for Enriques surfaces, relying on the works about analytic torsions and the Borcherds' $\Phi$-function by Kawaguchi--Mukai--Yoshikawa \cite{KMY}, Dai--Yoshikawa \cite{DaiYoshikawa} and Yoshikawa \cite{Yoshik3surfinv}.
    
    \item In the category of schemes, a natural isomorphism of $\QBbb$-line bundles up to sign exists by work of Franke \cite{Franke} and the first author \cite{Dennis-these}. It is compatible with the arithmetic Riemann--Roch theorem, but is far more general and stronger.
\end{itemize}

The following proposition establishes a version of Conjecture \ref{conj:GRR-iso} in the setting of arithmetic varieties (cf. Section \ref{subsec:ARR}). This is an application of the arithmetic Riemann--Roch theorem \ref{thm:ARR}. Recall that an arithmetic ring $A$ comes together with a finite collection of complex embeddings $\Sigma$, closed under complex conjugation. We will write $A^{\times, 1}$ for the group of elements $u\in A^{\times}$ with $|\sigma(u)|=1$ for all embedding $\sigma\in\Sigma$. For instance, if $A$ is the ring of integers of a number field then  $A^{\times, 1}$ is a finite group.  If $A=\QBbb$ or $\RBbb$, then $A^{\times,1}=\{\pm 1\}$.
If $A=\CBbb$, then $A^{\times,1}$ is the unit circle in $\CBbb$. 

\begin{proposition}
\label{prop:GRR-ARR-iso}
Let $f\colon \Xcal\to S$ be a smooth projective morphism of arithmetic varieties over an arithmetic ring $A$, with Calabi--Yau fibers. Let $X_{\infty}$ be the generic fiber of $f$ and write $\chi=\chi(X_{\infty})$. Assume that $S\to\Spec A$ is surjective and has geometrically connected fibers. 
\begin{enumerate}
    \item There exists an integer $\kappa\geq 1$ and an isomorphism of line bundles on $S$
        \begin{displaymath}
            \GRR\colon\lambda_{\bcov}(\Xcal/S)^{\otimes 12\kappa}\overset{\sim}{\longrightarrow}(f_{\ast}K_{\Xcal/S})^{\otimes \chi\kappa},
        \end{displaymath}
    with the property of being an isometry for the Quillen-BCOV and $L^{2}$ metrics on $\lambda_{\bcov}(\Xcal/S)$ and $f_{\ast}K_{\Xcal/S}$, respectively.
    \item If $\GRR^{\prime}$ is another such isomorphism, for another choice of integer $\kappa^{\prime}\geq 1$, then
    \begin{displaymath}
        \GRR^{\prime\ \otimes \kappa}= \GRR^{\ \otimes \kappa^{\prime}}
    \end{displaymath}
    up to multiplication by some $u\in A^{\times, 1}$. Consequently, the formation of $\GRR$ is compatible with any base change between geometrically connected arithmetic varieties over $A$, up to the power $\kappa$ and multiplication by a unit in $A^{\times, 1}$.
\end{enumerate}
\end{proposition}
\begin{proof}
The first claim is a restatement of the identity \eqref{eq:ARR-BCOV} in $\ACH^{1}(S)_{\QBbb}$, together with the isomorphism $\cHat_{1}: \widehat{\Pic}(S) \overset{\sim}{\to} \ACH^{1}(S)$ and the very definition of $\widehat{\Pic}(S)$ as the group of isomorphism classes of hermitian line bundles over $S$. 

For the second claim, notice that both $\GRR^{\prime\ \otimes\kappa}$ and $\GRR^{\ \otimes\kappa^{\prime}}$ induce isometries between the hermitian line bundles $\lambda_{\bcov}(\Xcal/S)^{\otimes 12 \kappa\kappa^{\prime}}$ and $(f_{\ast}K_{\Xcal/S})^{\otimes \chi \kappa\kappa^{\prime}}$, endowed with the Quillen-BCOV and $L^{2}$ metrics, respectively. These isomorphisms differ by multiplication by a unit $u\in\Gamma(S,\Ocal_{S}^{\times})$. The isometry property guarantees that the induced holomorphic function on $S^{\an}$ has modulus one, and is constant on the connected components. Hence, if we fix $\sigma\in\CBbb$, $u$ is constant on $S_{\sigma}^{\an}$. If we see $u$ in $\Gamma(S_{\sigma},\Ocal_{S_{\sigma}}^{\times})$, we infer from this that it satisfies the descent condition with respect to $S_{\sigma}\to\Spec\CBbb$. It follows that $u$ already satisfies the descent condition with respect to $S\to\Spec A$. This is easily seen if $S$ is affine, by the flatness of $S\to\Spec A$, and in general one may replace $S$ by the disjoint union of the open subsets of an affine covering of $S$. Because $S\to\Spec A$ is actually faithfully flat by assumption, we conclude that $u\in A^{\times}$. Now $u$ has modulus one as a function on $S^{\an}$, which exactly means $u\in A^{\times, 1}$. The base change property then follows from the compatibility of $\lambda_{\bcov}(\Xcal/S)$ and $f_{\ast}K_{\Xcal/S}$ with base change, and the fact that the Quillen and Hodge metrics are preserved as well.
\end{proof}

\begin{remark}\label{rmk:GRR-iso}
\begin{enumerate}
    \item If $A^{\times, 1}$ is a finite group of order $d$, then the second claim of the corollary entails
   \begin{displaymath}
         \GRR^{\prime\ \otimes d\kappa}= \GRR^{\ \otimes d\kappa^{\prime}}.
   \end{displaymath}
   Therefore, after possibly adjusting $\kappa$, the isomorphism is uniquely determined.  
   \item\label{rmk:GRR-iso-2} The proposition applies to the mirror family of Calabi--Yau hypersurfaces studied in Section \ref{sec:dworkandmirrorfamilies}. Here $A = \QBbb$, and therefore the resulting isomorphism is determined by the previous remark.
\end{enumerate}
\end{remark}

\subsection{Strongly unipotent monodromy and distinguished sections}\label{subsec:distinguished-sections}

The below discussion is based on \cite{Delignemirror} and \cite[\textsection 6.3, \textsection 7.1]{Morrison-adapted}.

\subsubsection*{Hodge--Tate structures}
Let $(V, F^\bullet, W_\bullet)$ be a mixed Hodge structure on a $\QBbb$-vector space $V$, where $F^\bullet$ is the decreasing Hodge filtration of $V_\CBbb$, and $W_\bullet$ is the increasing weight filtration of $V$. We also write $W_\bullet$ for the induced filtrations on $V_\RBbb$ and $V_\CBbb$.

\begin{definition}\label{def:Hodge-Tate}
A mixed Hodge structure is Hodge--Tate if the Hodge filtration is opposite to the weight filtration, in the sense that for any $k$ the natural map 
\begin{displaymath}
    F^k \oplus W_{2k-2} \to V_\CBbb
\end{displaymath} is an isomorphism. Equivalently, if the following two conditions are satisfied: 

\begin{enumerate}
    \item $\Gr_{2k}^W V = W_{2k}/W_{2k-1}$ is isomorphic to a sum of Tate twists $\QBbb(-k).$ In other words, it is purely of type $(k,k).$
    \item $\Gr_{2k+1}^W V=\{ 0\}.$
\end{enumerate}
\end{definition}

According to \cite[\textsection 6]{Delignemirror}, if a limiting Hodge structure has this property, it should be viewed as maximally degenerate. An example of this situation is $H^{n-1}_{\lim}$ for the mirror family around $\infty$, as explained in the proof of Lemma \ref{lemma:MHS-infty}. For Calabi--Yau degenerations over $\DBbb^\times$, this condition on the limiting middle Hodge structure implies that the monodromy is maximally unipotent.

It follows from the definition of Hodge--Tate mixed Hodge structure that the natural map
\begin{equation}\label{HodgeTate-consequence}  
F^p/F^{p+1} \hookrightarrow V_\CBbb/F^{p+1} =(F^{p+1}\oplus W_{2p})/F^{p+1} \to W_{2p} \to W_{2p}/W_{2p-2}
\end{equation}
is an isomorphism, and that there are natural isomorphisms 
\begin{equation}\label{HodgeTate-intersection-1}
    F^p \cap W_{2p} \simeq \Gr^W_{2p}(V_\CBbb)
\end{equation}
and
\begin{equation}\label{HodgeTate-intersection-2}
    F^p \cap W_{2p} \simeq \Gr^p_{F}(V_\CBbb)
\end{equation}
compatible with the isomorphism \eqref{HodgeTate-consequence}.

\subsubsection*{Distinguished sections} Suppose now that we are provided a variation of integrally polarized Hodge structures $(\VBbb_\ZBbb, \Fcal^\bullet)$ of weight $w$ over $\Dbold^\times = (\DBbb^\times)^d$. Here $\VBbb_\ZBbb$ is an integral local system, and $\Fcal^p$ is the Hodge filtration  of $\Vcal:=\VBbb_\ZBbb \otimes_\ZBbb\Ocal_{\Dbold^\times}$. Denote by $\nabla$ the flat connection on $\Vcal$ and suppose the local monodromies are unipotent. Denote by $T_j$ the endomorphism of the local system $\VBbb_{\ZBbb}$, given by the monodromy around the coordinate axis $(s_j=0)$ of $(\DBbb^{\times})^{d}$. Consider the family of operators $N_j := \log T_j$ over $\Dbold^\times$. Let $\WBbb_k$ be the associated increasing weight monodromy filtration of $\VBbb_\QBbb$, and denote by $\Wcal_k = \WBbb_k \otimes \Ocal_{\Dbold^\times}$. The bundle $\Wcal_k$ is preserved by $\nabla$, satisfies $N_j \WBbb_k \subseteq \WBbb_{k-2}$, and for any positive real numbers $a_j > 0$, with $N = \sum a_j N_j$, we have an isomorphism
 \begin{displaymath}
     N^k : \Gr^W_{w+k} \VBbb_\RBbb \to \Gr^W_{w-k} \VBbb_\RBbb.
 \end{displaymath}
By the results of Schmid \cite{schmid},  associated to $(\VBbb_\ZBbb, \Fcal^\bullet)$ and for any base point $s\in \Dbold^\times$, there is a limiting mixed Hodge structure $V_{\lim}$ on $\VBbb_{\QBbb, s}$. Its  weight filtration is given by $\WBbb_{k,s}$. The following lemma can be found in \cite[Sec. 6]{Delignemirror}:

\begin{lemma}\label{lemma:hodgetatedeligne}
If $V_{\lim}$ is Hodge--Tate, then, after possibly shrinking $\Dbold^\times$,  $(\VBbb_\ZBbb, \Fcal^p, \WBbb_k)$ is a variation of mixed Hodge structures over $\Dbold^\times$, with the same Hodge numbers as $V_{\lim}$. In particular, the natural morphism $\Fcal^p \oplus \Wcal_{2p-2} \to \Vcal$ is an isomorphism over $\Dbold^\times$. 
\end{lemma}

\begin{remark}\label{remarkextensionhodge}
The Hodge and the weight filtrations on $\Vcal$ extend as subvector bundles of the Deligne extension $\widetilde{\Vcal}$ over $\Dbold.$ The extended filtrations continue to be opposite in the sense of Lemma \ref{lemma:hodgetatedeligne}.
\end{remark}

From now on, we suppose the limiting Hodge structure is Hodge--Tate. In this setting, after Lemma \ref{lemma:hodgetatedeligne}, we have the analogues of the isomorphisms \eqref{HodgeTate-consequence}, \eqref{HodgeTate-intersection-1} and \eqref{HodgeTate-intersection-2}, namely isomorphisms
\begin{equation}\label{HodgeTateconsequence-family}
    \Fcal^p/\Fcal^{p+1} \to \Wcal_{2p}/\Wcal_{2p-2},
\end{equation}
\begin{equation}\label{HodgeTate-intersection-1-family}
    \Fcal^p \cap \Wcal_{2p} \simeq \Gr_{2p}^W \Vcal
\end{equation}
and
\begin{equation}\label{HodgeTate-intersection-2-family}
    \Fcal^p \cap \Wcal_{2p} \simeq \Gr^p_F \Vcal.
\end{equation}
Since $N_j \WBbb_{2p} \subseteq \WBbb_{2p-2} \subseteq \WBbb_{2p-1}$, $N_j$ and hence the monodromy act trivially on the local system $\Gr_{2p}^W \VBbb_{\QBbb}$, which is thus constant. In particular, $\Gr_{2p}^{W}\Vcal$ is a trivial flat vector bundle. We abusively often identify this trivial local system with the vector space $H^{0}(\Dbold^{\times}, \Gr_{2p}^W \VBbb_{\QBbb})$. 

\begin{definition}\label{prop:distinguishedsections}
The distinguished sections of $\Fcal^p/\Fcal^{p+1}$ are the global sections corresponding to $\Gr^W_{2p} \VBbb_\CBbb \subseteq  \Gr^W_{2p}(\Vcal) $ under the isomorphism in \eqref{HodgeTateconsequence-family}. A basis of distinguished sections will be called a distinguished basis. It is unique up to a matrix transformation with constant complex coefficients.
\end{definition}

For the determinant bundles, we have the following corollary: 

\begin{corollary}\label{corollary:distinguisheddeterminant}
Keep the notations and assumptions of Definition \ref{prop:distinguishedsections}. The exterior products of a distinguished basis provides a local frame for $\det(\Fcal^p/\Fcal^{p+1})$.
\end{corollary}

Any frame as in the corollary will be called a distinguished trivialization. It is unique up to a scalar constant in $\CBbb^\times$. 

\begin{remark}\label{integraldistinguished}
Using the integral lattice $\VBbb_\ZBbb$ one naturally defines an integral structure $\Gr^W_{2k} \VBbb_\ZBbb$ on $\Gr^W_{2k} \VBbb_\QBbb$. Accordingly, there are integral distinguished sections and trivializations. The corresponding integral distinguished trivialization of $\det(\Fcal^p/\Fcal^{p+1})$ is unique up to sign. 

\end{remark}

Suppose now that we are in the geometric case of a projective family of complex manifolds $f: \Xcal \to \Dbold^\times$, endowed with a relatively ample line bundle. We assume that for any $k$, $R^k f_\ast \ZBbb$ is a local system with unipotent local monodromies. Each variation $(R^{k} f_\ast \QBbb)_{\prim}$ is integrally polarized and Schmid's theory in \cite{schmid} recalled above applies. By the Lefschetz decomposition $R^k f_\ast \QBbb$ admits a limiting Hodge structure, and in particular a monodromy weight filtration $\WBbb_\bullet$. The constructions are independent of the choice of ample line bundle. 

In the geometric case, we can provide an alternative description of the distinguished sections in terms of the behaviour of periods:

\begin{lemma}\label{lemma:char-distinguished-sections}
If the local monodromies of $R^k f_\ast \CBbb$ are unipotent, and the limiting Hodge structure is Hodge--Tate, the distinguished sections of $\Fcal^{p}/\Fcal^{p+1}$ uniquely correspond to elements $\eta \in \Fcal^p$, such that:
\begin{enumerate}
    \item $\int_{\gamma} \eta =0 $ for all multivalued flat homology cycles $\gamma$ in $\WBbb'_{-2p-2}\subseteq (R^k f_\ast \CBbb)^\vee$.
    \item $\int_{\gamma} \eta$ is constant for all multivalued flat homology cycles $\gamma$ in $\WBbb'_{-2p} \subseteq (R^k f_\ast \CBbb)^\vee.$
\end{enumerate}
\end{lemma}

\begin{proof}
By definition of the dual weight filtration, recalled at the end of \textsection \ref{subsec:generalities-hodge}, and equation \eqref{HodgeTate-intersection-2-family}, we can identify $\Gr_F^{p} \Vcal$ with the set of sections of $\Fcal^p$ whose periods along cycles in $\WBbb'_{-2p-2}$ vanish. Moreover, by \eqref{HodgeTate-intersection-1-family}, $\eta \in \Fcal^p \cap \Wcal_{2p} $ corresponds to a distinguished section exactly when $\nabla \eta \in  \Wcal_{2p-2}\otimes\Omega^{1}_{\Dbold^{\times}}$. This in turn is equivalent to $\int_\gamma \nabla \eta=0$ for all $\gamma$ of $\WBbb'_{-2p}$. The statement then follows from the formula \begin{equation*}
    d\left(\int_\gamma \eta\right) = \int_\gamma \nabla \eta
\end{equation*}
for flat multivalued homology sections $\gamma$.
\end{proof}

\subsubsection*{Strongly unipotent monodromy degenerations}

\mbox{}
\begin{definition}
We say that a projective family $f: \Xcal \to \Dbold^\times$ of complex manifolds is of strongly unipotent monodromy if:
\begin{enumerate}
    \item for all $k \geq 0$, all local monodromies $R^k f_\ast \QBbb$ are unipotent;
    \item for all $k \geq 0$, the variations of Hodge structures associated to the local systems $R^k f_\ast \QBbb$ have limits at 0 which are Hodge--Tate.
\end{enumerate}
\end{definition}
In the situation of a family of Calabi--Yau manifolds in relative dimension 3 with one-dimensional complex moduli, with $h^{1,0}=h^{2,0}=0$ and unipotent monodromies, strongly unipotent monodromy is equivalent to $N^3 \neq 0$ on $H^3_{\lim}$. This is the usual definition of maximally unipotent monodromy. In general, for a family of Calabi--Yau manifolds $f\colon \Xcal \to \Dbold^\times$ of relative dimension $n$, the definition of strongly maximally unipotent monodromy is stronger than imposing that $N^{n} \neq 0$ on $R^n f_\ast \CBbb.$

We now show that our results on the mirror family $f\colon\Zcal\to\DBbb_{\infty}^{\times}$ provide an example of the previous phenomena and constructions. In preparation for the discussion, recall the minimal decomposition introduced in Proposition \ref{prop:orthogonal-decompositions}, and in particular the local system $\VBbb$ and its associated flat vector bundle $\Vcal$. 
\begin{lemma}
The mirror family $f\colon\Zcal \to \DBbb_\infty^{\times}$ has strongly unipotent monodromy. 
\end{lemma}

\begin{proof}
Outside the middle cohomology $n-1$, being Hodge--Tate follows from the fact that the variation of Hodge structures associated to $R^{2p} f_\ast \QBbb$ is purely of type $(p,p)$ (see Lemma \ref{lemma:Hodge-numbers-crepant}) and the monodromy is trivial by Lemma \ref{lemma:trivial-hodge}.

In the middle cohomology, by the Lefschetz decomposition, it is enough to deal with the local system $(R^{n-1} f_\ast \QBbb)_{\prim}$. This is a sum of $(R^{n-1} f_\ast \QBbb)_{\min}$ and $\VBbb.$ The limiting Hodge structure associated to $\VBbb$ is Hodge--Tate by Proposition \ref{prop:orthogonal-decompositions} and Lemma \ref{lemma:trivial-V-local}. That  $(R^{n-1} f_\ast \QBbb)_{\min}$ is unipotent and has limiting mixed Hodge structure which is Hodge--Tate follows from Lemma \ref{lemma:MHS-infty} and its proof.
\end{proof}

The following proposition summarizes the results of \textsection \ref{subsec:triviality-hodge-bundles} and \textsection \ref{subsubsec:periods}, to the effect of describing distinguished trivializations of the Hodge bundles.
\begin{proposition}\label{prop:disttrivializationDwork}
The distinguished trivializations of the determinants of the Hodge bundles $R^q  \Omega_{\Zcal/\DBbb_\infty^\times}^{p}$ are described as follows: 
\begin{enumerate}
    \item Suppose $2p \neq n-1.$ Any basis of the trivial local system $R^{2p} f_\ast \CBbb$ provides a distinguished trivialization of $\det R^p f_\ast\Omega_{\Zcal/\DBbb_{\infty}^{\times}}^{p}$.
    \item Suppose $2k \neq n-1$ . Then $\widetilde{\eta}_k$ are distinguished trivializations of $R^{k} \Omega_{\Zcal/\DBbb_\infty^\times}^{n-1-k}$.
    \item Suppose $2k = n-1$. For any polarization $L$, any basis $u$ (resp. $v$) of the trivial local systems $R^{2n-4} f_\ast \CBbb$ (resp. $\VBbb$), the section $\widetilde{\eta}_k \wedge (\det Lu) \wedge \det v$ is a distinguished trivialization of $\det R^{k} f_\ast  \Omega_{\Zcal/\DBbb_\infty^\times}^{k}$.
\end{enumerate}
\end{proposition}

\subsection{Relationship with mirror symmetry} We now present our refinement of the BCOV conjecture, for degenerating families of Calabi--Yau manifolds with strongly unipotent monodromy. The statement predicts that $\GRR$ realizes genus one mirror symmetry. We then show that the conjecture holds for the case of mirrors of hypersurfaces in projective space, as a consequence of our previous main theorems. The case of K3 surfaces is not covered by those considerations, but a proof is also provided. 

To prepare for the formulation of the conjecture, let $f\colon \Xcal\to \Dbold^{\times}=(\DBbb^{\times})^{d}$ be a projective morphism of Calabi--Yau $n$-folds, with $d=h^{1,n-1}$ the dimension of the deformation space of the fibers, effectively parametrized and with strongly unipotent monodromy. We denote by $\widetilde{\eta}_{p,q} $ an integral distinguished trivialization of $\det R^q f_\ast\Omega_{\Xcal/\Dbold^{\times}}^{p}$, which is unique up to sign (see Corollary \ref{corollary:distinguisheddeterminant} and Remark \ref{integraldistinguished}). Using these, both bundles appearing in the conjectural Grothendieck--Riemann--Roch isomorphism \eqref{eq:conjecture-functorial-GRR} admits canonical trivializations. Precisely, up to sign, the BCOV bundle is canonically trivialized by:
\begin{equation}\label{eq:bcovdistinguished}
    \widetilde{\eta}_{\bcov} := \bigotimes_{p,q} \widetilde{\eta}_{p,q}^{\otimes (-1)^{p+q} p }.
\end{equation}
Likewise, $\widetilde{\eta}_{n,0}$ trivializes the $f_\ast K_{\Xcal/\Dbold^\times}$. Expressed in these trivializations, we can write
\begin{equation}\label{eq:definition-GRR-function}
    \GRR(\Xcal/\Dbold^{\times})\colon\ \widetilde{\eta}_{\bcov}^{12\kappa} \mapsto \GRR(s) \cdot\widetilde{\eta}_{n,0}^{\chi \kappa }
\end{equation}
for an invertible holomorphic function $s\mapsto\GRR(s)$ on $\Dbold^\times$. 

In the above situation, it is expected that there are canonical mirror coordinates $\qbold=(q_1, \ldots, q_d)$ on $\Dbold$. In \cite{morrison-mirror}, for one-dimensional moduli, this is constructed through exponentials of quotients of well selected periods. For mirrors of hypersurfaces, this amounts to Zinger's mirror map recalled in \eqref{eq:Zinger-mirror-map}. A general alternative construction of mirror coordinates in the Hodge-Tate setting is suggested in \cite[Sec. 14]{Delignemirror}.

\begin{conjecture}\label{conj:GRR-iso-F1}
Let $f\colon \Xcal\to \Dbold^{\times}=(\DBbb^{\times})^{d}$ be a projective morphism of Calabi--Yau $n$-folds, with $d=h^{1,n-1}$ the dimension of the deformation space of the fibers, effectively parametrized with strongly unipotent monodromy. In the mirror coordinates $\qbold=(q_1, \ldots, q_d)$ of $\Dbold$, the function defined in \eqref{eq:definition-GRR-function} becomes  
\begin{displaymath}
    \GRR(\qbold)=C\cdot\exp\left((-1)^{n}F_{1}^A(\qbold)\right)^{24\kappa},
\end{displaymath}
where $C$ is a constant, 
\begin{displaymath}
    F_{1}^A(\qbold)=-\frac{1}{24}\sum_{k=1}^{d}\left(\int_{X^{\vee}}\mathrm{c}_{n-1}(X^{\vee})\wedge\omega_{k}\right)\log q_{k}
    +\sum_{\beta\in H_{2}(X^{\vee},\ZBbb)}\GW_1(X^{\vee},\beta)\ \qbold^{\langle\underline{\omega},\beta\rangle}
\end{displaymath}
is a generating series of genus one Gromov--Witten invariants on a mirror Calabi--Yau manifold $X^{\vee}$, and:
\begin{itemize}
    \item $\underline{\omega}=(\omega_{1},\ldots,\omega_{d})$ is some basis of $H^{1,1}(X^{\vee})\cap H^{2}(X^{\vee},\ZBbb)$ formed by ample classes.
    \item $\GW_1(X^{\vee},\beta)$ is the genus one Gromov--Witten invariant on $X^{\vee}$ associated to the class $\beta$. 
    \item $\qbold^{\langle\underline{\omega},\beta\rangle}=\prod_{k}q_{k}^{\langle\omega_{k},\beta\rangle}$.
\end{itemize}
\end{conjecture}

As supporting evidence, we consider the case of the mirror family of Calabi--Yau hypersurfaces in $\PBbb^{n}$, and settle the second part of the Main Theorem in the introduction:

\begin{theorem}\label{thm:refined-BCOV-Dwork}
Let $n \geq 4$. Then Conjecture \ref{conj:GRR-iso} and Conjecture \ref{conj:GRR-iso-F1} are true, up to a constant, for the mirror family $f\colon\Zcal \to \DBbb_\infty^{\times}$ in a neighborhood of the MUM point.
\end{theorem}

\begin{proof}
First of all, the existence of a natural isomorphism as in Conjecture \ref{conj:GRR-iso} is provided by Proposition \ref{prop:GRR-ARR-iso} and Remark \ref{rmk:GRR-iso} \eqref{rmk:GRR-iso-2}. Secondly, for Conjecture \ref{conj:GRR-iso-F1}, consider  
$\widetilde{\eta}_{\bcov}$ defined as in \eqref{eq:bcovdistinguished}. Since distinguished trivializations are equal up to a constant, for the purpose of proving Conjecture \ref{conj:GRR-iso-F1} we can suppose that the sections $\widetilde{\eta}_{p,q}$ are actually those determined by Proposition \ref{prop:disttrivializationDwork}.
By the isometry property of $\GRR$ and the very definition of the BCOV invariant, we have 
\begin{displaymath}
    \tau_{\bcov}^{12\kappa} = \frac{\|\GRR( \widetilde{\eta}_{\bcov}^{12\kappa})\|_{\l2}^{2}}{ \|\widetilde{\eta}_{\bcov}\|_{\l2, \bcov}^{24\kappa}}.
\end{displaymath}
In other words,
\begin{equation}\label{eq:tau-bcov-norm-GRR}
     \tau_{\bcov} = |\GRR(\qbold)|^{1/6\kappa} \frac{\|\widetilde{\eta}_{n-1,0}\|_{\l2}^{\chi/6 }}{\|\widetilde{\eta}_{\bcov}\|_{\l2, \bcov}^{2}}.
\end{equation}
As in the proof of Theorem \ref{thm:pre-formula} (see also \cite[Prop. 4.2]{cdg2}), the quantity $\|\widetilde{\eta}_{\bcov}\|_{\l2, \bcov}$ coincides with the factor $\prod_{k=0}^{n-1}\|\widetilde{\eta}_{k}\|_{\l2}^{n-1-k}$ up to a constant. We conclude by comparing \eqref{eq:tau-bcov-norm-GRR} with Theorem \ref{thm:BCOV-mirror}.
\end{proof}

The cases of one and two dimensional Calabi--Yau varieties are not covered by the above result. The one dimensional case essentially corresponds to the Kronecker limit formula recalled in \textsection \ref{subsec:ApplicationKronecker}. We now study the case of $K3$ surfaces. Since $h^{1,1}=20$ for a $K3$ surface, our one-dimensional Dwork-type family cannot be a mirror family. It is still expected that the mirror of a K3 surface is a K3 surface, a systematic construction in terms of polarized lattices can be found in for example \cite{Dolgachev}. We will assume this below.

\begin{proposition}\label{prop:refined-BCOV-K3}
Conjecture \ref{conj:GRR-iso} and Conjecture \ref{conj:GRR-iso-F1} are true, up to a constant, for any mirror family of a K3 surface. Moreover, $\kappa = 1$.
\end{proposition}

\begin{proof}
The BCOV line takes a particularly simple form for a $K3$ surface $X$, its square can be written as: 
\begin{equation}\label{k3iso}
\lambda_{\bcov}(X)^{\otimes 2} = \det H^{2,0}(X)^{\otimes 4} \otimes \det H^{1,1}(X)^{\otimes 2} \otimes \det H^{2,2}(X)^{\otimes 4} \simeq \det H^{2,0}(X)^{\otimes 4}.
\end{equation}
The isomorphisms $\det H^{1,1}(X)^{\otimes 2} \simeq \mathbb{C}$ and $\det H^{2,2}(X) \simeq \mathbb{C}$ are both induced by Serre duality and are thus isometries, for the $L^2$ norms and standard metric on $\CBbb$. Since $\chi(X) = 24,$ the square of the right hand side of Conjecture \ref{conj:GRR-iso} is provided by the same object. 

Let $f\colon\Xcal \to \Dbold^{\times}$ be a family of K3 surfaces. The previous construction
  globalizes to an isomorphism of line bundles 
  \begin{equation*}
  \lambda_{\bcov}(\Xcal/\Dbold^{\times})^{\otimes 2}\overset{\sim}{\longrightarrow} (f_{\ast} K_{\Xcal/\Dbold^{\times}})^{\otimes 4}
  \end{equation*}
  compatible with base change. Taking 6th powers and setting $\kappa = 1$, this proves Conjecture \ref{conj:GRR-iso} in this case. We hence propose that $\GRR$ is induced by \eqref{k3iso}. 

Following the  proof of Theorem \ref{thm:refined-BCOV-Dwork}, to prove Conjecture \ref{conj:GRR-iso-F1}, we need to construct distinguished trivializations of both sides. For $H^{1,1}$, we choose the section of $\det R^1f_{\ast}\Omega^1_{\Xcal/\Dbold^{\times}} = (\det R^2f_{\ast} \CBbb)\otimes\Ocal_{\Dbold^{\times}}$ induced by a generator of $\det R^2f_{\ast}\ZBbb$, and analogously for $\det R^2f_{\ast} \Omega^2_{\Xcal/\Dbold^{\times}} = (\det R^4f_{\ast}\CBbb)\otimes\Ocal_{\Dbold{\times}}$.  Their $L^2$ norms are locally constant by \cite[Prop. 4.2]{cdg2}. Picking any distinguished section $\widetilde{\eta}_{2,0}$ of $f_{\ast}K_{\Xcal/\Dbold^{\times}}$, it allows us to write down the section $\widetilde{\eta}_{\bcov}$ of \eqref{eq:bcovdistinguished}.
 
The analogous formula to \eqref{eq:tau-bcov-norm-GRR} becomes, in this case, 

\begin{displaymath}
     \tau_{\bcov} = |\GRR(\qbold)|^{1/6} \frac{\|\widetilde{\eta}_{2,0}\|_{\l2}^{4 }}{\|\widetilde{\eta}_{\bcov}\|_{\l2, \bcov}^{2}} = C |\GRR(\qbold)|^{1/6},
\end{displaymath}
for a constant $C > 0$. By triviality of the Gromov--Witten invariants for K3 surfaces (see for example \cite[Corollary 3.3]{LeePark}), to prove Conjecture \ref{conj:GRR-iso-F1} we need to prove that $\tau_{\bcov}$ is constant. This is the content of \cite[Thm. 5.12]{cdg2}. 
\end{proof}

\section{A Chowla--Selberg formula for the BCOV invariant}
In this section we discuss an example of use of the arithmetic Riemann--Roch theorem to evaluate the BCOV invariant of a Calabi--Yau manifold with complex multiplication, similar to the derivation of the Chowla--Selberg formula from the Kronecker limit formula for elliptic curves. In such situations, or more generally for Calabi--Yau manifolds whose Hodge structures have some extra symmetries, we expect that the BCOV invariant can be evaluated in terms of special values of $\Gamma$ functions or other special functions.

Let $p\geq 5$ be a prime number, and define $n=p-1$. We consider the mirror family $f\colon\Zcal\to U$ to Calabi--Yau hypersurfaces of degree $p$ in $\PBbb^{n}$. The restriction on the dimension here has been made to simplify the exposition. The special fiber $Z_{0}$ is a crepant resolution of $X_{0}/G$, where $X_{0}$ is now the Fermat hypersurface
\begin{displaymath}
    x_{0}^{p}+\ldots+x_{n}^{p}=0.
\end{displaymath}
The quotient $X_{0}/G$ has an extra action of $\mu_{p}\subset\CBbb$: a $p$-th root of unity $\xi\in\CBbb$ sends a point $(x_{0}\colon\ldots\colon x_{n})$ to $(x_{0}\colon\ldots\colon x_{n-1}\colon \xi x_{n})$. This action induces a $\QBbb$-linear action of $K=\QBbb(\mu_{p})\subset\CBbb$ on $H^{n-1}(X_{0},\QBbb)^{G}$. As a rational Hodge structure, the latter is isomorphic to $H^{n-1}(Z_{0},\QBbb)$. For this, see \textsection \ref{subsec:hodge-bundles}, and especially Lemma \ref{prop:iso-prim-coh} and Proposition \ref{prop:orthogonal-decompositions} (we are in odd dimension, and all the cohomology is primitive now). Hence $H^{n-1}(Z_{0},\QBbb)$ inherits a $\QBbb$-linear action of $K$. Observe that $[K\colon\QBbb]=p-1$, which is exactly the dimension of $H^{n-1}(Z_{0},\QBbb)$. We say that $Z_{0}$ has complex multiplication by $K$. Similary, the algebraic de Rham cohomology $H^{n-1}(Z_{0},\Omega_{Z_{0}/\QBbb}^{\bullet})$ affords a $\QBbb$-linear action of $K$. Indeed, this is clear for $H^{n-1}(X_{0},\Omega_{X_{0}/\QBbb}^{\bullet})^{G}$, since the action of $\mu_{p}$ on $X_{0}$ by automorphisms can actually be defined over $\QBbb$ and commutes with the $G$ action. Then, we transfer this to the cohomology of $Z_{0}$ via Lemma \ref{prop:iso-prim-coh}, which in this case provides an isomorphism $H^{n-1}(Z_{0},\Omega_{Z_{0}/\QBbb}^{\bullet})\simeq H^{n-1}(X_{0},\Omega_{X_{0}/\QBbb}^{\bullet})^{G}$.

Let us fix a non-trivial $\xi\in\mu_{p}$. If we base change $H^{n-1}(Z_{0},\QBbb)$ to $K$, we have an eigenspace decomposition
\begin{displaymath}
    H^{n-1}(Z_{0},K)=\bigoplus_{k=0}^{p-1} H^{n-1}(Z_{0},K)_{\xi^{k}}.
\end{displaymath}
Hence, $\xi$ acts by multiplication by $\xi^{k}$ on $H^{n-1}(Z_{0},K)_{\xi^{k}}$. Similarly, for algebraic de Rham cohomology:
\begin{displaymath}
    H^{n-1}(Z_{0},\Omega^{\bullet}_{Z_{0}/K})=\bigoplus_{k=0}^{p-1} H^{n-1}(Z_{0},\Omega^{\bullet}_{Z_{0}/K})_{\xi^{k}}.
\end{displaymath}
If we compare with $H^{n-1}(X_{0},\Omega^{\bullet}_{X_{0}/K})^{G}$, and we recall the construction of the sections $\theta_{k}$ and $\eta_{k}^{\circ}$ (cf. \textsection \ref{subsec:Griffiths-sections}), we see by inspection that $\xi$ acts on $\eta_{k}^{\circ}$ by multiplication by $\xi^{k+1}$. Therefore, we infer that the non-trivial eigenspaces only occur when $1\leq k\leq p-1$ and
\begin{displaymath}
    H^{n-1}(Z_{0},\Omega^{\bullet}_{Z_{0}/K})_{\xi^{k}}= K\eta_{k-1}^{\circ}=H^{k-1}(Z_{0}, \Omega^{n-k}_{Z_{0}/K}).
\end{displaymath}
Hence, the eigenspace $H^{n-1}(Z_{0},\Omega^{\bullet}_{Z_{0}/K})_{\xi^{k}}$ has Hodge type $(n-k,k-1)$. 

The period isomorphism relating algebraic de Rham and Betti cohomologies decomposes into eigenspaces as well. We obtain refined period isomorphisms
\begin{displaymath}
    \per_{k}\colon H^{n-1}(Z_{0},\Omega^{\bullet}_{Z_{0}/K})_{\xi^{k}}\otimes_{K}\CBbb\overset{\sim}{\longrightarrow} H^{n-1}(Z_{0},K)_{\xi^{k}}\otimes_{K}\CBbb.
\end{displaymath}
Evaluating the isomorphism on $K$-bases of both sides, we obtain a period, still denoted $\mathrm{per}_{k}\in\CBbb^{\times}/K^{\times}$. 
\begin{lemma}\label{lemma:Gross}
Fix an algebraic closure $\ov{\QBbb}$ of $\QBbb$ in $\CBbb$. Then there is an equality in $\CBbb^{\times}/\ov{\QBbb}^{\times}$
\begin{displaymath}
    \per_{k}=\frac{1}{\pi}\Gamma\left(\frac{k+1}{p}\right)^{p}.
\end{displaymath}
\end{lemma}
\begin{proof}
The claim is equivalent to the analogous computation on $X_{0}$. Hidden behind this phrase is the comparison of cup products on $X_{0}$ and $Z_{0}$ accounted for by Lemma \ref{lemma:cup-prod}. On $X_{0}$, the formula for the period is well-known, and given for instance in Gross \cite[Sec. 4, p. 206]{Gross} (see more generally \cite[Chap. I, Sec. 7]{deligne:hodge-cycles}). Notice that the author would rather work with the Fermat hypersurface $x_{0}^{p}+\ldots+x_{n-1}^{p}=x_{n}^{p}$. 
However, as we compute periods up to algebraic numbers, by applying the obvious isomorphism of varieties defined over $\ov{\QBbb}$, the result is the same. Also, we have used standard properties of the $\Gamma$ function to transform \emph{loc. cit.} in our stated form.
\end{proof}

\begin{theorem} \label{thm:ChowlaSelberg}
For $Z_{0}$ of dimension $p-2$, with $p\geq 5$ prime, the BCOV invariant satisfies
\begin{displaymath}
    \tau_{\bcov}(Z_{0})=\frac{1}{\pi^{\sigma}}
    \left(\Gamma\left(\frac{1}{p}\right)^{\chi(Z_{0})/12}\ \prod_{k=1}^{p-1} \Gamma\left(\frac{k}{p}\right)^{p-k-1}\right)^{2p}\quad\text{in}\quad\RBbb^{\times}/\RBbb\cap \ov{\QBbb}^{\times},
\end{displaymath}
where
\begin{displaymath}
    \sigma=p\left(\frac{\chi(Z_{0})}{12}+\frac{(p-1)(p-2)}{2}\right)+\frac{1}{2}\sum_{k}(-1)^{k}k^{2}b_{k}.
\end{displaymath}
\end{theorem}
\begin{proof}
We apply Theorem \ref{thm:pre-formula}, written in terms of the sections $\eta_{k}^{\circ}$ instead of $\eta_{k}$ (which vanish at 0). Up to rational number, this has the effect of letting down the term $(\psi^{n+1})^{a}$ in that statement. We are thus lead to evaluate the $L^{2}$ norms of the sections $\eta_{k}^{\circ}$. By \cite[Lemma 3.4]{Maillot-Rossler}, the $L^{2}$ norms satisfy
\begin{displaymath}
    \|\eta_{k}^{\circ}\|_{\l2}^{2}=(2\pi)^{-(p-2)}|\per_{k}|^{2}.
\end{displaymath}
It is now enough to plug this expression in Theorem \ref{thm:pre-formula}, as well as the value of $\per_{k}$ provided by Lemma \ref{lemma:Gross}. 
\end{proof}
Combining Theorem \ref{thm:ARR} and the conjecture of Gross--Deligne (cf. \cite{Fresan, Maillot-Rossler} for up to date discussions and positive results), one can propose a general conjecture for the values of the BCOV invariants of some Calabi--Yau varieties with complex multiplication. For this to be plausible, it seems however necessary to impose further conditions on the Hodge structure. Other recent examples of Calabi--Yau manifolds whose BCOV invariants should adopt a special form are given in \cite{rank2}.


\begin{thebibliography}{CdlOEvS20}

\bibitem[Bat94]{Batyrev}
V.~V. Batyrev.
\newblock Dual polyhedra and mirror symmetry for {C}alabi-{Y}au hypersurfaces
  in toric varieties.
\newblock {\em J. Algebraic Geom.}, 3(3):493--535, 1994.

\bibitem[BB96]{Batyrev-Borisov}
V.~V. Batyrev and L.~A. Borisov.
\newblock Mirror duality and string-theoretic {H}odge numbers.
\newblock {\em Invent. Math.}, 126(1), 1996.

\bibitem[BCOV94]{BCOV}
M.~Bershadsky, S.~Cecotti, H.~Ooguri, and C.~Vafa.
\newblock Kodaira-{S}pencer theory of gravity and exact results for quantum
  string amplitudes.
\newblock {\em Comm. Math. Phys.}, 165(2):311--427, 1994.

\bibitem[BD96]{Batyrev-Dais}
V.~V. Batyrev and D.~I. Dais.
\newblock Strong {M}c{K}ay correspondence, string-theoretic {H}odge numbers and
  mirror symmetry.
\newblock {\em Topology}, 35(4):901--929, 1996.

\bibitem[Beh97]{GrWiAlGe}
K.~Behrend.
\newblock Gromov-{W}itten invariants in algebraic geometry.
\newblock {\em Invent. Math.}, 127(3):601--617, 1997.

\bibitem[BG14]{Dworkpencil}
G.~Bini and A.~Garbagnati.
\newblock Quotients of the {D}work pencil.
\newblock {\em J. Geom. Phys.}, 75:173--198, 2014.

\bibitem[BvS95]{batyrev-straten}
V.~V. Batyrev and D.~van Straten.
\newblock Generalized hypergeometric functions and rational curves on
  {C}alabi-{Y}au complete intersections in toric varieties.
\newblock {\em Comm. Math. Phys.}, 168(3):493--533, 1995.

\bibitem[CdlOEvS20]{rank2}
P.~Candelas, X.~de~la Ossa, M.~Elmi, and D.~van Straten.
\newblock A one parameter family of {C}alabi-{Y}au manifolds with attractor
  points of rank two.
\newblock {\em J. High Energy Phys.}, (10):202, 73, 2020.

\bibitem[CG11]{Corti-Golyshev}
A.~Corti and V.~Golyshev.
\newblock Hypergeometric equations and weighted projective spaces.
\newblock {\em Sci. China Math.}, 54(8):1577--1590, 2011.

\bibitem[CL12]{CostelloLi}
K.~{Costello} and S.~{Li}.
\newblock {Quantum BCOV theory on Calabi-Yau manifolds and the higher genus
  B-model}.
\newblock {\em arXiv e-prints}, page arXiv:1201.4501, January 2012.

\bibitem[Del]{Delignemirror}
P.~Deligne.
\newblock Local behavior of {H}odge structures at infinity.
\newblock In {\em Mirror symmetry, {II}}, volume~1 of {\em AMS/IP Stud. Adv.
  Math.}, pages 683--699.

\bibitem[Del71]{DeligneHodge2}
P.~Deligne.
\newblock Th\'eorie de {H}odge. {II}.
\newblock {\em Inst. Hautes \'Etudes Sci. Publ. Math.}, (40):5--57, 1971.

\bibitem[Del87]{Deligne-determinant}
P.~Deligne.
\newblock Le d\'{e}terminant de la cohomologie.
\newblock In {\em Current trends in arithmetical algebraic geometry ({A}rcata,
  {C}alif., 1985)}, volume~67 of {\em Contemp. Math.}, pages 93--177. Amer.
  Math. Soc., Providence, RI, 1987.

\bibitem[DHZ98]{DHZ}
D.~I. Dais, M.~Henk, and G.~M. Ziegler.
\newblock All abelian quotient {C}.{I}.-singularities admit projective crepant
  resolutions in all dimensions.
\newblock {\em Adv. Math.}, 139(2):194--239, 1998.

\bibitem[DHZ06]{DHZ-Gorenstein}
D.~I. Dais, M.~Henk, and G.~M. Ziegler.
\newblock On the existence of crepant resolutions of {G}orenstein abelian
  quotient singularities in dimensions {$\ge4$}.
\newblock In {\em Algebraic and geometric combinatorics}, volume 423 of {\em
  Contemp. Math.}, pages 125--193. Amer. Math. Soc., Providence, RI, 2006.

\bibitem[DMOS82]{deligne:hodge-cycles}
P.~Deligne, J.~S. Milne, A.~Ogus, and K.-Y. Shih.
\newblock {\em Hodge cycles, motives, and {S}himura varieties}, volume 900 of
  {\em Lecture Notes in Mathematics}.
\newblock Springer-Verlag, Berlin-New York, 1982.

\bibitem[Dol96]{Dolgachev}
I.~V. Dolgachev.
\newblock Mirror symmetry for lattice polarized {$K3$} surfaces.
\newblock volume~81, pages 2599--2630. 1996.
\newblock Algebraic geometry, 4.

\bibitem[DY20]{DaiYoshikawa}
X.~{Dai} and K.-I. {Yoshikawa}.
\newblock {Analytic torsion for log-{E}nriques surfaces and {B}orcherds
  product}.
\newblock {\em arXiv e-prints}, page arXiv:2009.10302, September 2020.

\bibitem[EFiMM18]{cdg}
D.~Eriksson, G.~Freixas~i Montplet, and C.~Mourougane.
\newblock Singularities of metrics on {H}odge bundles and their topological
  invariants.
\newblock {\em Algebraic Geometry}, 5:1--34, 2018.

\bibitem[EFiMM21]{cdg2}
D.~Eriksson, G.~Freixas~i Montplet, and C.~Mourougane.
\newblock B{COV} invariants of {C}alabi-{Y}au manifolds and degenerations of
  {H}odge structures.
\newblock {\em Duke Math. J.}, 170(3):379--454, 2021.

\bibitem[Eri08]{Dennis-these}
D.~Eriksson.
\newblock {\em Formule de {L}efschetz fonctorielle et applications
  g\'{e}om\'{e}triques}.
\newblock PhD thesis, Universit\'{e} Paris-Sud 11, 2008.

\bibitem[EZT14]{ElZein}
F.~El~Zein and L\^{e}~D{\~u}ng Tr\'{a}ng.
\newblock Mixed {H}odge structures.
\newblock In {\em Hodge theory}, volume~49 of {\em Math. Notes}, pages
  123--216. Princeton Univ. Press, Princeton, NJ, 2014.

\bibitem[FLY08]{FLY}
H.~Fang, Z.~Lu, and K.-I. Yoshikawa.
\newblock Analytic torsion for {C}alabi-{Y}au threefolds.
\newblock {\em J. Differential Geom.}, 80(2):175--259, 2008.

\bibitem[Fra92]{Franke}
J.~Franke.
\newblock Riemann--{R}och in functorial form.
\newblock Unpublished, 1992.

\bibitem[Fre17]{Fresan}
J.~Fres\'{a}n.
\newblock Periods of {H}odge structures and special values of the gamma
  function.
\newblock {\em Invent. Math.}, 208(1):247--282, 2017.

\bibitem[FZ20]{Yeping-3}
L.~{Fu} and Y.~{Zhang}.
\newblock {Motivic integration and the birational invariance of BCOV
  invariants}.
\newblock {\em arXiv e-prints}, page arXiv:2007.04835, July 2020.

\bibitem[{G}{\"a}h13]{Gahrs}
S.~{G}{\"a}hrs.
\newblock Picard-{F}uchs equations of special one-parameter families of
  invertible polynomials.
\newblock In {\em Arithmetic and geometry of {K}3 surfaces and {C}alabi-{Y}au
  threefolds}, volume~67 of {\em Fields Inst. Commun.}, pages 285--310.
  Springer, New York, 2013.

\bibitem[GKZ08]{GKZ}
I.~M. Gelfand, M.~M. Kapranov, and A.~V. Zelevinsky.
\newblock {\em Discriminants, resultants and multidimensional determinants}.
\newblock Modern Birkh\"auser Classics. Birkh\"auser Boston, Inc., Boston, MA,
  2008.

\bibitem[Gri69]{Griffiths-residues}
P.~A. Griffiths.
\newblock On the periods of certain rational integrals. {I}, {II}.
\newblock {\em Ann. of Math. (2) 90 (1969), 460-495; ibid. (2)}, 90:496--541,
  1969.

\bibitem[Gro78]{Gross}
B.~H. Gross.
\newblock On the periods of abelian integrals and a formula of {C}howla and
  {S}elberg.
\newblock {\em Invent. Math.}, 45(2):193--211, 1978.
\newblock With an appendix by David E. Rohrlich.

\bibitem[GS90a]{GS:AIT}
H.~Gillet and C.~Soul\'{e}.
\newblock Arithmetic intersection theory.
\newblock {\em Inst. Hautes \'{E}tudes Sci. Publ. Math.}, (72):93--174 (1991),
  1990.

\bibitem[GS90b]{GS:characteristic}
H.~Gillet and C.~Soul\'{e}.
\newblock Characteristic classes for algebraic vector bundles with {H}ermitian
  metric. {I}.
\newblock {\em Ann. of Math. (2)}, 131(1):163--203, 1990.

\bibitem[GS90c]{GS:characteristic-2}
H.~Gillet and C.~Soul\'{e}.
\newblock Characteristic classes for algebraic vector bundles with {H}ermitian
  metric. {II}.
\newblock {\em Ann. of Math. (2)}, 131(2):205--238, 1990.

\bibitem[GS92]{GS:ARR}
H.~Gillet and C.~Soul\'e.
\newblock An arithmetic {R}iemann-{R}och theorem.
\newblock {\em Invent. Math.}, 110(3):473--543, 1992.

\bibitem[Har66]{Hartshorne}
R.~Hartshorne.
\newblock {\em Residues and duality}.
\newblock Lecture notes of a seminar on the work of A. Grothendieck, given at
  Harvard 1963/64. With an appendix by P. Deligne. Lecture Notes in
  Mathematics, No. 20. Springer-Verlag, Berlin-New York, 1966.

\bibitem[HSBT10]{HSBT}
M.~Harris, N.~Shepherd-Barron, and R.~Taylor.
\newblock A family of {C}alabi-{Y}au varieties and potential automorphy.
\newblock {\em Ann. of Math. (2)}, 171(2):779--813, 2010.

\bibitem[Huy05]{Huy}
D.~Huybrechts.
\newblock {\em Complex geometry}.
\newblock Universitext. Springer-Verlag, Berlin, 2005.
\newblock An introduction.

\bibitem[Ill94]{illusie}
L.~Illusie.
\newblock Autour du th\'eor\`eme de monodromie locale.
\newblock {\em Ast\'erisque}, (223):9--57, 1994.
\newblock P{\'e}riodes $p$-adiques (Bures-sur-Yvette, 1988).

\bibitem[KMY18]{KMY}
S.~Kawaguchi, S.~Mukai, and K.-I. Yoshikawa.
\newblock Resultants and the {B}orcherds {$Phi$}-function.
\newblock {\em Amer. J. Math.}, 140(6):1471--1519, 2018.

\bibitem[KO68]{Katz-Oda}
N.~M. Katz and T.~Oda.
\newblock On the differentiation of de {R}ham cohomology classes with respect
  to parameters.
\newblock {\em J. Math. Kyoto Univ.}, 8:199--213, 1968.

\bibitem[KP08]{Klemm-Pandharipande}
A.~Klemm and R.~Pandharipande.
\newblock Enumerative geometry of {C}alabi-{Y}au 4-folds.
\newblock {\em Comm. Math. Phys.}, 281(3):621--653, 2008.

\bibitem[LP07]{LeePark}
J.~Lee and T.~Parker.
\newblock A structure theorem for the {G}romov-{W}itten invariants of
  {K}\"{a}hler surfaces.
\newblock {\em J. Differential Geom.}, 77(3):483--513, 2007.

\bibitem[Mor93]{morrison-mirror}
D.~R. Morrison.
\newblock Mirror symmetry and rational curves on quintic threefolds: a guide
  for mathematicians.
\newblock {\em J. Amer. Math. Soc.}, 6(1):223--247, 1993.

\bibitem[Mor97]{Morrison-adapted}
D.~R. Morrison.
\newblock Mathematical aspects of mirror symmetry.
\newblock In {\em Complex algebraic geometry ({P}ark {C}ity, {UT}, 1993)},
  volume~3 of {\em IAS/Park City Math. Ser.}, pages 265--327. Amer. Math. Soc.,
  Providence, RI, 1997.

\bibitem[MR04]{Maillot-Rossler}
V.~Maillot and D.~Roessler.
\newblock On the periods of motives with complex multiplication and a
  conjecture of {G}ross-{D}eligne.
\newblock {\em Ann. of Math. (2)}, 160(2):727--754, 2004.

\bibitem[MR12]{MaRo}
V.~Maillot and D.~R\"ossler.
\newblock On the birational invariance of the {BCOV} torsion of {C}alabi-{Y}au
  threefolds.
\newblock {\em Comm. Math. Phys.}, 311(2):301--316, 2012.

\bibitem[SC67]{Chowla-Selberg}
A.~Selberg and S.~Chowla.
\newblock On {E}pstein's zeta-function.
\newblock {\em J. Reine Angew. Math.}, 227:86--110, 1967.

\bibitem[Sch73]{schmid}
W.~Schmid.
\newblock Variation of {H}odge structure: the singularities of the period
  mapping.
\newblock {\em Invent. Math.}, 22:211--319, 1973.

\bibitem[Ste77]{Steenbrink-mixedonvanishing}
J.~Steenbrink.
\newblock Mixed {H}odge structure on the vanishing cohomology.
\newblock In {\em Real and complex singularities ({P}roc. {N}inth {N}ordic
  {S}ummer {S}chool/{NAVF} {S}ympos. {M}ath., {O}slo, 1976)}, pages 525--563.
  Sijthoff and Noordhoff, Alphen aan den Rijn, 1977.

\bibitem[Ste76]{Steenbrink-limits}
J.~Steenbrink.
\newblock Limits of {H}odge structures.
\newblock {\em Invent. Math.}, 31(3):229--257, 1975/76.

\bibitem[Voi99]{VoisinMirror}
C.~Voisin.
\newblock {\em Mirror symmetry}, volume~1 of {\em SMF/AMS Texts and
  Monographs}.
\newblock American Mathematical Society, Providence, RI; Soci\'{e}t\'{e}
  Math\'{e}matique de France, Paris, 1999.
\newblock Translated from the 1996 French original by Roger Cooke.

\bibitem[Voi07]{Voisin-II}
C.~Voisin.
\newblock {\em Hodge theory and complex algebraic geometry. {II}}, volume~77 of
  {\em Cambridge Studies in Advanced Mathematics}.
\newblock Cambridge University Press, Cambridge, english edition, 2007.

\bibitem[Yas04]{Yasuda}
T.~Yasuda.
\newblock Twisted jets, motivic measures and orbifold cohomology.
\newblock {\em Compos. Math.}, 140(2):396--422, 2004.

\bibitem[Yos99]{yoshikawa-discriminant}
K.-I. Yoshikawa.
\newblock Discriminant of theta divisors and {Q}uillen metrics.
\newblock {\em J. Differential Geom.}, 52(1):73--115, 1999.

\bibitem[Yos04]{Yoshik3surfinv}
K.-I. Yoshikawa.
\newblock {$K3$} surfaces with involution, equivariant analytic torsion, and
  automorphic forms on the moduli space.
\newblock {\em Invent. Math.}, 156(1):53--117, 2004.

\bibitem[Yos17]{Yoshikawa-orbifold}
K.-I. Yoshikawa.
\newblock Analytic torsion for {B}orcea--{V}oisin threefolds.
\newblock In {\em Geometry, Analysis and Probability}, volume 310 of {\em
  Progress in Math.}, pages 279--361. Springer International Publisher, 2017.

\bibitem[{Zha}20]{Yeping-2}
Y.~{Zhang}.
\newblock {BCOV invariant and blow-up}.
\newblock {\em arXiv e-prints}, page arXiv:2003.03805, March 2020.

\bibitem[Zin08]{Zingerstvsred}
A.~Zinger.
\newblock Standard versus reduced genus-one {G}romov-{W}itten invariants.
\newblock {\em Geom. Topol.}, 12(2):1203--1241, 2008.

\bibitem[Zin09]{Zingerreduced}
A.~Zinger.
\newblock The reduced genus 1 {G}romov-{W}itten invariants of {C}alabi-{Y}au
  hypersurfaces.
\newblock {\em J. Amer. Math. Soc.}, 22(3):691--737, 2009.

\bibitem[ZZ08]{ZaZi}
D.~Zagier and A.~Zinger.
\newblock Some properties of hypergeometric series associated with mirror
  symmetry.
\newblock In {\em Modular forms and string duality}, volume~54 of {\em Fields
  Inst. Commun.}, pages 163--177. Amer. Math. Soc., Providence, RI, 2008.

\end{thebibliography}
\end{document}